\documentclass[12pt]{amsart}%

\usepackage{epsfig}
\usepackage{amsmath}
\usepackage{amsfonts}
\usepackage{rotating}
\usepackage{graphicx}
\usepackage{pstricks, pst-node, pst-text, pst-3d}
\usepackage{amsbsy}
\usepackage[numbers]{natbib}
\usepackage{wrapfig}
\usepackage{amssymb}
\usepackage{latexsym}
\newtheorem{theorem}{Theorem}
\newtheorem{proposition}{Proposition}

\newtheorem{lemma}{Lemma}
\newtheorem{remark}{Remark}

\newtheorem{definition}{Definition}
\theoremstyle{remark}

\usepackage{mathptmx}      
\usepackage{latexsym}
\newcommand{\re}{\text{\rm Re }}
\newcommand{\im}{\text{\rm Im }}
\newcommand{\Sb}{\text{\bf S}}
\newcommand{\tb}{\text{\bf t}}

\newcommand{\Hol}{\text{\rm Hol }}
\newcommand{\emk}{\text{\rm cap }}
\newcommand{\D}{\mathbb{D}}

\newcommand{\Ha}{\mathbb{H}}

\newcommand{\de}{\partial}
\newcommand{\R}{\mathbb R}

\newcommand{\C}{\mathbb C}

\newcommand{\B}{\mathbb B}

\def\rg{{\sf rg}\,}

\DeclareMathOperator{\Gr}{Gr}

\DeclareMathOperator{\kernel}{ker}

\DeclareMathOperator{\cokernel}{coker}
\DeclareMathOperator{\virtcard}{virtcard}
\DeclareMathOperator{\virtdim}{virtdim}
\DeclareMathOperator{\ind}{ind}
\DeclareMathOperator{\id}{id}

\DeclareMathOperator{\spn}{span}

\begin{document}

\title{Classical and stochastic L\"owner-Kufarev equations
}
\thanks{All the authors have been supported by the ESF Networking Programme HCAA. The author$^{\dag}$ has been  supported by ERC grant `HEVO - Holomorphic Evolution Equations', no. 277691; The authors$^{\ddag}$ has been partially supported by the Ministerio de Econom\'{\i}a y Competitividad and the European Union (FEDER), projects
MTM2009-14694-C02-02 and MTM2012-37436-C02-01, and  by La Consejer\'{\i}a de Educaci\'{o}n y Ciencia de la Junta de Andaluc\'{\i}a; The author$^{\sharp}$ has been  supported by the grants of the Norwegian Research Council \#204726/V30 and  \#213440/BG.}
\author[Bracci, Contreras, D{\'\i}az-Madrigal, Vasil'ev]{Filippo Bracci$^{\dag}$, Manuel D. Contreras$^{\ddag}$, Santiago D{\'\i}az-Madrigal$^{\ddag}$, and Alexander Vasil'ev$^{\sharp}$}

\address{F.~Bracci:\\
              Dipartimento Di Matematica,
Universit\`{a} di Roma `Tor Vergata'\\
Via Della Ricerca Scientifica 1\\ 00133
Roma, Italy}
\email{fbracci@mat.uniroma2.it}           

\address{M.~D.~Contreras and  S.~D{\'\i}az-Madrigal:\\
              Departamento de Matem\'atica Aplicada II\\
Escuela T\'ecnica Superior de Ingenier{\'\i}a\\
Universidad de Sevilla\\
Camino de los Descubrimientos, s/n\\
41092 Sevilla, Spain}
 \email{contreras@us.es; madrigal@us.es}
 
\address{A.~Vasil'ev:
Department of Mathematics\\
University of Bergen\\
 P.O.~Box~7803\\
 Bergen N-5020, Norway}
 \email{alexander.vasiliev@math.uib.no}

\maketitle

\begin{abstract}
In this paper we present  a historical and scientific account of the development of the theory of the L\"owner-Kufarev classical and stochastic
equations spanning the 90-year period from the seminal paper by K.~L\"owner in 1923 to recent generalizations and stochastic versions and their relations to conformal field theory.

\end{abstract}

\tableofcontents

\section{Introduction}
\label{intro}

The L\"owner theory went through several periods of its development. It was born in 1923 in the seminal paper by L\"owner~\cite{Loewner}, its {\bf formation}
was completed in 1965 in the paper by Pommerenke~\cite{Pommerenke1} and was finally formulated thoroughly in his monograph~\cite{Pommerenke} unifying L\"owner's
ideas of semigroups and evolution equations and Kufarev's contribution~\cite{Kufarev, Kufarev47} of the $t$-parameter  differentiability (in the Carath\'eodory kernel) of a family of conformal maps  of simply connected domains
$\Omega(t)$ onto the canonical domain $\mathbb D$, the unit disk in particular, and of PDE for subordination L\"owner chains of general type. It is worthy  mentioning that the 20-year period between papers~\cite{Loewner} and~\cite{Kufarev} was illuminated by Goluzin's impact~\cite{Goluzin36} on further applications of L\"owner's parametric method to extremal problems for univalent functions, as well as works by Fekete and Szeg\"o~\cite{FeketeS}, Peschl~\cite{Peschl}, Robinson~\cite{Robinson}, Komatu~\cite{Komatu}, Bazilevich~\cite{Bazilevich36, Bazilevich} related to this period.

The next period became {\bf applications} of the {\it parametric method} to
solving concrete problems for conformal maps culminating at the de~Branges proof~\cite{Branges} of the celebrating Bieberbach conjecture~\cite{Bieberbach} in 1984. During approximately
16 years after this proof the interest to univalent functions had been somehow decaying. However, the {\bf modern period} has been marked by burst of interest to Stochastic (Schramm)-L\"owner Evolution (SLE) which has implied an elegant description of several 2D conformally invariant statistical physical systems at criticality by means of a solution to the Cauchy problem for the L\"owner equation with a random driving term given by 1D Brownian motion, see \cite{SLE, Schramm}. At the same time, several connections with mathematical physics were
discovered recently, in particular, relations with a singular representation of the Virasoro algebra using L\"owner martingales in \cite{BB, Friedrich, FriedrichWerner, MK} and with a Hamiltonian formulation and  a construction of the KP integrable hierarchies in \cite{MarkProkhVas, MarkinaVasiliev, MV2011}.

\section{L\"owner}

Karel L\"{o}wner was born on  May 29, 1893 in L\'any, Bohemia.
 He also used
the German spelling Karl of his first name until his immigration to the U.S.A. in 1939 after Nazis occupied Prague, when he changed it to Charles Loewner as a new start in a new country.
Although coming from a Jewish family and speaking Czech at home, all of his education was in German, following his father's wish. After finishing a German Gimnasium in Prague in 1912 he entered the Charles University of Prague and received his Ph.D. from this university in 1917 under the supervision of Georg Pick.

After four and a half years in German Technical University in Prague where the mathematical environment was not stimulating enough, L\"owner took up a position at the University of Berlin (now Humboldt Universit\"at Berlin) with great enthusiasm where he started to work surrounded by Schur, Brauer, Hopf, von~Neumann, Szeg\"o, and other  famous names. Following  a brief lectureship  at Cologne in 1928 L\"owner returned back to the Charles University of Prague where he hold a chair in Mathematics until the occupation of Czechoslovakia in 1939. When the Nazis
occupied Prague, he was put in jail. Luckily, after paying
the `emigration tax'  twice over he was allowed to leave the country with
his family and moved, in 1939, to the U.S.A. After von Neumann arranged a position for him at the Louisville University, L\"owner was able to start from the bottom his career. Following further appointments at the Brown (1944) and Syracuse (1946) Universities, he moved in 1951 to his favorite place, California, where he worked at the Stanford University until his death. His former Ph.D. student in Prague, Lipman Bers, testifies~\cite{OConnor} that L\"owner was a man whom everybody liked, a man at peace with himself, a man who was incapable of malice. He was  a great teacher, Lipman Bers, Adriano Garcia, Gerald Goodman, Carl Fitzgerald,  Roger Horn, Ernst Lammel, Charles Titus were his students among 26 in total. L\"owner died as a Professor Emeritus in Stanford on January 8, 1968, after a brief illness. For more about L\"owner's biography, see~\cite{OConnor}.

L\"owner's work covers wide areas of complex analysis and
differential geometry, and displays his deep
understanding of Lie theory and his passion for semigroup theory.
Perhaps it is worth to mention here two of his discoveries. First was his introduction of infinitesimal methods for univalent mappings~\cite{Loewner}, where he
defined basic notions of semigroups of conformal maps of $\mathbb D$ and their evolution families leading to the sharp estimate $|a_3|\leq 3$ in the Bieberbach conjecture
$|a_n|\leq n$  in the class $\Sb$ of normalized univalent  maps $f(z)=z+a_2z^2+\dots$ of the unit disk $\mathbb D$.

 It is well-known that the
Bieberbach conjecture was finally proven by de Branges in 1984 \cite{Branges1, Branges} . In his proof, de Branges
introduced ideas related to L\"owner's but he used only a distant relative of L\"owner's
equation in connection with his main apparatus, his own rather sophisticated theory
of composition operators.
However, elaborating on de Branges' ideas, FitzGerald and Pommerenke
\cite{F-P} discovered how to avoid altogether the composition operators and rewrote the proof in classical
terms, applying the \emph{bona fide} L\"owner equation and providing in this way a direct and powerful application of L\"owner's classical method.

The second paper~\cite{Loewner2}, we want to emphasize on, is dedicated to properties of $n$-monotonic matrix functions,
which turned to be of importance for electrical engineering and for quantum physics. L\"owner was also interested in the problems of fluid mechanics, see e.g.,~\cite{LoewnerFM},
while working at Brown University 1944--1946 on a war-related program.
Starting with some unusual applications of the theory of univalent functions to the flow of incompressible fluids, he later applied his methods to the difficult
problems of the compressible case. This led him naturally to the study of partial differential equations in which he obtained significant differential inequalities and theorems regarding general conservation laws.

\section{Kufarev}

The counterpart subordination chains and the L\"owner-Kufarev PDE for them were introduced by Pavel Parfen'evich Kufarev (Tomsk,  March 18,
1909 -- Tomsk, July 17, 1968).
Kufarev entered the Tomsk State University in 1927, first, the Department of Chemistry, and then, the Department of Mathematics, which he
successfully finished in 1931.  After a year experience in industry in Leningrad he returned back to Tomsk in 1932 and all his academic life
was connected with this university, the first university in Siberia. Kufarev started to work at the Department of Mathematics led by professor
Lev Aleksandrovich Vishnevski{\u\i} (1887--1937). At the same time, two prominent Western mathematicians came to Tomsk.
One was Stefan Bergman (1895--1977).
Being of Jewish origin he was forced from his position in Berlin in 1933 due to the `Restoration of the civil service', an anti-Semitic Hitler's law. Bergman came to Tomsk
in 1934 and was working there until 1936, then at Tbilisi in Georgia  in 1937 and had to leave Soviet Union under Stalin's oppression towards foreign scientists.
For two years
he was working at the Institute Henri Poincar\'e in Paris, and then, left France for the United States in 1939 because of German invasion, he finally worked at Stanford from 1952 together with L\"owner.

The history of the other one is more tragic. Fritz Noether (1884--1941), the youngest   son of Max Noether (1844--1921) and the younger brother of Emmy Noether (1882--1935), came to
Tomsk in the same year as Bergman, because of the same reason, and remained there until 1937, when Vishnevski{\u\i} and `his group' were accused of espionage.
Noether and Vishnevski{\u\i} were arrested by NKVD (later KGB) and Noether was transported to Orel concentration camp where he was jailed until 1941. Nazis approached Orel in 1941 and many prisoners were executed on September 10, 1941 (Stalin's order from September 8,1941). Fritz Noether was among them following the official Soviet version of 1988, cf.~\cite{Dick}.

Fitz's wife Regine returned back to Germany from Soviet Union in 1935 under psychological depression and soon died. The same year his famous sister Emmy died in Pennsylvania after a feeble operation.   His two sons were deported from Soviet Union. Fortunately, they were given refuge in Sweden as a first step.
However, many years later Evgeniy Berkovich and Boris M. Schein published a correspondence~\cite{Berkovich}, also~\cite{Shein}, in which prof. Boris Shein referred to prof. Saveli{\u\i} Vladimirovich Falkovich (1911--1982,  Saratov), who met Fritz Noether in Moscow metro in the late fall of 1941. Noether and Falkovich knew each other and Noether talked his story on the arrest and tortures in Tomsk NKVD. In particular, NKVD agents confiscated many of his things and books. Noether  said that he finally was released from Orel Central and went to Lubyanka (NKVD/KGB headquarters)
for some traveling documents to visit his family. Then the train stopped and their conversation was interrupted. Since he did not come to Tomsk (his son's Gottfried Noether testimony), he most probably was arrested again. This ruins the official Soviet version of Noether's death and indicates that the story is not completed yet\footnote{I have also some personal interest to this story because me, Falkovich and Shein worked in different  periods at the same Saratov State University. {\it A.~Vasil'ev}}. One of the possible reasons is that several Jewish prisoners (first of all, originating from Poland,~e.g., Henrik Erlich and Victor Alter; and from Germany)
were released following Stalin's plans of creating an Anti-Hitler Jewish Committee, which further was not realized. Noether could be among them being rather known mathematician
himself and having support from Einstein and Weyl.

Kufarev was also accused of espionage being a member of Vishnevski{\u\i}'s `terrorist group' together with Boris~A.~Fuchs, but they were not arrested in contrast to Fritz Noether, see~\cite{indictment}.

Bergman and Noether supported Kufarev's thesis defense in 1936. Then Kufarev was awarded the Doctor of Sciences degree (analogue of German habilitation) in 1943 and remained the unique mathematician in Siberia with this degree until 1957 when the {\it Akademgorodok} was founded  and Mikhail A. Lavrentiev invited many first-class mathematicians to Novosibirsk. Kufarev became a full professor of the Tomsk State University in 1944, served as a dean of the Faculty of Mechanics and Mathematics 1952--1955, and remained a professor of the university until his death in 1968 after a hard illness.

The main interests of Kufarev were in the theory of univalent functions and applications to fluid mechanics, in particular, in Hele-Shaw flows, see an overview in~\cite{Vas09}. His main results on the parametric method were published in two papers~\cite{Kufarev, Kufarev47}, where he considered subordination chains of general type and wrote corresponding PDE for mapping functions, but he returned back to this method all the time together with his students, combining variational and parametric methods,
creating the parametric method for multiply connected domains, half-plane version of the parametric method, etc.

\section{Pommerenke and unification of the parametric method}

Christian Pommerenke (born December 17, 1933 in Copenhagen, Professor Emeritus at the Technische Universit\"at Berlin) became that person who unified L\"owner and Kufarev's ideas and
thoroughly combined analytic properties of the ordering of the images
of univalent mappings of the unit disk with evolutionary aspects of semigroups of conformal maps. He seems to have been the
first one to use the expression `L\"owner chain' for describing
the family of `increasing' univalent mappings in L\"owner's
theory.

Recall that we denote the unit disk by $\mathbb D=\{\zeta:\,|\zeta|<1\}$ and the class of normalized univalent maps $f\colon \mathbb D \to \mathbb C$, $f(z)=z+a_2z^2+\dots$ by $\Sb$.
A $t$-parameter family $\Omega(t)$ of simply connected hyperbolic univalent domains forms a {\it L\"owner subordination chain}  in the complex plane $\mathbb
C$, \ for \ $0\leq t< \tau$ (where $\tau$ may be $\infty$), if
$\Omega(t)\subseteq \Omega(s)$, whenever $t<s$, and the family is continuous in the Carath\'eodory sense.
We
suppose that the origin is a point of
$\Omega(0)$.

A L\"owner subordination chain $\Omega(t)$ is described by a $t$-dependent family of conformal maps $z=f(\zeta,t)$
from $\mathbb D$ onto $\Omega(t)$, normalized by $f(\zeta,t)=a_1(t)\zeta+a_2(t)\zeta^2+\dots$,
$a_1(t)>0$, $\dot{a}_1(t)>0$.  Pommerenke \cite{Pommerenke1, Pommerenke}
described governing evolution equations in partial and ordinary derivatives, known now as
the L\"owner-Kufarev equations.

One can normalize the growth of
evolution of a subordination chain by the conformal radius of
$\Omega(t)$ with respect to the origin setting  $a_1(t)=e^t$.

We say that the  function $p$ is from the Carath\'eodory class if it is analytic in $\mathbb D$, normalized as $p(\zeta)=1+p_1\zeta+p_2\zeta^2+\dots,\quad
\zeta\in \mathbb D,$ and such that $\re p(\zeta)>0$ in~$\mathbb D$.
Given a L\"owner subordination
chain of domains $\Omega(t)$ defined for $t\in [0,\tau)$, there exists
a function $p(\zeta,t)$, measurable in $t\in [0,\tau)$ for any fixed $z\in \mathbb D$, and from the Carath\'eodory class for almost all   $t\in [0,\tau)$, such that
the conformal mapping $f\colon\mathbb D\to \Omega(t)$ solves the equation
\begin{equation}
\frac{\partial f(\zeta,t)}{\partial t}=\zeta\frac{\partial
f(\zeta,t)}{\partial \zeta}p(\zeta,t),\label{LK}
\end{equation}
for $\zeta\in \mathbb D$ and for almost all $t\in [0,\tau)$.   The
equation (\ref{LK}) is called the L\"owner-Kufarev equation due to
two seminal papers: by L\"owner \cite{Loewner} who considered the case when
\begin{equation}
p(\zeta,t)=\frac{e^{iu(t)}+\zeta}{e^{iu(t)}-\zeta},\label{yadro}
\end{equation}
where $u(t)$ is a continuous function regarding to $t\in [0,\tau)$,
 and by Kufarev \cite{Kufarev}
who proved differentiability of $f$ in $t$ for all $\zeta$  in the case of general $p$ from the Carath\'eodory class.

 Let us consider a reverse process. We are given an
initial domain $\Omega(0)\equiv \Omega_0$ (and therefore, the
initial mapping $f(\zeta,0)\equiv f_0(\zeta)$), and an analytic function
$p(\zeta,t)$ of positive real part normalized by $p(\zeta,t)=1+p_1\zeta+\dots$. Let us solve the equation (\ref{LK})
and ask ourselves, whether the solution $f(\zeta,t)$ defines a subordination
chain of simply connected univalent domains $f(\mathbb D,t)$. The initial condition
$f(\zeta,0)=f_0(\zeta)$ is not given on the characteristics of the
partial differential equation (\ref{LK}), hence the solution exists
and is unique but not necessarily univalent. Assuming $s$ as a parameter along the characteristics
we have $$ \frac{dt}{ds}=1,\quad \frac{d\zeta}{ds}=-\zeta
p(\zeta,t), \quad \frac{df}{ds}=0,$$ with the initial conditions
$t(0)=0$, $\zeta(0)=z$, $f(\zeta,0)=f_0(\zeta)$, where $z$ is in
$\mathbb D$.  Obviously, we can assume $t=s$. Observe that the domain of $\zeta$ is the entire unit disk. However, the solutions to
the second equation of the characteristic system range within the unit disk but do not fill it.
Therefore, introducing another letter $w$ (in order to distinguish the function $w(z,t)$  from the variable $\zeta$) we arrive at the Cauchy problem for the  L\"owner-Kufarev
equation in ordinary derivatives
\begin{equation}
\frac{dw}{dt}=-wp(w,t),\label{LKord}
\end{equation}
 for a function $\zeta=w(z,t)$
with the initial condition $w(z,0)=z$. The equation (\ref{LKord}) is a non-trivial  characteristic
equation for (\ref{LK}). Unfortunately, this approach requires the
extension of $f_0(w^{-1}(\zeta,t))$ into the whole $\mathbb D$ (here $w^{-1}$ means the inverse function in $\zeta$) because the solution to
(\ref{LK}) is the function $f(\zeta,t)$  given as
$f_0(w^{-1}(\zeta,t))$, where $\zeta=w(z,s)$ is a solution of the
initial value problem for the characteristic equation (\ref{LKord})
that maps $\mathbb D$ into $\mathbb D$. Therefore, the solution of the initial
value problem for the equation (\ref{LK}) may be non-univalent.

On the other hand,  solutions to the
equation (\ref{LKord}) are holomorphic univalent functions
$w(z,t)=e^{-t}(z+a_2(t)z^2+\dots)$ in the unit disk that map $\mathbb D$  into
itself. Every function $f$ from the class $\Sb$ can be
represented by the limit
\begin{equation}
f(z)=\lim\limits_{t\to\infty}e^t w(z,t),\label{limit}
\end{equation}
where $w(z,t)$ is a solution to \eqref{LKord}  with some  function $p(z,t)$ of positive real part for almost
all $t\geq 0$ (see \cite[pages 159--163]{Pommerenke}). Each function
$p(z,t)$ generates a unique function from the class $\Sb$. The
reciprocal statement is not true. In general, a function $f\in \Sb$
can be obtained using different functions $p(\cdot,t)$.

Now we are ready to formulate the condition of  univalence of the solution to the equation (\ref{LK}) in terms of the limiting function \eqref{limit}, which
can be obtained by combination of known results of \cite{Pommerenke}.

\begin{theorem}\label{ThProkhVas}{\rm \cite{Pommerenke, ProkhVas}} Given  a function
$p(\zeta,t)$ of positive real part normalized by $p(\zeta,t)=1+p_1\zeta+\dots$, the solution to   the equation (\ref{LK})
is unique, analytic and univalent with respect to $\zeta$ for almost all $t\geq 0$, if and only if, the initial condition
$f_0(\zeta)$ is taken in the form \eqref{limit}, where the function $w(\zeta,t)$ is the solution to the equation \eqref{LKord}
with the same driving function $p$.
\end{theorem}

Concluding this section we remark that the L\"owner and L\"owner-Kufarev equations are described in several monographs~\cite{Ahlfors, Aleksandrov, Conway, Duren, Goluzin, Graham, Hayman, Pommerenke, Rosenblum}.

\section{Half-plane version}

In 1946, Kufarev \cite[Introduction]{Kuf1946} first mentioned   an evolution equation in the upper
half-plane $\mathbb H$ analogous to the one introduced by L\"owner in the unit
disk, and was first studied by Popova~\cite{Popova54} in 1954. In 1968, Kufarev, \ Sobolev and Sporysheva \cite{KSS}
introduced a combination  of Goluzin-Shiffer's variational and  parametric methods for this equation for the class of univalent
functions in the upper half-plane, which
is known to be related to  physical problems in
hydrodynamics. They showed its application to the extremal problem of finding the range of $\{\re e^{i\alpha} f(z),\, \im f(z)\}$, $\im z >0$. Moreover, during the second half of the past
century, the Soviet school intensively studied Kufarev's
equations for $\mathbb H$. We ought to cite here at least contributions by
Aleksandrov \cite{Aleksandrov}, Aleksandrov and  Sobolev \cite{AleksSTSob},  Goryainov and  Ba \cite{Goryainov, Goryainov-Ba}. However, this work was mostly unknown
to many Western mathematicians, in particular, because some of it appeared in journals not easily accessible from the outside of the Soviet Union. In fact, some of Kufarev's papers were not even reviewed by Mathematical Reviews. Anyhow, we refer the reader to \cite{obituary}, which contains a complete bibliography of his papers.

In order to introduce Kufarev's equation properly, let us fix some notation. Let $\gamma$ be a Jordan arc in
the upper half-plane $\mathbb H$ with starting point $\gamma(0)=0$. Then there exists a unique conformal map
$g_t:\Ha\setminus \gamma[0,t]\to\Ha$ with the normalisation
$$
g_t(z)=z+\frac{c(t)}{z}+O\left(\frac{1}{z^2}\right),\quad z\sim\infty.
$$
After a reparametrisation of the curve $\gamma$, one can assume that $c(t)=2t$.
Under this normalisation, one can show that $g_t$ satisfies the
following differential equation:
\begin{equation}\label{hydro-Low}
\frac{d g_t(z)}{d t}=\frac{2}{g_t(z)-\xi(t)}, \qquad g_0(z)=z.
\end{equation}
The equation is valid up to a time $T_z\in (0,+\infty]$ that
can be characterised as the first time $t$ such that $g_t(z)\in
\R$ and where $h$ is a continuous real-valued function.
Conversely, given a continuous function
$h\colon[0,+\infty)\to\R$, one can consider the following
initial value problem for each $z\in\Ha$:
\begin{equation}\label{hydro2}
\frac{d w}{dt}=\frac{2}{w-\xi(t)}, \qquad w(0)=z\;.
\end{equation}
Let $t\mapsto w^z(t)$ denote the unique solution to this Cauchy problem and let $g_t(z):=w^z(t)$. Then $g_t$
maps holomorphically a (not necessarily slit) subdomain of the upper half-plane $\Ha$ onto $\Ha$. Equation \eqref{hydro2} is nowadays known  as the {\it chordal L\"owner differential equation} with
the function $h$ as the driving term. The name is due to the fact that the curve $\gamma [0,t]$ evolves in time as
$t$ tends to infinity into a sort of chord joining two boundary points. This kind of construction can be used
to model evolutionary aspects of decreasing families of domains in the complex plane. The equation \eqref{LKord} with the function $p$ given by \eqref{yadro} in this context is called
the {\it radial L\"owner equation}, because in the slit case, the tip of slit tends to the origin in the unit disk.

Quite often it is presented the half-plane version considering the inverse
of the functions $g_t$ (see, i.e. \cite{KSS}). Namely, the conformal
mappings $f_t=g_t^{-1}$ from~$\mathbb H$ onto
$\mathbb H\setminus\gamma\big([0,t)\big)$  satisfy the  PDE
\begin{equation}\label{EQ_chordal_PDE}
\frac{\partial f_t(z)}{\partial t}=-f_t'(z)\frac{2}{z-\xi(t)}.
\end{equation}

We remark that using the Cayley transform $T(z)=\frac{z+1}{1-z}$ we
obtain  that the
chordal L\"owner equation in the unit disk takes the form
\begin{equation}\label{chordal}
\frac{\partial h_t(z)}{\partial t}=-h_t'(z)(1-z)^2p(z,t),
\end{equation}
where $\re p(z, t)\geq 0$ for all $t\geq 0$ and $z\in \D$.
From a geometric point of view, the difference between this family of
parametric functions and those described by
the L\"owner-Kufarev equation (\ref{LK}) is clear: the ranges of the
solutions of \eqref{chordal} decrease with $t$ while in the former
equation \eqref{LK} increase. This duality `decreasing' versus
`increasing' has been recently analyze in \cite{SMP-duality}  and,
roughly speaking, we can say that the `decreasing' setting can be deduced
from the  `increasing' one.

The stochastic version of \eqref{hydro2} with a random entry $\lambda$ will be discussed in Section~\ref{SLE}.

Analogues of the L\"owner-Kufarev methods appeared also in the theory of planar quasiconformal maps but only few concrete problems were solved using them, see~\cite{Shah} and
\cite{GR}.

\section{Applications to extremal problems. Optimal control}

After L\"owner himself, one of the first who applied L\"owner's
method to extremal problems in the theory of univalent
functions in 1936 was Gennadi{\u\i} Mikhailovich Goluzin (1906--1952, St.~Petersburg--Leningrad, Russia)~\cite{Goluzin36, Goluzin36a} and Ernst Peschl (1906, Passau--1986, Eitorf, Germany)~\cite{Peschl}.
Goluzin was a founder of Leningrad school in geometric function theory, a student of Vladimir I.~Smirnov.
He obtained in an elegant way
several new and sharp estimates. The most important of them
is the sharp estimate in the  rotation theorem. Namely, if $f\in \Sb$, then
\begin{equation*}
|\arg f'(z)|\leq
\begin{cases} 4\arcsin |z| & \text{if $0<|z|\leq\frac{1}{\sqrt{2}}$,}
\\
\pi+\log\frac{|z|^2}{1-|z|^2} &\text{if $\frac{1}{\sqrt{2}}<|z|<1$.}
\end{cases}
\end{equation*}

Goluzin himself proved~\cite{Goluzin36a} sharpness only for the case $0<|z|\leq\frac{1}{\sqrt{2}}$, Bazilevich~\cite{Bazilevich36} completed the proof for $\frac{1}{\sqrt{2}}<|z|<1$ later during the same year.

Ernst Peschl was a student of Constantin Carath\'eodory and  obtained his doctorate at Munich in 1931. He spent two years 1931--1933  it Jena working with Robert K\"onig, then moved to
M\"unster, Bonn, Braunschweig and finally return to the Rheinische Friedrich-Wilhelms University in Bonn where he worked until his retirement in 1974. In contrast with other heroes
of our story Peschl remained in Germany under Nazi and even was a member of the Union of National Socialist Teachers 1936--1938 (he was thrown out of the Union of National Socialist Teachers since he had not paid his membership fees), however, it was only in order to continue his academic career. Being fluent in French, Peschl served as a military interpreter
during the II-nd World War until 1943. Peschl~\cite{Peschl} applied L\"owner's method to prove the following statement. If $f\in\Sb$ and $f(z)=z+a_2z^2+\dots$, then $2\Phi\left(\re \frac{a_2}{2}\right)-1\leq \re(a_3-a_2^2)\leq 1$, where
\[
\Phi(x)=\frac{x^2}{\varphi^2(x)}(2\varphi(x)-1),
\]
where $\varphi(x)$ is a unique solution to the equation  $x+\varphi e^{1-\varphi}=0$. The result is sharp. The paper is much deeper than only this inequality and became the first serious treatment of the problem of description of the coefficient body $(a_2,a_3,\dots, a_n)$ in the class $S$, which was later treated in a nice monograph by Schaeffer and Spencer~\cite{SS}.   Let us mention that the above inequality was repeated by Goryainov~\cite{Goryainov80} in 1980 who also presented all possible extremal functions.

Related result is a nice matter of the paper by Fekete and Szeg\"o~\cite{FeketeS}, see also~\cite[page 104]{Duren} who ingeniously applied L\"owner's method in order to disprove
a conjecture $|c_{2n+1}|\leq 1$ by Littlewood and Paley~\cite{LP} on coefficients of odd univalent functions $\Sb^{(2)}\subset \Sb$ defined by the expansion $h(z)=z+c_3z^3+c_5z^5+\dots$.
\begin{theorem}{\rm \cite{FeketeS}} Suppose that $f\in \Sb$ and $0<\alpha<1$. Then
\[
|a_3-\alpha a_2^2|\leq 1+2^{-2\alpha/(1-\alpha)}.
\]
This bound is sharp for all $0<\alpha<1$.
\end{theorem}

The choice $\alpha=\frac{1}{4}$ and a simple recalculation of coefficients of functions from $\Sb^{(2)}$ versus $\Sb$
\[
c_3=\frac{1}{2}a_2,\quad c_5=\frac{1}{2}(a_3-\frac{1}{4}a_2),
\]
lead to the corollary $|c_5|\leq \frac{1}{2}+e^{-2/3}=1.013\dots$. Since the result is sharp, there exist odd univalent functions with coefficients bigger than 1.
Using similar method Robertson~\cite{Robertson} proved that $|c_3|^2+|c_5|^2\leq 2$, which is a particular case of his conjecture $\sum_{k=1}^{n}|c_{2k-1}|^2\leq n$ for the class $\Sb^{(2)}$ with $c_1=1$.

Among other papers on coefficient estimates, let us distinguish the first proof of the Bieberbach conjecture for $n=4$ by Garabedian and Schiffer~\cite{GarSchiffer} in 1955 where
L\"owner's method was also used. Walter Hayman wrote in his review on this paper that the method `not only carries Bieberbach's conjecture one step further from the point where L\"owner placed it over 30 years ago, but also gives a hope, in principle at least, to prove the conjecture for the next one or two coefficients by an increase in labour, rather than an essentially new method.' Finally, the complete proof of the Bieberbach conjecture by de~Branges~\cite{Branges1, Branges} in 1984 used the parametric method. It is worthy mentioning that the coefficient
problem for the inverse functions is much simpler and was solved by L\"owner in the same 1923 paper.
\begin{theorem}{\rm \cite{Loewner}} Suppose that $f\in \Sb$ and that
\[
z=\phi(w)=f^{-1}(w)=w+b_2w^2+\dots
\]
is the inverse function. Then
\[
|b_n|\leq \frac{1\cdot 3\cdot 5\cdots (2n-1)}{(n+1)!}2^n,
\]
with the equality for the function $f(z)=\frac{z}{(1+z)^2}$.
\end{theorem}

Let us just mention that most of elementary estimates of functionals in the class $\Sb$, such as $|f(z)|$, $|\arg \frac{f(z)}{z}|$, $|f'(z)|$, can be obtained by L\"owner's method, see e.g.,
\cite{Aleksandrov, Duren, Hayman, Pommerenke}. In particular, the sharp estimate
\[
\bigg|\arg\frac{zf'(z)}{f(z)}\bigg|\leq\log\frac{1+|z|}{1-|z|}
\]
implies that for every $0<r\leq$tanh$(\pi/4)$ the disk $|z|<r$ is mapped by $f\in \Sb$ onto a starlike domain with respect to the origin. The constant tanh$(\pi/4)$ is sharp and is called
the radius of starlikeness. Krzy\.z~\cite{Krzyz} used the L\"owner method to obtain the radius of close-to-convexity 0.80\dots \ in the class $\Sb$.

In 1950-1952, Kufarev and his student Fales \cite{Kufarev50, Kufarev51, Kufarev52} solved a Lavrentiev problem on the weighted product of the conformal radii of two non-overlapping domains in $\mathbb D$ using his parametric method. Namely, let $z_1$ and $z_2$ be two points in $\mathbb D$ and let $\Gamma$ be a Jordan curve splitting $\mathbb D$ into two
domains $B_1$ and $B_2$ such that $z_1\in B_1$ and $z_2\in B_2$. Let $f_1$ and $f_2$ be two conformal maps of $\mathbb D$ onto $B_1$ and $B_2$ respectively, $f_k(0)=z_k$, $f'_k(0)>0$, $k=1,2$.
The problem is to find $\Gamma$ which maximizes the functional $J(\Gamma;z_1,z_2)=|f'_1(0)|^{\alpha}|f'_2(0)|^{\beta}$, where $\alpha,\beta\geq 0$. For $\alpha=\beta$ the maximum is attained if $\Gamma$ is the non-Euclidean line in $\mathbb D$ which bisects orthogonally the non-Euclidean segment which connects the two points $z_1$ and $z_2$. This was obtaind by Lavrentiev~\cite{Lavr} himself by different approach. For $\alpha\neq\beta$ the answer is much more complicated.

After Kufarev, the Tomsk school in geometric function theory was lead by Igor Aleksandrovich Aleksandrov (born May 11, 1932, Novosibirsk, Russia) who completed his Ph.D. in the Tomsk State University in 1958 under supervision by Kufarev. Together with his students, he developed the parametric method combining it with the Goluzin-Schiffer variational method in 60' and solved several important extremal problems. In 1963 he defended the Doctor of Sciences degree.
This time specialists in univalent functions turned from estimation of functionals on the class $\Sb$ to evaluation of ranges of systems of functionals. One of the natural system is
$I(f)=\{\log \frac{f(z)}{z}, \log \frac{zf'(z)}{f(z)}\}$ for a chosen branch of $\log$, such that $I(id)=(0,0,0,0)$. When $f$ runs over $\Sb$, the range $\mathfrak{B}=\{I(f)\colon f\in \Sb\}$ of $I(f)$ fills a closed bounded set in $\mathbb R^4$. The boundary
of this set was completely described by Popov~\cite{Popov65}, Gutlyanski{\u\i}~\cite{Gut70},  Goryainov and Gutlyanski{\u\i}~\cite{GorGut} (for the sub-class $\Sb_M$ of bounded functions $|f(z)|\leq M$ from $\Sb$).
Goryainov~\cite{Gor83, Gor82, Gor83a} moreover,
obtained
the uniqueness
 of the extremal functions for the boundary of $\mathfrak{B}$ and their form. Goryainov, based on L\"owner-Kufarev equations and on results by Pommerenke~\cite{Pommerenke} and Gutlyanski{\u\i}~\cite{Gut70} on representation of functions from $\Sb$ by \eqref{limit}, created a method
of determination of all boundary functions for this and other systems of functionals, see also~\cite{Goryainov80}. It was time when a part of mathematicians lead by prof. Georgi{\u\i} Dmitrievich Suvorov moved from Tomsk to Donetsk (Ukraine) to a newly established (1966) Institute of Applied Mathematics of the Academy of Sciences of the Ukraine. Vladimir Ya. Gutlyanski{\u\i} was a student of Aleksandrov, Georgi{\u\i} D. Suvorov of Kufarev, and Victor V. Goryainov of Gutlyanski{\u\i}. There were several related works on studying $\mathfrak{B}$,
its projections, generalizations to
other similar results for the classes of univalent functions by Chernetski{\u\i}, Astakhov and other members of this active group. Let us mention  earlier works by Aleksandrov and Kopanev ~\cite{AleksandrovKopanev} who determined the domain of variability $\log f'(z)$ for functions in $\Sb$ for each fixed $z$ by L\"owner-Kufarev method. Earlier Arthur Grad~\cite[pages 263--291]{SS} obtained this domain by variational method. The same authors~\cite{AleksandrovKopanev66} found the range of the system of functionals $\{\log|\frac{f(z)}{z}|,\log|f'(z)|\}$ in the class $\Sb$.

That time several powerful methods for solution of extremal problems competed. Most known are the variational method of Schiffer and Goluzin~\cite{SS, SS54}, area principle~\cite{Lebedev}, method of extremal metrics~\cite{Kuzmina}. The problem is hard first of all because the univalence condition for the class $\Sb$ makes it a non-linear manifold and
all variational methods must be very special. In this situation, a considerable progress was achieved when a general optimization principle~\cite{Pontryagin} appeared in 1964, now known as the Pontryagin Maximum Principle (PMP). It turned out that the L\"owner and L\"owner-Kufarev equations being viewed as evolution equations give controllable systems of differential equations where the driving term or control function is provided by the Carath\'eodory function $p$. Perhaps it was L\"owner himself who first noticed all advantages of combination
of his parametric method with PMP. His last student Gerald S.~Goodman~\cite{Goodman66, Goodman} first explicitly stated this combination in 1966. However, he did not return back to this topic, and later (and independently) Aleksandrov and Popov~\cite{AleksandrovPopov} gave a real start to investigations in this direction. Further applications of this combination
were performed in \cite{AleksandrovKPopov} finding the range of the system of functionals $\{\re (e^{i\alpha}a_2),\re(a_3-a_2^2)\}$ in the class $\Sb$, $\alpha\in [0,\pi/2]$. We observe that
during this period Vladimir Ivanovich Popov was a real driving force of this process, see e.g.,~\cite{Popov69}. However, after some first success \cite{AleksandrovPopova, AleksandrovZK} with already known problems or their modifications, some insuperable difficulties did not allow to continue with a significant progress although the ideas were rather clear. Real breakthrough happened in 1984 when Dmitri Prokhorov~\cite{Prokhorov84} started a series of papers in which he solved several new problems by applying PMP together with the L\"owner-Kufarev equation~\eqref{LKord}.

Let us describe a typical application of PMP in the L\"owner-Kufarev theory. As an example we consider one of the most difficult problems in the theory of univalent functions, namely,
description of the boundary of the coefficient body in the class $\Sb$. A beautiful book by Schaeffer and Spenser~\cite{SS} can be our starting point. If $V_n$ denotes the
$n$-th coefficient body, $V_n=V_n(f)=\{a_2,a_3,\dots,a_n\}$, $f=z+a_2z^2+a_3z^3+\dots$ in the class $\Sb$, then the first non-trivial body $V_3$ was completely described in \cite{SS}.
The authors also give some qualitative description of general $V_n$.
Several qualitative results on $\partial V_n$ were obtained in the monograph~\cite{Bobenko}. In particular, Bobenko developed a method of second variation more general than
that by Duren and Schiffer~\cite{DS}, and proved that $\partial V_n$ is smooth except sets of smaller dimension.

Now let $f(z,t)=e^tw(z,t)$, where $w(z,t)$ is a solution to the L\"owner equation \eqref{LKord} with the function $p$ given by \eqref{yadro}. We introduce the matrix $A(t)$ and the vector $a(t)$ as
\[
A(t)=
\left(
\begin{matrix}
0 & 0 & \dots & 0 & 0 \\
a_1(t) & 0 & \dots & 0 & 0 \\
a_2(t) & a_1(t) & \dots & 0 & 0 \\
\vdots & \vdots & \ddots & \vdots & \vdots \\
a_{n-1}(t) & a_{n-2}(t) & \dots & a_1(t) & 0
\end{matrix}
\right),\quad
a(t)=
\left(
\begin{matrix}
a_1(t)  \\
a_2(t) \\
a_3(t) \\
\vdots  \\
a_{n}(t)
\end{matrix}
\right),
\]
where $a_1(t)\equiv 1$ and $a_2(t)$, $a_3(t)$, $\dots$ are the coefficients of the function $f(z,t).$ Substituting $f(z,t)$ in \eqref{LKord} we obtain a controllable system
\begin{equation}\label{CS}
\frac{da(t)}{dt}=-2\sum\limits_{k=1}^{n-1}e^{-k(t+iu(t))}A^{k}(t)a(t),
\end{equation}
with the initial condition $a^T(0)=(1,0,0,\dots,0)$. The coefficient body  $V_n$ is a reachable set for the controllable system~\eqref{CS}.

The results concerning the
structure and properties of $V_n$ include

\begin{itemize}

\item[(i)] $V_n$ is homeomorphic to a $(2n-2)$-dimensional ball and its
boundary $\partial V_n$ is homeomorphic to a $(2n-3)$-dimensional
sphere;

\item[(ii)] every point $x\in \partial V_n$ corresponds to exactly one function $f\in
\Sb$ which will be called a {\it boundary function} for $V_n$;

\item[(iii)] with the exception for a set of smaller dimension,
at every point $x\in \partial V_n$ there exists a normal vector
satisfying the Lipschitz condition;

\item[(iv)] there exists a connected open set $X_1$ on $\partial V_n$,
such that the boundary $\partial V_n$ is an analytic hypersurface at
every point of $X_1$. The points of $\partial V_n$ corresponding to
the functions that give the extremum to a linear functional belong
to the closure of $X_1$.

\end{itemize}

It is worthy noticing that all boundary functions have a similar
structure. They map the unit disk $\mathbb D$ onto the complex plane
$\mathbb C$ minus piecewise analytic Jordan arcs forming a tree with
a root at infinity and having at most $n-1$ tips. This assertion
underlines the importance of multi-slit maps in the coefficient
problem for univalent functions.

Solutions $a(t)$ to (\ref{CS}) for different
control function $u(t)$ (piecewise continuous, in general)  represent  all points of $\partial
V_n$ as $t\to\infty$. The trajectories $a(t)$, $0\leq t <
\infty$, fill $V_n$ so that every point of $V_n$ belongs to a
certain trajectory $a(t)$. The endpoints of these trajectories
can be interior or else boundary points of $V_n$. This way, we
set $V_n$ as the closure of the reachable set for the control system
(\ref{CS}).

According to property (ii) of $V_n$,
every point $x\in \partial V_n$ is attained by exactly one
trajectory  $a(t)$ which is determined by a choice of the
piecewise continuous control function $u(t)$. The function $f\in \Sb$
corresponding to $x$  is a multi-slit map of $\mathbb D$. If the boundary
tree of $f$ has only one tip, then there is a unique continuous
control function $u(t)$ in $t\in [0,\infty)$ that corresponds to
$f$. Otherwise one obtains multi-slit maps for piecewise continuous control $u$.

In order to reach a boundary point $x\in \partial V_n$ corresponding to a
one-slit map, the trajectory $a(t)$ has to obey extremal
properties, i.e., to be an {\it optimal trajectory}. The continuous
control function $u(t)$ must be optimal, and hence, it satisfies a
necessary condition of optimality. PMP is a
powerful tool to be used that provides a joint interpretation of two
classical necessary variational conditions: the Euler equations and
the Weierstrass inequalities (see, e.g., \cite{Pontryagin}).

To realize the maximum principle we consider an {\it adjoint vector} (or momenta)
\[
\psi(t)=\left(
\begin{array}{c}
\psi_1(t)\\
\cdot\\
\cdot\\
\cdot\\
\psi_n(t)
\end{array} \right),
\]
with the complex valued coordinates $\psi_1,\dots,\psi_n$, and the {\it
Hamiltonian function}
\[
H(a,{\psi},u)=\re\left[\left(-2\sum\limits_{k=1}^{n-1}e^{-k(t+iu(t))}A^{k}(t)a(t)\right)^T\bar{\psi}\right],
\]
where $\bar{\psi}$ means the vector with complex conjugate
coordinates. To come to the Hamiltonian formulation for the
coefficient system we require that $\bar{\psi}$ satisfies the
adjoint system of differential equations
\begin{equation}
\frac{d\bar{\psi}}{dt}=-\frac{\partial H}{\partial a},\quad 0\leq
t<\infty.\label{psi}
\end{equation}
Taking into account (\ref{CS}) we rewrite (\ref{psi}) as
\begin{equation}
\frac{d\bar{\psi}}{dt}=2\sum\limits_{k=1}^{n-1}\left(e^{-k(t+iu(t))}(k+1)(A^T)^k\right)\bar{\psi},\quad \psi(0)=\xi.\label{psi2}
\end{equation}
The maximum principle states that any optimal control function
$u^*(t)$ possesses a maximizing property for the Hamiltonian
function along the corresponding trajectory, i.e.,
\begin{equation}
\max\limits_{u}H(a^*(t),\psi^*(t),u)=H(a^*(t),\bar{\psi}^*(t),u^*),
\quad t\geq 0, \label{max}
\end{equation}
where $a^*$ and $\psi^*$ are solutions to the system (\ref{CS},
\ref{psi2}) with $u=u^*(t)$.

The maximum principle (\ref{max}) yields that
\begin{equation}
\frac{\partial H(a^*(t),\psi^*(t),u)}{\partial
u}\bigg|_{u=u^*(t)}=0.\label{max1}
\end{equation}
Evidently, (\ref{CS}), (\ref{psi}), and (\ref{max1}) imply that
\begin{equation}
\frac{d\, H(a^*(t),\psi^*(t),u^*(t))}{dt}=0,\label{ham}
\end{equation}
for an optimal  control function $u^*(t)$.

If the boundary function $f$ gives a point at the boundary hypersurface $\partial V_n$, then its rotation is also  a boundary function. The rotation operation $f(z)\to e^{-i\alpha}f(e^{i\alpha} z)$ gives a curve on $\partial V_n$ and establishes a certain symmetry of the boundary hypersurface and a change of variables $u\to u+\alpha$, $\psi\to e^{i\alpha}\psi$. This allows us
to normalize the adjoint vector as $\im \psi_n=0$ and $\psi_n=\pm 1$, which corresponds to the projection of $V_n$ onto the hypersurface $\im a_n=0$ of dimension $2n-3$.
By abuse of notation, we continue to write $V_n$ for this projection.
 Further
study reduces to investigation of critical points of the equation  $\frac{\partial}{\partial u}H(a(t),\psi(t),u)=0$, and search and comparison of local extrema. In the case when
two local maxima coincide, the Gamkrelidze theory~\cite{Gamkrelidze} of sliding regimes comes into play. Then one considers the generalized L\"owner equation in which
the function $p$ takes the form
\begin{equation}
p(z,t)=\sum\limits_{k=1}^n \lambda_k\frac{e^{iu_k(t)}+z}{e^{iu_k(t)}-z},\quad \sum\limits_{k=1}^n \lambda_k=1, \label{yadro2}
\end{equation}
where $u_k(t)$ is a continuous functions regarding to $t\in [0,\infty)$, and $\lambda_k\geq 0$ are constant.

\begin{theorem}{\rm \cite{Prokhorov90, Prokhorov93}}
Let  $f\in\Sb$ gives  a nonsingular boundary point of the set $V_n$ and let $f$ map the unit disk onto the plane with piecewise-analytic slits having $m$ finite tips. Then there exist
$m$
real-valued functions $u_1,...,u_m$ continuous on $[0,\infty)$ and positive numbers $\lambda_1,\dots,\lambda_m$, $\sum\limits_{k=1}^m \lambda_k=1$, such that a solution $w = w(z,t)$ to the Cauchy problem for the generalized L\"owner differential equation (\ref{LKord},\ref{yadro2})
represents $f$ according to the formula $f(z) = \lim_{t\to\infty} f(z,t)$. This representation is unique.
\end{theorem}

The boundary $\partial V_n$ is then parametrized by the initial conditions $\xi$ in \eqref{psi2}. If the equation \eqref{psi2} allows to choose only one optimal control $u^*$, then the
boundary point is given by a boundary function that maps the unit disk onto the plane with piecewise-analytic slit having a unique finite tips. If the Hamiltonian function has
some $m$ equal local maximums, then this case is called the sliding regime, and it is provided by $m-1$ equations with respect to $\xi$, which determine the equality of the values of the Hamiltonian function at $m$ critical points. The constants $\lambda_k$ represent additional controls of the problem.

Therefore, the sliding regime of the optimal control problem with $m$ maximum points of the Hamiltonian function is realized when $\xi\in\mathfrak M_m$, where the manifold $\mathfrak M_m$ of dimension $2n-3-m$ is defined by $m-1$ equations for the values of the Hamiltonian function at $m$ critical points for $t = 0$. It is shown in \cite{Prokhorov93} that the sliding regime with $m$ maximum points of function $H$ at $t = 0$ preserves this property for $t > 0$. The number of maximum points may only decrease and only because of the joining of some of them at certain instants $t$.

\begin{theorem}{\rm \cite{Prokhorov90, Prokhorov93}}
The boundary hypersurface $\partial V_n$, $n\geq 2$, is a union of the sets $\Omega_1,\dots,\Omega_{n-1}$,
every pair of which does not have mutual interior points. Each set $\Omega_m$, $1 \geq m \geq n - 1$, corresponds to
the manifold $\mathfrak M_m$, $\mathfrak M_1 = \mathbb R^{2n-4}$, so that the parametric representation
\[
\Omega_m=\left\{a(\infty,\xi,\lambda)\colon \xi\in\mathfrak M_m;\, \xi_n=\pm1;\, \lambda_1,\dots,\lambda_m\geq 0 ;\,\sum\limits_{k=1}^m \lambda_k=1\right\}
\]
holds, where $a(\infty,\xi,\lambda)$ is the manifold coordinate of the system of bicharacteristics $a,\psi$ for the Hamiltonian system (\ref{CS},\ref{psi2}) with continuous branches of the optimal controls given by \eqref{max1}.
\end{theorem}

This theorem can be used in investigations of topologic, metric, smooth, and analytic properties of the boundary hypersurface $\partial V_n$. In particular, the qualitative results in \cite{SS, Bobenko} can be obtained as corollaries. Equally, this method work for the subclass $\Sb^M$ of bounded functions $|f(z)|\leq M$ from $\Sb$. In this case, one stops at the moment $t=\log M$.

Let us here mention a slightly different approach proposed by Friedland and Schiffer \cite{FS76, FS77} in 1976-77 to the problem of description of $V_n$. Instead of the ordinary L\"owner differential equation \eqref{LKord}, they considered the L\"owner partial differential equation \eqref{LK} with the function $p$ in the form \eqref{yadro}. Their main conclusion
states that if the initial condition for the equation \eqref{LK} is a boundary function for $V_n$, then the solutions to  \eqref{LK} form a one-parameter curves on $\partial V_n$ (for special choice of $u$). The conclusions on the special character of the Koebe functions reflect the angular property of the corresponding points at the boundary hypersurface.

Besides these general results on the coefficient body several concrete problems and conjectures were solved by this method. In particular, several projections of $V_3$ were described
in \cite{Vas85}. The range of the system of set  of functionals $\{f(r_1), f(r_2)\}$, $0<r_1<r_2<1$, was given in~\cite{Vas90} in the subclass $\Sb_R$ of functions from $\Sb$ with real coefficients.
Prokhorov~\cite{Prokhorov91} described $V_4$ in the class $\Sb_R^M=\Sb_R\cap \Sb^M$ generalizing and improving Tammi's results~\cite{Tammi1, Tammi2}. The range of the system
of functionals $\{f(r),a_2,a_3\}$ was described in \cite{Vas90} in the class $\Sb_R^M$.

We would like to mention here some conjectures which were confirmed or disproved by the above described method. First let us introduce the so-called Pick function that plays in the class $\Sb^M$ a role
analogous to the Koebe function $k_{\theta}(z)=\frac{z}{(1-e^{i\theta}z)^2}$ in the class $\Sb$. The Pick function $p_{\theta}(z)$ satisfies the equation
\[
\frac{M^2 p_{\theta}}{(M-p_{\theta})^2}=k_{\theta}(z),\quad M>1,\quad z\in\mathbb D.
\]
The coefficients of $p_{0}(z)=z+p_2(M)z^2+\dots$ we denote by $p_n(M)$. The Jakubowski conjecture says that the estimate $|a_n|\leq p_n(M)$ holds for even $n$ in the class $\Sb^M$ for
sufficiently large $M$. This conjecture was proved in \cite{Prokhorov93a}. As for odd coefficients of univalent functions, it is easy to verify that the necessary maximality conditions for the Pick functions fail for large $M$, see \cite{Prokhorov93a}.

Bombieri \cite{Bombieri} in 1967 posed the problem to find
$$
\sigma_{mn}:=\liminf_{f\to K, f\in S}\frac{n-\R a_n}{m-\R a_m},
\;\;\;m,n\geq2,
$$
$f\to K$ locally uniformly in $U$. We call $\sigma_{mn}$ the Bombieri numbers.
 He conjectured that $\sigma_{mn}=B_{mn}$, where
$$
B_{mn}=\min\limits_{\theta\in [0,2\pi)}\frac{n\,\sin \theta-\sin(n\theta)}{m\,\sin \theta-\sin(m\theta)}.
$$
and proved that $\sigma_{mn}\leq B_{mn}$ for $m=3$ and $n$ odd. It
is noteworthy that Bshouty and Hengartner \cite{BH} proved
Bombieri's conjecture for functions from $S$ having real
coefficients in their Taylor's expansion. Continuing this
contribution by Bshouty and Hengartner, the conjecture for
the whole class $S$ has been  disproved by Greiner and
Roth \cite{GreinerR} for $n=2$, $m=3$, $f\in S$. Actually, they have got
the sharp Bombieri number $\sigma_{32}=(e-1)/4e<1/4=B_{32}$.

It is easily seen that $\sigma_{43}=B_{43}=\sigma_{23}=B_{23}=0$.
Applying L\"owner's parametric representation for univalent
functions and the optimal control method we  found \cite{PV} the exact
Bombieri numbers $\sigma_{42},\sigma_{24}, \sigma_{34}$ and their
numerical approximations $\sigma_{42}\approx 0.050057\dots$,
$\sigma_{24}\approx 0.969556\dots$, and $\sigma_{34}\approx
0.791557\dots$ (the Bombieri conjecture for these permutations of
$m,n$ suggests $B_{42}=0.1$, $B_{24}=1$, $B_{34}=0.828427\dots$). Of
course, our method permits us to reprove the result of \cite{GreinerR} about
$\sigma_{32}$.

Our next target is the fourth coefficient $a_4$ of a function from
$\Sb^M$.  The
sharp estimate $|a_2|\leq 2(1-1/M)=p_2(M)$ in the class $S^M$ is
rather trivial and has been obtained by Pick \cite{Pick} in 1917. The
next coefficient $a_3$ was estimated independently by
Schaeffer and Spencer \cite{SS45} in 1945 and Tammi \cite{Tammi53}
in 1953. The Pick function does not give the maximum to $|a_3|$
and the estimate is much more difficult. Schiffer and Tammi
\cite{ST} in 1965 found that $|a_4|\leq p_4(M)$ for any $f\in S^M$ with
$M>300$. This result was repeated by Tammi \cite[page 210]{Tammi1} in a
weaker form ($M>700$) and there it was conjectured that this
constant could be decreased until 11. The case of function with
real coefficients is simpler: the Pick function gives the maximum
to $|a_4|$ for $M\geq 11$ and this constant is sharp (see \cite{Tammi73},
\cite[page 163]{Tammi2}). By our suggested method we showed \cite{PV} that the Pick
function locally maximizes $|a_4|$ on $S^M$ if $M>
M_0=22.9569\dots$ and does not for $1< M<M_0$. This disproves
Tammi's conjecture.

Among other results obtained by PMP we mention here the Charzy\'nski-Tammi conjecture on  functions from $\Sb^M$ close to identity (or $M$ close to 1). The suggested extremal function
analogous to the Koebe and Pick functions is $n-1$-symmetric function
\[
Q_{n-1}^M(z)=\left(p_0^{M^{n-1}}(z^{n-1})\right)^{1/(n-1)}.
\]
The conjecture proved by Schiffer and Tammi \cite{ST} and Siewierski \cite{Sw} states that the coefficient of the function $Q_{n-1}^M(z)$ gives the extremum to the coefficient $a_n$
of a function from $\Sb^M$, namely
\[
|a_n|\leq \frac{2}{n-1}\left(1-\frac{1}{M^{n-1}}\right),\quad n\geq 2.
\]
Prokhorov~\cite{Prokhorov97} found an asymptotic estimate in the above problem in terms of $\log^2 M$ as $M$ is close to 1. For $M\to \infty$
Prokhorov and Nikulin~\cite{ProkhorovN} obtained also asymptotic estimates in the coefficient problem for the class $\Sb^M$ with . In particular,
\[
|a_n|\leq n-\frac{n(n^2-1)}{3}\frac{1}{M}+o\left(\frac{1}{M}\right),\quad M\to\infty.
\]
More about coefficient problems solved by PMP and L\"owner-Kufarev theory, see~\cite{Prokhorov01}.

\section{One-slit maps}

It is interesting that in spite of many known properties of the L\"owner equations several geometric question
remained unclear until recently.

Let us return back to the original L\"owner equation \eqref{LKord} with the driving function $p$ given by \eqref{yadro}
\begin{equation}
\frac{dw}{dt}=-w\frac{e^{iu(t)}+w}{e^{iu(t)}-w}, \quad w(z,0)\equiv z, \label{LE}
\end{equation}
Solutions  $w(z,t)$ to \eqref{LE}  map $\mathbb D$ onto $\Omega(t)\subset \mathbb D$.
 If
$\Omega(t)= \mathbb D\setminus \gamma(t)$, where $\gamma(t)$ is a Jordan curve in $\mathbb D$ except one of its endpoints, then the
driving term $u(t)$ is uniquely defined and we call the corresponding map $w$ a {\it slit
map}. However, from 1947 \cite{Kufarev47a} it is known that  solutions to (\ref{LE}) with
continuous $u(t)$ may give non-slit maps, in particular, $\Omega(t)$ can be a family of
hyperbolically convex digons in $\mathbb D$.

Marshall and Rohde \cite{Marshall} addressed the following question: {\it Under which
condition on the driving term $u(t)$ the solution to (\ref{LE}) is a slit map?} Their result
states that if $u(t)$ is Lip(1/2) (H\"older continuous with  exponent 1/2), and if for a certain
constant $C_{\mathbb D}>0$, the norm $\|u\|_{1/2}$ is bounded $\|u\|_{1/2}<C_{\mathbb D}$, then the solution $w$ is a
slit map, and moreover, the Jordan arc $\gamma(t)$ is s quasislit (a quasiconformal image of an interval within a Stolz angle). As they also proved, a
converse statement without the norm restriction holds. The absence of the norm restriction in
the latter result is essential. On one hand, Kufarev's example \cite{Kufarev47a} contains
$\|u\|_{1/2}=3\sqrt{2}$, which means that $C_{\mathbb D}\leq 3\sqrt{2}$. On the other hand,
Kager, Nienhuis, and Kadanoff  \cite{Kadanoff} constructed exact slit solutions to the
half-plane version of the L\"owner equation with arbitrary norms of the driving term.

Reminding the half-plane version of the L\"owner equation let $\mathbb H=\{z: \im
z>0\}$, $\mathbb R=\partial \mathbb H$. The functions $h(z,t)$, normalized near infinity by
$h(z,t)=z-2t/z+b_{-2}(t)/z^2+\dots$,  solving the equation
\begin{equation}
\frac{dh}{dt}=\frac{-2}{h-\lambda(t)}, \quad h(z,0)\equiv z, \label{LE2}
\end{equation}
where $\lambda(t)$ is a real-valued
continuous driving term, map $\mathbb H$ onto a subdomain of $\mathbb H$.  The difference in the sign between \eqref{LE2} and \eqref{hydro2} is because in \eqref{hydro2} the equation is for the inverse mapping.
In some papers, e.g., \cite{Kadanoff, Lind}, the authors work with  equations (\ref{LE},
\ref{LE2}) changing (--) to (+) in their right-hand sides, and with the mappings of slit domains
onto $\mathbb D$ or $\mathbb H$. However, the results remain the same for both versions.

The question about the slit mappings and the behaviour of the driving
term $\lambda(t)$ in the case of the half-plane $\mathbb H$ was addressed by Lind \cite{Lind}.
The techniques used by Marshall and Rohde carry over to prove a similar result in the case of
the equation (\ref{LE2}), see \cite[page 765]{Marshall}. Let us denote by $C_{\mathbb H}$ the
corresponding bound for the norm $\|\lambda\|_{1/2}$. The main result by Lind is the sharp
bound, namely $C_{\mathbb H}=4$.

Marshall and Rohde \cite{Marshall} remarked that there exist many examples of driving
terms $u(t)$ which are not Lip(1/2), but which generate slit solutions with simple arcs $\gamma(t)$. In
particular, if $\gamma(t)$ is tangent to $\mathbb T$, then $u(t)$ is never Lip(1/2).

Our result~\cite{ProkhVas09} states that if $\gamma(t)$ is a circular arc tangent to $\mathbb R$, then the
driving term $\lambda(t)\in$Lip(1/3). Besides, we prove that $C_{\mathbb D}=C_{\mathbb H}=4$,
and consider properties of singular solutions to the one-slit L\"owner equation. Moreover, examples of non-slit
solutions filling the whole spectrum $[4,\infty)$ were given in \cite{IPV, LMR}.

The authors analyzed in \cite{LMR} L\"owner traces $\gamma(t)$ driven by $\lambda(t)$ asymptotic to $k\sqrt{1-t}$.
They proved  a form of stability of the self-intersection for such $\lambda(t)$. Being slightly rephrased it reads as follows.
\begin{theorem}{\rm \cite{LMR}}
Let the driving term $\lambda\colon [0,1]\to \mathbb R$ be sufficiently regular with
 the above asymptotic of  $\lambda(t)$
 \[
 \lim\limits_{t\to 1}\frac{|\lambda(1)-\lambda(t)|}{\sqrt{1-t}}=k>4.
 \]
 Then $\gamma(1-0)$ exists, is real, and $\gamma$ intersects $\mathbb R$ at the same angle as the trace for $\lambda=k\sqrt{1-t}$.
\end{theorem}
Namely,
\[
\lim\limits_{t\to 1}\arg(\gamma(t)-\gamma(1))=\pi\frac{1-\sqrt{1-16/k^2}}{1+\sqrt{1-16/k^2}}.
\]
The method of proof of the above theorem also applies to the case $|k|<4$. In this case the trace $\gamma$ driven by $\lambda$ is a Jordan arc, and $\gamma$ is asymptotically similar to the logarithmic spiral at $\gamma(1)\in \mathbb H$.

Another our result~\cite{IPV} states that an analytic orthogonal slit requires  the 1/2 Lipschitz  vanishing norm, exactly as in Kadanoff's {\it et al}. examples \cite{Kadanoff} with
a line-slit and a circular slit. In this case the conformal radius approaching the origin is of
order Lip 1/2 (compare with Earle and Epstein \cite{Earle}).

\section{Univalence criteria}

Several important univalence criteria can be obtained by means of the L\"owner-Kufarev differential equation. For example, a function $f(z)=z+\dots$ analytic in $\mathbb D$ is spirallike
of type $\alpha\in(-\pi/2,\pi/2)$ (and therefore, univalent) if and only if
\[
\re\left(e^{i\alpha}z\frac{f'(z)}{f(z)}\right)>0,\quad\text{in $\mathbb D$},
\]
see \cite{Spacek, Robertson61} and \cite[page 172]{Pommerenke}. Spirallikeness means that a function $f(z)$ is analytic, univalent, and if $w\in f(\mathbb D),$ $\tau\geq 0$, then
$we^{-e^{-i\alpha}\tau}\in f(\mathbb D)$. If $\alpha=0$, then we obtain the usual class of starlike functions $S^*$.

Next criterion is obtained by integration of the L\"owner-Kufarev equation \eqref{LKord} with a special choice of the driving function $p$ from the Carath\'eodory class.
Let us choose
\[
p(z,t)=\frac{1}{h(z)+th_0(z)}, \quad h_0(z)=i\beta+\alpha\frac{z g'(z)}{g(z)},\quad g\in S^*.
\]
Integrating \eqref{LKord} as a Bernoulli-type equation, and letting $t\to\infty$ leads to the limiting function
\[
f(z)=\left(\frac{\alpha+i\beta}{1+i\alpha}\int_0^z h(z) z^{i\beta-1}g^{\alpha}(z)\, dz\right)^{1/(\alpha+i\beta)}=z+\dots,
\]
$\alpha>0$, $\beta\in \mathbb R$, which is the Bazilevich class $B_{\alpha,\beta}$ introduced in~\cite{Bazilevich55, Bazilevich64}.
Prokhorov~\cite{Prokhorov75} and Sheil-Small~\cite{Sheil} proved that the class $B_{\alpha,\beta}$ is characterized  by the condition
\[
\int_{\theta_1}^{\theta_2}\re F(re^{i\theta})\,d\theta> -\pi,\qquad 0<r<1,\quad 0<\theta_2-\theta_1<2\pi,
\]
\[
F(z)=1+z\frac{f''(z)}{f'(z)}+(\alpha-1+i\beta)z\frac{f'(z)}{f(z)},
\]
under the additional assumption that $f(z)f'(z)/z\neq 0$ on $\mathbb T=\partial \mathbb D$. Moreover, the boundary of $f(\mathbb D)$ is accessible from outside
by the curves $w=a(1+bt)^{1/(\alpha+i\beta)}$.

In 1972, J.~Becker~\cite{Becker72} assumed that the driving function $p$ in \eqref{LKord} satisfied the inequality
\[
\bigg|\frac{p(z,t)-1}{p(z,t)+1}\bigg|\leq k<1, \quad z\in\mathbb D, \quad 0\leq t<\infty.
\]
Then the solutions in the form \eqref{limit} have $k$-quasiconformal extension to the Riemann sphere $\hat{\mathbb C}$. This allowed him to
show that the inequality
\[
(1-|z|^2)\bigg|z\frac{f''(z)}{f'(z)}\bigg|\leq k<1
\]
in $\mathbb D$ for a  analytic function $f(z)=z+\dots$ implies that $f$ is univalent and has a $k$-quasiconformal extension to $\hat{\mathbb C}$. This improves a result of Duren, Shapiro and Shields~\cite{DSS}. Later in 1984 Becker and Pommerenke~\cite{BP} established the criteria
\[
(1-|z|^2)|zf''(z)/f'(z)|\le1,\ f'(0)\ne0\ \ (z\in{\mathbb D})
\]
\[
2\,\text{Re}\,z|f''(z)/f'(z)|\le1\ \ (z\in \mathbb H)
\]
\[
(|z|^2-1)|zf''(z)/f'(z)|\le1\ \ (z\in\mathbb D^-),
\]
where $\mathbb H$ is the right half-plane and $\mathbb D^-$ is the exterior of $\mathbb D$. In all inequalities  the constant 1 is  the best possible. The first criterium
implies that the boundary of $f(\mathbb D)$ is Jordan whereas the second and the third do not necessary imply this. Various univalence conditions were obtained later,
see~\cite{Aks} for more complete list of them.

\section{Semigroups}
Looking at the classical radial L\"owner-Kufarev equation \eqref{LKord} and the classical chordal  L\"owner-Kufarev equation \eqref{hydro2}, one notices that there is a similitude between the two. Indeed, we can write both equations in the form
\[
\frac{d z(t)}{d t}=G(z,t),
\]
with
\[
G(z,t)=(\tau-z)(1-\overline{\tau}z)p(z,t),
\]
where $\tau=0, 1$ and  $\re  p(z,t)\geq0$ for all $z\in \mathbb D$ and $t\geq 0$.

The reason for the previous formula is not at all by chance, but it reflects a very important feature of `Herglotz vector fields' (see Section \ref{generaltheory} for the definition). In order to give a rough idea of what we are aiming, consider the case $\tau=0$ (the radial case). Fix $t_0\in [0,+\infty)$. Consider the holomorphic vector field $G(z):=G(z,t_0)$. Let $g(z):=|z|^2$. Then,
\begin{equation}\label{stringo}
dg_z(G(z))=2\re \langle G(z), z\rangle =-|z|^2\re p(z,t_0) \leq 0, \quad \forall z\in \mathbb D.
\end{equation}
This Lyapunov type inequality has a deep geometrical meaning. Indeed, \eqref{stringo} tells that $G$ points toward the center of the level sets of $g$, which are concentric circles centered at $0$. For each $z_0\in \mathbb D$, consider then the Cauchy problem
\begin{equation}\label{Cauchy}
\begin{cases}
\frac{d w(t)}{dt}=G(w(t)),\\
w(0)=z_0
\end{cases}
\end{equation}
and let $w^{z_0}:[0,\delta)\to \mathbb D$ be the maximal solution (such a solution can propagate also in the `past', but we just consider the `future' time). Since $G$ points inward with respect to all circles centered at $0$, the flow $t\mapsto w^{z_0}(t)$ cannot escape from the circle of radious $g(z_0)$. Therefore, the flow is defined for all future times, namely, $\delta=+\infty$. This holds for all $z_0\in \mathbb D$.

Hence, the vector field $G$ has the feature to be $\mathbb R^+$-semicomplete, that is, the maximal solution of the initial value problem \eqref{Cauchy}  is defined in the interval $[0,+\infty)$. To understand how one can unify both radial and chordal L\"owner theory, we dedicate this section to such vector fields and their flows.

A family  of
holomorphic self-maps of the unit disk $(\phi_t)$ is
a {\sl  (one-parameter) semigroup (of holomorphic functions)} if $\phi\colon (\mathbb R^+, +)\to {\sf
Hol}(\mathbb D,\mathbb D)$ is a continuous homomorphism between the
semigroup of non-negative real numbers and the semigroup of
holomorphic self-maps of the disk with respect to composition,
endowed with the topology of uniform convergence on compact sets.
In other words:
\begin{itemize}
\item[$\bullet$] $\phi_0=\textrm{id}_\mathbb D$;
\item[$\bullet$] $\phi_{t+s}=\phi_s\circ\phi_t$
for all $s$,~$t\geq 0$;
\item[$\bullet$] $\phi_t$ converges to $\phi_{t_0}$
uniformly on compact sets as $t$ goes to~$t_0$.
\end{itemize}
It can be shown that if $(\phi_t)$ is a semigroup, then $\phi_t$ is univalent for all $t\geq 0$.

Semigroups of holomorphic maps are a classical subject of study, both as (local/global) flows of
continuous dynamical systems and from the point of view of
`fractional iteration', the problem of embedding the discrete set of iterates
generated by a single self-map into a one-parameter family (a problem that is still open even in the disk). It is difficult to
 exactly date the birth of this notion but it seems that the
first paper dealing with semigroups of holomorphic maps and their asymptotic
behaviour is due to Tricomi in 1917 \cite{Tricomi}. Semigroups of holomorphic maps also appear in connection with the theory of
Galton-Watson processes (branching processes) started in the
40s by Kolmogorov and Dmitriev \cite{Harris}. An extensive recent survey \cite{Goryainov12} gives a complete overview on details.
Furthermore, they are an important tool in the
theory of strongly continuous semigroups of operators between
spaces of analytic functions (see, for example, \cite{Siskakis}).

A very important contribution to the theory of semigroups of
holomorphic self-maps of the unit disk is due to Berkson and Porta
\cite{Berkson-Porta}. They proved that:

\begin{theorem}[\cite{Berkson-Porta}] A semigroup of holomorphic self-maps of the unit disk $(\phi_t)$
is in fact real-analytic in the variable
$t$, and is the solution of the Cauchy problem
\begin{equation}\label{semig-eq}
\frac{\partial \phi_t(z)}{\partial t}=G(\phi_t(z)), \quad \phi_0(z)=z\;,
\end{equation}
where the map $G$, the {\sl infinitesimal generator} of the semigroup, has the form
\begin{equation}\label{G-BK}
G(z)=(z-\tau)(\overline{\tau}z-1)p(z)
\end{equation}
for some $\tau \in \overline{\mathbb D}$ and a holomorphic function
$p\colon\mathbb D\to\mathbb C$ with $\re p\geq 0$.

Conversely, any vector field of the form \eqref{G-BK} is semicomplete and if, for $z\in \mathbb D$, we  take $w^z$ the solution of the initial value problem
$$
\frac{dw}{dt}=G(w), \quad w(0)=z,
$$
then $\phi_t(z):=w^z(t)$ is a semigroup of analytic functions.
\end{theorem}
Expression \eqref{G-BK} of the infinitesimal generator is known as its {\sl Berkson-Porta decomposition}.
Other characterizations of vector fields which are infinitesimal generators of semigroups can be seen in \cite{BCD} and references therein.

The dynamics of the semigroup $(\phi_t)$ are governed by the
analytical properties of the infinitesimal generator $G$.
For instance, all the functions of the semigroup have a common fixed point at $\tau$
(in the sense of non-tangential limit if $\tau$ belongs to the
boundary of the unit disk) and asymptotically tends to $\tau$,
which can thus be considered a sink point of the dynamical
system generated by $G$.

When $\tau=0$, it is clear that \eqref{G-BK} is a particular case of
\eqref{LKord}, because the infinitesimal
generator $G$ is of the form $-wp(w)$, where $p$ belongs to the Carath\' eodory class.   As a
consequence, when the semigroup has a fixed point in the
unit disk (which, up to a conjugation by an automorphism of the disk,
amounts to taking $\tau=0$), {\sl once
differentiability in $t$ is proved} Berkson-Porta's theorem can be
easily deduced from L\"owner's theory. However, when the semigroup has no common fixed
points in the interior of the unit disk, Berkson-Porta's result is really a new
advance in the theory.

We have already remarked that semigroups give rise to evolution families;
they also provide examples of L\"owner chains. Indeed,
Heins \cite{Heins} and Siskasis \cite{Siskakis-tesis} have
independently proved that if $(\phi_t)$ is a semigroup of holomorphic self-maps of the
unit disk then there exists a (unique, when suitably
normalised) holomorphic function $h\colon\mathbb D\to \mathbb C$, the {\sl K\"onigs function} of the
semigroup,
such that $h(\phi_t(z))=m_t(h(z))$ for all $t\geq
0$, where $m_t$ is an affine map (in other words, the semigroup is semiconjugated to a
semigroup of affine maps). Then it is easy to see that the maps $f_t(z):=m_t^{-1}(h(z))$, for $t\geq
0$, form a L\"owner chain (in the sense explained in the next section).

The theory of semigroups of holomorphic self-maps has been extensively studied and
generalised: to Riemann surfaces (in particular, Heins \cite{Heins} has shown that Riemann
surfaces with non-Abelian fundamental group admit no non-trivial semigroup of
holomorphic self-maps); to several complex variables; and to infinitely dimensional
complex Banach spaces, by Baker, Cowen, Elin, Goryainov, Poggi-Corradini,
 Pommerenke, Reich,  Shoikhet, Siskakis,  Vesentini, and many others. We refer to
\cite{BCD} and the books \cite{Abate} and
\cite{Reich-Shoikhet} for references and more information on
the subject.

\section{General theory: Herglotz vector fields, evolution families and L\"owner chains}\label{generaltheory}

Although the chordal and radial versions of the L\"owner Theory share common
ideas and structure, on their own they can only be regarded as parallel but
independent theories. The approach of~\cite{Goryainov-Ba, Goryainov,
Goryainov1996, Dubovikov, Goryainov-Kudryavtseva, GoryainovTalk} does show that
there can be (and actually are) much more independent variants of L\"owner
evolution bearing similar structures, but does not solve the problem of
constructing a unified theory covering all the cases.

Recently a new unifying approach has been suggested by Gumenyuk and the three first
authors~\cite{BCM1, BCM2,AbstractLoew, SMP}.
In the previous section we saw that the vector fields which appear in radial and chordal L\"owner equations have the property to be infinitesimal generators for all fixed times. We will exploit such a fact to define a general family of Herglotz vector fields.
Thus, relying partially on the theory of one-parametric semigroups, which can
be regarded as the autonomous version of L\"owner Theory, we can build a new general theory.

\begin{definition}[\cite{BCM1}]~\label{D_BCM-VF}
Let $d\in [1,+\infty]$. A {\it weak holomorphic vector field of order $d$} in
the unit disk $\mathbb{D}$ is a function
$G:\mathbb{D}\times\lbrack0,+\infty)\rightarrow \mathbb{C}$ with the following
properties:

\begin{enumerate}
\item[WHVF1.] for all $z\in\mathbb{D},$ the function $\lbrack
0,+\infty)\ni t\mapsto G(z,t)$ is measurable,

\item[WHVF2.] for all $t\in\lbrack0,+\infty),$ the function $
\mathbb{D}\ni z\mapsto G(z,t)$ is holomorphic,

\item[WHVF3.] for any compact set $K\subset\mathbb{D}$ and  any $T>0$ there
exists a non-negative function $k_{K,T}\in L^{d}([0,T],\mathbb{R})$ such that
\[
|G(z,t)|\leq k_{K,T}(t)
\]
for all $z\in K$ and for almost every $t\in\lbrack0,T].$
\end{enumerate}
We say that $G$ is a {\it (generalized) Herglotz vector field} of order $d$ if,
in addition to conditions~WHVF1\,--\,WHVF3 above, for almost every $t\in
[0,+\infty)$ the holomorphic function $G(\cdot, t)$ is an infinitesimal
generator of a one-parametric semigroup in ${\sf Hol}(\mathbb D,\mathbb D)$.
\end{definition}

Herglotz vector fields in the unit disk can be
decomposed by means of Herglotz functions (and this the reason
for the name). We begin with the following definition:

\begin{definition}
Let $d\in [1,+\infty]$. A {\sl Herglotz function of order $d$}
is a function $p:\mathbb{D}\times\lbrack0,+\infty
)\mapsto\mathbb{C}$ with the following properties:

\begin{enumerate}
\item For all $z\in\mathbb{D},$ the function $\lbrack0,+\infty
)\ni t \mapsto p(z,t)\in\mathbb{C}$ belongs to
$L_{loc}^{d}([0,+\infty),\mathbb{C})$;

\item For all $t\in\lbrack0,+\infty),$ the function
$\mathbb{D}\ni z \mapsto p(z,t)\in\mathbb{C}$ is holomorphic;

\item For all $z\in\mathbb{D}$ and for all $t\in\lbrack0,+\infty),$ we
have $\re p(z,t)\geq0.$
\end{enumerate}
\end{definition}

Then we have the following result which, using the Berkson-Porta formula, gives a general form of the classical Herglotz vector fields:

\begin{theorem}[\cite{BCM1}]Let
$\tau:[0,+\infty)\rightarrow\overline {\mathbb{D}}$ be a
measurable function and let $p:\mathbb D\times [0,+\infty)\to \mathbb C$ be a
Herglotz function of order $d\in [1,+\infty)$. Then the map
$G_{\tau,p}:\mathbb{D}\times\lbrack
0,+\infty)\rightarrow\mathbb{C}$ given by
\begin{equation*}
G_{\tau,p}(z,t)=(z-\tau(t))(\overline{\tau(t)}z-1)p(z,t),
\end{equation*}
for all $z\in\mathbb{D}$ and for all $t\in\lbrack0,+\infty),$
is a Herglotz vector field  of order $d$ on the unit disk.

Conversely, if $G:\mathbb D\times [0,+\infty)\to \mathbb C$ is a Herglotz
vector field of order $d\in [1,+\infty)$ on the unit disk, then
there exist a measurable function
$\tau:[0,+\infty)\rightarrow\overline {\mathbb{D}}$ and a
Herglotz function $p:\mathbb D\times [0,+\infty)\to \mathbb C$  of order $d$
such that $G(z,t)=G_{\tau, p}(z,t)$ for almost every $t\in
[0,+\infty)$ and all $z\in \mathbb D$.

Moreover, if $\tilde{\tau}:[0,+\infty)\rightarrow\overline
{\mathbb{D}}$ is another measurable function and
$\tilde{p}:\mathbb D\times [0,+\infty)\to \mathbb C$ is another Herglotz
function of order $d$ such that $G=G_{\tilde{\tau}, \tilde{p}}$
for almost every $t\in [0,+\infty)$ then
$p(z,t)=\tilde{p}(z,t)$ for almost every $t\in [0,+\infty)$ and
all $z\in \mathbb D$ and $\tau(t)=\tilde{\tau}(t)$ for almost all
$t\in [0,+\infty)$ such that $G(\cdot, t)\not\equiv 0$.
\end{theorem}

\begin{remark} The generalized L\"owner\,-\,Kufarev equation
\begin{equation}\label{EQ_BCM-ODE}
\frac{dw}{dt}=G(w,t),~~ t\ge s,\quad w(s)=z,
\end{equation}
resembles the radial L\"owner\,-\,Kufarev ODE when $\tau\equiv0$ and
$p(0,t)\equiv1$. Furthermore, with the help of the Cayley map between~$\mathbb D$ and
$\mathbb H$, the chordal L\"owner equation appears to be the special case
of~\eqref{EQ_BCM-ODE} with~$\tau\equiv1$.
\end{remark}

We also give a generalization of the concept of evolution families
 in {\it the whole semigroup} ${\sf Hol}(\mathbb D,\mathbb D)$ as follows:
\begin{definition}(\cite{BCM1})~\label{D_BCM-EF}
A family $(\varphi_{s,t})_{t\ge s\ge0}$ of holomorphic self-maps of the unit
disk is an {\it evolution family of order $d$} with $d\in [1,+\infty]$  if
\begin{enumerate}
\item[EF1.] $\varphi_{s,s}={\textrm{ Id}}_{\mathbb{D}}$ for all $s\ge0$,

\item[EF2.] $\varphi_{s,t}=\varphi_{u,t}\circ\varphi_{s,u}$ whenever $0\leq
s\leq u\leq t<+\infty,$

\item[EF3.] for every $z\in\mathbb{D}$ and every $T>0$ there exists a
non-negative function $k_{z,T}\in L^{d}([0,T],\mathbb{R})$ such that
\[
|\varphi_{s,u}(z)-\varphi_{s,t}(z)|\leq\int_{u}^{t}k_{z,T}(\xi)d\xi
\]
whenever $0\leq s\leq u\leq t\leq T.$
\end{enumerate}
\end{definition}
Condition EF3 is to guarantee that any evolution family can be obtained via
solutions of an ODE which resembles both the radial and chordal
L\"owner\,-\,Kufarev equations. The vector fields that drive this generalized
L\"owner\,-\,Kufarev ODE are referred to as {\it Herglotz vector fields.}

\begin{remark}\label{RM_EF-uni}
Definition~\ref{D_BCM-EF} does not require elements of an evolution family to
be univalent. However, this condition is satisfied. Indeed, by
Theorem~\ref{TH_BCM-EF-VF}, any evolution family $(\varphi_{s,t})$ can be
obtained via solutions to the generalized L\"owner\,-\,Kufarev ODE. Hence the
univalence of $\varphi_{s,t}$'s follows from the uniqueness of solutions to
this~ODE. For an essentially different direct proof see
\cite[Proposition~3]{BCM2}.
\end{remark}

\begin{remark} Different notions of evolution families considered previously
in the literature can be reduced to special cases of $L^d$-evolution families
defined above.

The Schwarz lemma and distortion estimates imply that solutions of the classical L\"owner-Kufarev equation \eqref{LKord} are evolutions families of order $\infty$. Also, it can be proved that  for all semigroups of analytic maps $(\phi_t)$, the biparametric family $(\varphi_{s,t}):=(\phi_{t-s})$ is also an evolution family of order $\infty$.
\end{remark}

Equation~\eqref{EQ_BCM-ODE} establishes a 1-to-1 correspondence between
evolution families of order~$d$ and Herglotz vector fields of the same order.
Namely, the following theorem takes place.

\begin{theorem}{\rm (\cite[Theorem~1.1]{BCM1})}\label{TH_BCM-EF-VF} For any evolution family
$(\varphi_{s,t})$ of order $d\in[1,+\infty]$ there exists an (essentially)
unique Herglotz vector field $G(z,t)$ of order $d$ such that for every~$z\in
\mathbb D$ and every~$s\ge0$ the function $[s,+\infty)\ni t\mapsto
w_{z,s}(t):=\varphi_{s,t}(z)$ solves the initial value
problem~\eqref{EQ_BCM-ODE}.

Conversely, given any Herglotz vector field $G(z,t)$ of order
$d\in[1,+\infty]$, for every~$z\in \mathbb D$ and every~$s\ge0$ there exists a unique
solution $[s,+\infty)\ni t\mapsto w_{z,s}(t)$ to the initial value
problem~\eqref{EQ_BCM-ODE}. The formula $\varphi_{s,t}(z):=w_{z,s}(t)$ for all
$s\ge0$, all $t\ge s$, and all $z\in\mathbb D$, defines an evolution
family~$(\varphi_{s,t})$ of order~$d$.
\end{theorem}
Here by {\it essential uniqueness} we mean that two Herglotz vector fields
$G_1(z,t)$ and $G_2(z,t)$ corresponding to the same evolution family must
coincide for a.e. $t\ge0$.

The strong relationship between semigroups and evolution families on the one side and Herglotz vector fields and infinitesimal generators on the other side, is very much reflected by the so-called `product formula' in convex domains of S. Reich and D. Shoikhet \cite{Reich-Shoikhet}. Such a formula can be rephrased as follows: let $G(z,t)$ be a Herglotz vector field. For almost all $t\geq 0$, the holomorphic vector field $\mathbb D\ni z\mapsto G(z,t)$ is an infinitesimal generator. Let $(\phi_r^t)$ be the associated semigroups of holomorphic self-maps of $\mathbb D$. Let $(\varphi_{s,t})$ be the evolution family associated with $G(z,t)$. Then, uniformly on compacta of $\mathbb D$ it holds
\[
\phi_t^r=\lim_{m\to \infty} \varphi_{t,t+\frac{r}{m}}^{\circ m}=\lim_{m\to \infty} \underbrace{(\varphi_{t,t+\frac{r}{m}}\circ \ldots\circ \varphi_{t,t+\frac{r}{m}})}_m.
\]
Using such a formula for the case of the unit disk $\mathbb D$, in \cite{BCD3} it has been proved the following result which gives a description of semigroups-type evolution families:

\begin{theorem}
Let $G(z,t)$ be a Herglotz vector field of order $d$ in $\mathbb D$ and let $(\varphi_{s,t})$ be its associated evolution family. The following are equivalent:
\begin{enumerate}
  \item there exists a function $g\in L^d_{loc}([0,+\infty),\mathbb C)$ and an infinitesimal generator $H$ such that $G(z,t)=g(t)H(z)$ for all $z\in \mathbb D$ and almost all $t\geq 0$,
  \item $\varphi_{s,t}\circ \varphi_{u,v}=\varphi_{u,v}\circ \varphi_{s,t}$ for all $0\leq s\leq t$ and $0\leq u\leq v$.
\end{enumerate}
\end{theorem}

In order to end up the picture started with the classical L\"owner theory, we should put in the frame also the L\"owner chains.
The general notion of a L\"owner chain has been given\footnote{See
also~\cite{AbstractLoew} for an extension of this notion to complex manifolds and with a complete different approach even in the unit disk.
The construction of a L\"owner chain associated with a given evolution family
proposed there differs essentially from the one we used in~\cite[Theorems 1.3
and~1.6]{SMP}.} in~\cite{SMP}.
\begin{definition}(\cite{SMP})~\label{D_SMP-LC}
A family $(f_t)_{t\ge0}$ of holomorphic maps of~$\mathbb D$ is called a {\it L\"owner
chain of order $d$} with $d\in [1,+\infty]$  if it satisfies the following conditions:
\begin{enumerate}
\item[LC1.] each function $f_t:\mathbb D\to\mathbb C$ is univalent,

\item[LC2.] $f_s(\mathbb D)\subset f_t(\mathbb D)$ whenever $0\leq s < t<+\infty,$

\item[LC3.] for any compact set $K\subset\mathbb{D}$ and any $T>0$ there exists a
non-negative function $k_{K,T}\in L^{d}([0,T],\mathbb{R})$ such that
\[
|f_s(z)-f_t(z)|\leq\int_{s}^{t}k_{K,T}(\xi)d\xi
\]
whenever $z\in K$ and $0\leq s\leq t\leq T$.
\end{enumerate}
\end{definition}

This definition of (generalized) L\"owner chains matches the abstract notion of
evolution family introduced in~\cite{BCM1}. In particular the following
statement holds.

\begin{theorem}{\rm (\cite[Theorem~1.3]{SMP})}\label{TH_SMP-EF-LC}
For any L\"owner chain $(f_t)$ of order $d\in[1,+\infty]$, if we define
$$
\varphi_{s,t}:= f_t^{-1}\circ f_s \quad \text{whenever }0\le s\le t,
$$
then $(\varphi_{s,t})$ is an evolution family of the same order~$d$.
Conversely, for any evolution family $(\varphi_{s,t})$ of
order~$d\in[1,+\infty]$, there exists a L\"owner chain $(f_t)$ of the same
order~$d$ such that
\begin{equation*}
 f_t\circ\varphi_{s,t}=f_s\quad  \text{whenever }0\le s\le t.
\end{equation*}
\end{theorem}
In the situation of this theorem we say that the L\"owner chain~$(f_t)$ and the
evolution family~$(\varphi_{s,t})$ are {\it associated with} each other. It was
proved in~\cite{SMP} that given an evolution family~$(\varphi_{s,t})$, an
associated L\"owner chain~$(f_t)$ is unique up to conformal maps
of~$\cup_{t\ge0}f_t(\mathbb D)$. Thus there are essentially one two different types of L\"owner chains: those such that $\cup_{t\ge0}f_t(\mathbb D)=\mathbb C$ and those such that $\cup_{t\ge0}f_t(\mathbb D)$ is a simply connected domain different from $\mathbb C$ (see \cite{SMP} for a characterization in terms of the evolution family associated with).

Thus in the framework of the approach described above the essence of L\"owner
Theory is represented by the essentially 1-to-1 correspondence among L\"owner
chains, evolution families, and Herglotz vector fields.

Once the previous correspondences are established, given a
L\"owner chain $(f_t)$ of order $d$, the general L\"owner-Kufarev PDE
\begin{equation}\label{gen-Low-Kuf}
\frac{\partial f_t(z)}{\partial t}=-G(z,t)\frac{\partial f_t(z)}{\partial z}.
\end{equation}
 follows by differentiating the
structural equation \begin{equation}\label{gen-Low}
\frac{dz}{dt}=G(z,t), \qquad z(0)=z.
\end{equation}
 Conversely,
given a Herglotz vector field $G(z,t)$ of order $d$, one can
build the associated L\"owner chain (of the same order $d$),
solving \eqref{gen-Low-Kuf} by means of the associated
evolution family.

The Berkson-Porta decomposition of a Herglotz
vector field $G(z,t)$ also gives information on the dynamics of the
associated evolution family. For instance, when
$\tau(t)\equiv \tau\in \mathbb D$, the point $\tau$ is a (common)
fixed point of $(\varphi_{s,t})$ for all $0\leq s\leq t<+\infty$.
Moreover, it can be proved that, in such a case, there exists a
unique locally absolutely continuous function
$\lambda\colon[0,+\infty )\rightarrow\mathbb{C}$ with $\lambda'\in
L_{loc}^{d}([0,+\infty),\mathbb{C})$, $\lambda(0)=0$ and $\re
\lambda(t)\geq\re \lambda(s)\geq0$ for all $0\leq s\leq
t<+\infty$  such that for all $s\leq t$
\[
\varphi_{s,t}^{\prime}(\tau)=\exp(\lambda (s)-\lambda(t)).
\]
A similar characterisation holds when $\tau(t)\equiv \tau\in\partial
\mathbb D$ (see \cite{BCD}).

Let us summarize this section. In the next scheme we show the main three notions of L\"owner theory we are dealing with in this paper and the relationship between them:

\begin{center}
{\color{red}\fbox{{\color{blue}L\"owner chains $(f_t)$}}}
\hbox{{\Large $\overset{\varphi_{s,t}=f_t^{-1}\circ f_s}{\longleftarrow
\hspace{-0.25cm}-\hspace{-0.25cm}-\hspace{-0.25cm}-\hspace{-0.25cm}-\hspace{-0.25cm}\longrightarrow}$}}
{\color{red}\fbox{{\color{blue}Evolution families $(\varphi_{s,t})$}}}

\hspace*{-0.6cm}
\mbox{
{ \begin{minipage}{0.3\textwidth}
L\"owner-Kufarev PDE \\ \\
$\frac{\partial f_t(z)}{\partial t}=-G(z,t)\frac{\partial f_t(z)}{\partial z}$
\end{minipage}
}}
$
\begin{array}{ccc}
\lower.1cm\hbox{$\uparrow$} & \hspace*{3.5cm}& \lower.1cm\hbox{$\uparrow$}\\
\rule[-1cm]{0.1mm}{2cm}& &\rule[-1cm]{0.1mm}{2cm}\\
\raise.1cm\hbox{$\downarrow$} & &\raise.1cm\hbox{$\downarrow$}
\end{array}$
\mbox{
\begin{minipage}{0.3\textwidth}
L\"owner-Kufarev ODE \\ \\
$\frac{dw}{dt}=G(w,t), \, w(s)=z$ \\ \\
$\varphi_{s,t}(z)=w(t)$
\end{minipage}}

{\color{red}\fbox{{\color{blue}Herglotz vector fields $G(w,t)=(w-\tau(t))(\overline{\tau(t)}w-1)p(w,t)$}}}
\end{center}
where $p$ is a Herglotz function.

\section{Extensions to multiply connected domains}

Y{\^u}saku Komatu (1914--2004), in 1942 \cite{Komatu}, was the first to generalise L\"owner's parametric representation to univalent holomorphic functions defined in a circular annulus and with images in the exterior of a disc. Later, Goluzin \cite{GoluzinM} gave a much simpler way to establish Komatu's results. With the same techniques, Li \cite{Li} considered a slightly different case, when the image of the annulus is the complex plane with two slits (ending at infinity and at the origin, respectively). See also \cite{Komatu1950}, \cite{LebedevRing}, \cite{Gutljanski_diss}. The monograph~\cite{Aleksandrov} contains a self-contained detailed
account on the Parametric Representation in the multiply connected setting.

Another way of adapting L\"owner's method to multiply connected domains was developed by Kufarev and Kuvaev \cite{Kufarev-Kuvaev}. They obtained a differential equation satisfied by automorphic functions realizing conformal covering mappings of the unit disc onto multiply connected domains with a gradually erased slit. This differential approach has been also followed by Tsai \cite{Tsai}. Roughly speaking, these results can be considered as a version for multiply connected domains of the slit-radial L\"owner equation. In a similar way, Bauer and Friedrich have developed a slit-chordal theory for multiply connected domains. Moreover, they have even dealt with stochastic versions of both the radial and the chordal cases.
In this framework the situation is more subtle than in the simply connected case, because moduli spaces enter the picture \cite{BaFriedrich1}, \cite{BaFriedrich2}.

Following the guide of the new and general approach to L\"owner theory in the unit disk as described in Section 9, the second and the third author of this survey jointly with Pavel Gumenyuk have developed a global approach to L\"owner theory for double connected domains which give a uniform framework to previous works of Komatu, Goluzin, Li en Pir and Lebedev. More interestingly, this abstract theory shows some phenomena not considered before and also poses many new questions.

Besides the similarities between the general approach for simple and double connected domains, there are a number of significative differences both in concepts and results. For instance, in order to develop an interesting substantial theory for the double
connected case, instead of a static reference domain $D$, one has to consider a family of expanding annuli
$D_t=\mathbb A_{r(t)}:=\{z:r(t)<|z|<1\}$, where ${r:[0,+\infty)\to[0,1)}$ is non-increasing and
continuous. The first who noticed this fact, in a very special case, was already Komatu~\cite{Komatu}. Such a family $(D_t)$, it is usually called a canonical domain system of order $d$, whenever $-\pi/\log(r)$ belongs to $AC^d([0,+\infty),\R)$.

This really non-optional dynamic context forces to modify the definitions of (again) the three basic elements of the general theory: semicomplete vector fields, evolution families and L\"owner chains. Nevertheless, there is still a (essentially) one-to-one correspondence between these three notions.

As in the unit disk, (weak) vector fields are introduced in this picture bearing in mind Carath\'eodory's theory of ODEs.

\begin{definition}\label{D_WHVF}
Let  $d\in[1,+\infty]$.  A function $G:\mathcal D\to\C$ is said to be a {\it weak holomorphic
vector field} of order~$d$ in the domain
$$
\mathcal D:=\{(z,t):\,t\ge0,\,z\in D_t\},
$$
if it satisfies the following three conditions:
\begin{description}
\item[WHVF1.] For each $z\in\C$ the function $G(z,\cdot)$ is measurable in~$E_z:=\{t:(z,t)\in\mathcal
D\}$.
\item[WHVF2.] For each $t\in E$ the function $G(\cdot,t)$ is holomorphic in
$D_t$.

\item[WHVF3.] For each compact set $K\subset\mathcal D$ there exists a non-negative function
$$
{k_K\in L^d\big(pr_{\R}(K),\R\big)},\quad
pr_{\R}(K):=\{t\in
E:~\exists\,z\in\C\quad(z,t)\in K\},
$$
such that
$$
|G(z,t)|\le k_K(t),\quad\text{for all $(z,t)\in K$}.
$$%
\end{description}
\end{definition}

Given a weak holomorphic vector field~$G$ in~$\mathcal D$ and an initial condition $(z,s)\in\mathcal D$, it is possible to consider the initial value problem,
\begin{equation}\label{EQ_CarODE-IVP}
\dot w=G(w,t),\quad w(s)=z.
\end{equation}
A solution to this problem  is any continuous function $w:J\to\C$ such that $J\subset E$ is an interval,
$s\in J$, $(w(t),t)\in\mathcal D$ for all $t\in J$ and
\begin{equation}\label{EQ_opL}
w(t)=z+\int_{s}^t G(w(\xi),\xi)\,d\xi,\quad t\in J.
\end{equation}
When these kind of problems have solutions well-defined globally to the right for any initial condition, the vector field $G$ is called semicomplete.

Putting together the main properties of the flows generated by semicomplete weak vector fields, we arrive to the concept of evolution families for the doubly connected setting.

\begin{definition}\label{def-ev}
A family $(\varphi_{s,t})_{0\leq s\leq t<+\infty}$ of holomorphic mappings $\varphi_{s,t}:D_s\to D_t$ is
said to be an {\it evolution family of order $d$ over~$(D_t)$} (in short, an {\it $L^d$-evolution
family}) if the following conditions are satisfied:
\begin{description}
\item[EF1.] $\varphi_{s,s}=\mathsf{id}_{D_s},$

\item[EF2.] $\varphi_{s,t}=\varphi_{u,t}\circ\varphi_{s,u}$ whenever $0\le
s\le u\le t<+\infty,$

\item[EF3.] for any closed interval $I:=[S,T]\subset[0,+\infty)$ and any $z\in D_S$ there exists a
non-negative function ${k_{z,I}\in
L^{d}\big([S,T],\mathbb{R}\big)}$ such that
\[
|\varphi_{s,u}(z)-\varphi_{s,t}(z)|\leq\int_{u}^{t}k_{z,I}(\xi)d\xi
\]
whenever $S\leq s\leq u\leq t\leq T.$
\end{description}
\end{definition}

As expected, there is a one-to-one
correspondence between evolution families over canonical domain systems and
semicomplete weak holomorphic vector fields, analogous to the correspondence
between evolution families and Herglotz vector fields in the unit
disk.

\begin{theorem}\label{TH_EF<->sWHVF}{\rm \cite[Theorem 5.1]{SMPAnI}} The following two assertions hold:
\begin{description}
\item[(A)] For any $L^d$-evolution family~$(\varphi_{s,t})$ over the canonical
domain system~$(D_t)$ there exists an essentially unique semicomplete weak
holomorphic vector field $G:\mathcal D\to\C$ of order~$d$ and a null-set
$N\subset [0,+\infty)$ such that for all $s\ge0$ the following statements hold:
\end{description}
\begin{description}
\item[(i)] the mapping $[s,+\infty)\ni t\mapsto \varphi_{s,t}\in\Hol(D_s,\C)$ is
differentiable for all $t\in[s,+\infty)\setminus N$;
\item[(ii)] $d\varphi_{s,t}/dt=G(\cdot,t)\circ\varphi_{s,t}$ for all $t\in[s,+\infty)\setminus N$.
\end{description}\vskip3mm

\begin{description}
\item[(B)] For any semicomplete weak holomorphic vector field ${G:\mathcal D\to\C}$
of order~$d$ the formula $\varphi_{s,t}(z):=w^*_s(z,t)$, $t\ge s\ge0$, $z\in
D_s$, where $w_s^*(z,\cdot)$ is the unique non-extendable solution to the
initial value problem
\begin{equation}\label{EQ_genGolKom}%
 \dot w=G(w,t),\quad w(s)=z,
\end{equation}%
defines an $L^d$-evolution family over the canonical domain system~$(D_t)$.
\end{description}
\end{theorem}

With the corresponding perspective, there is also a true version of the non-auto-nomous Berkson-Porta description of Herglotz vector fields. Now, the role played by functions associated with Scharwz kernel is fulfiled by a natural class of functions associated with the Villat kernel (see \cite[Theorem 5.6]{SMPAnI}).

L\"owner chains in the double connected setting are introduced in a similar way as it was done in the case of the unit disc.

\begin{definition}\label{D_Loew_chain} A family $(f_t)_{t\ge0}$ of holomorphic functions
$f_t:D_t\to\C$ is called a {\it Loew-ner chain} of order $d$ over~$(D_t)$ if it
satisfies the following conditions:
\begin{enumerate}
\item[LC1.] each function $f_t:D_t\to\C$ is univalent,

\item[LC2.] $f_s(D_s)\subset f_t(D_t)$ whenever $0\leq s < t<+\infty,$

\item[LC3.] for any compact interval $I:=[S,T]\subset[0,+\infty)$ and any compact set $K\subset D_S$  there exists a
non-negative function $k_{K,I}\in L^{d}([S,T],\mathbb{R})$ such that
\[
|f_s(z)-f_t(z)|\leq\int_{s}^{t}k_{K,I}(\xi)d\xi
\]
for all $z\in K$ and all $(s,t)$ such that $S\leq s\leq t\leq T$.
\end{enumerate}
\end{definition}

The following theorem shows that every L\"owner chain generates an
evolution family of the same order.

\begin{theorem}{\rm \cite[Theorem 1.9]{SMPAnII}}\label{TH_LC->EF}
Let $(f_{t})$ be a L\"owner chain of order $d$ over a canonical
domain system $(D_{t})$ of order $d.$ If we define
\begin{equation}\label{EQ_LC->EF}
\varphi _{s,t}:=f_{t}^{-1}\circ f_{s},\qquad 0\leq s\leq t<\infty,
\end{equation}%
then $(\varphi _{s,t})$ is an evolution family of order $d$ over
$(D_{t}).$
\end{theorem}

An interesting consequence of this result is thatany L\"owner chain over a canonical system of
annuli satisfies a PDE driven by a semicomplete weak holomorphic
vector field. Moreover, the concrete formulation of this PDE clearly resembles the celebrated L\"owner-Kufarev PDE appearing in the simple connected case.

The corresponding converse of the above theorem is more subtle than the one shown in the simple connected setting. Indeed, for any evolution family
$(\varphi_{s,t})$ there exists a (essentially unique) L\"owner chain $(f_t)$ of the same
order such that $\eqref{EQ_LC->EF}$ holds but the selection of this fundamental chain is affected by the different conformal types of the domains $D_t$ as well as the behavior of their elements with the index $I(\gamma)$ (with respect to zero) of certain closed curves $\gamma\subset\D_t$.

\begin{theorem}{\rm \cite[Theorem 1.10]{SMPAnII}}\label{Th_from_EF_to_LC}
Let $(\varphi_{s,t})$ be an evolution family of order
$d\in[1,+\infty]$ over the canonical domain system $D_{t}:=\mathbb
A_{r(t)}$ with $r(t)>0$ (a non-degenerate system). Let $r_\infty:=\lim_{t\to+\infty} r(t)$. Then
there exists a L\"owner chain $(f_t)$ of order $d$ over $(D_t)$ such
that
\begin{enumerate}
       \item $f_s=f_t\circ\varphi_{s,t}$ for all $0\leq s\leq t<+\infty$, i.e. $(f_t)$ is associated with $(\varphi_{s,t})$;
       \item $I(f_t\circ\gamma)=I(\gamma)$ for any closed curve $\gamma\subset D_{t}$ and any $t\ge0$;
       \item If $0<r_\infty<1$, then $\cup_{t\in [0,+\infty)} f_t(D_t)=\mathbb A_{r_\infty}$;
       \item If $r_\infty =0$, then $\cup_{t\in [0,+\infty)} f_t(D_t)$ is either $\D^*$, $\C\setminus \overline{\D}$, or $\C^*$.
     \end{enumerate}
If $(g_t)$ is another L\"owner chain over $(D_t)$ associated with
$(\varphi_{s,t})$, then there is a biholomorphism $F:\cup_{t\in
[0,+\infty)} g_t(D_t)\to \cup_{t\in [0,+\infty)} f_t(D_t)$ such that
$f_t=F\circ g_t$ for all $t\ge0$.
\end{theorem}

In general, a L\"owner chain associated with a given evolution family is not
unique. We call a L\"owner chain $(f_t)$ to be {\it standard} if it satisfies
conditions (2)\,--\,(4) from Theorem~\ref{Th_from_EF_to_LC}. It follows from
this theorem that the standard L\"owner chain $(f_t)$ associated with a given
evolution family, is defined uniquely up to a rotation (and scaling if
$\cup_{t\in [0,+\infty)} f_t(D_t)=\C^*$).

Furthermore, combining
Theorems~\ref{TH_LC->EF} and~\ref{Th_from_EF_to_LC} one can easily conclude
that for any L\"owner chain~$(g_t)$ of order~$d$ over a canonical domain
system~$(D_t)$ there exists a biholomorphism $F:\cup_{t\in [0,+\infty)}
g_t(D_t)\to L[(g_t)]$, where $L[(g_t)]$ is either $\D^*$, $\C\setminus
\overline{\D}$, $\C^*$, or $\mathbb A_{\rho}$ for some $\rho>0$, and such that the
formula $f_t=F\circ g_t$, $t\ge0$, defines a standard L\"owner chain of
order~$d$ over the canonical domain system~$(D_t)$. Hence, the conformal type of any L\"owner chain $\cup_{t\in [0,+\infty)}
g_t(D_t)$ can be identified with $\D^*$, $\C\setminus
\overline{\D}$, $\C^*$, or $\mathbb A_{\rho}$ for some $\rho>0$. Moreover, it is natural (well-defined) to say that
the {(conformal) type} of an evolution family $(\varphi_{s,t})$
is the conformal type of any L\"owner chain associated with it.

The following statement characterizes the conformal type of a L\"owner chain via dynamical properties of two associated evolution families over $(D_t)$. One of them is the usual one $(\varphi_{s,t})$ and the other one is defined as follows: for each $s\ge0$ and
$t\ge s$,
$$
\tilde\varphi_{s,t}(z):=r(t)/\varphi_{s,t}(r(s)/z).
$$
It is worth mentioning that at least one of the
families $(\varphi_{0,t})$ and $(\tilde\varphi_{0,t})$ converges
to~$0$ as $t\to+\infty$ provided $r_\infty=\lim_{t\to+\infty}r(t)=0$.

\begin{theorem}{\rm \cite[Theorem 1.13]{SMPAnII}}\label{TH_types_nondeg}
Let $\big((D_t),(\varphi_{s,t})\big)$ be a non-degenerate evolution
family and denote as before $r_\infty:=\lim_{t\to+\infty}r(t)$. In the above notation, the following statements hold:
\begin{itemize}
  \item[(i)] the conformal type of the evolution family $(\varphi_{s,t})$ is $\mathbb A_{\rho}$ for some $\rho>0$ if and only
  if $r_\infty>0$;
  \item[(ii)] the conformal type of the the evolution family $(\varphi_{s,t})$ is $\D^*$ if and only if $r_\infty=0$ and $\varphi_{0,t}$  does not converge
  to 0  as  $t\to+\infty$;
  \item[(iii)] the conformal type of the the evolution family $(\varphi_{s,t})$ is $\C\setminus \overline{\D}$ if and only if $r_\infty=0$ and $\tilde{\varphi}_{0,t}$  does not converge
  to 0  as  $t\to+\infty$;
  \item[(iv)] the conformal type of the the evolution family $(\varphi_{s,t})$ is $\C^*$ if and only if $r_\infty=0$  and both $\varphi_{0,t}\to0$ and $\tilde{\varphi}_{0,t}\to0$
  as $t\to+\infty$.
\end{itemize}
\end{theorem}

\section{Integrability}

In this section we plan reveal relations between contour dynamics tuned by the L\"owner-Kufarev equations and the Liouville (infinite-dimensional) integrability, which was exploited actively
since establishment of  the Korteweg-de Vries equation as an equation for spectral stability in the Sturm-Liouville problem, and construction of integrable hierarchies in a series of papers
by Gardner, Green, Kruskal, Miura, Zabusky, {\it et al.}, see e.g.,  \cite{Gardner, ZK}, and exact integrability results by Zakharov and Faddeev \cite{ZF}. In this section we mostly follow
our results in \cite{MarkinaVasiliev, MV2011}.

Recently, it has become clear that one-parameter expanding evolution families of simply
connected domains in the complex plane in some special models has been governed by infinite systems
of evolution parameters, conservation laws. This phenomenon reveals a bridge between a non-linear evolution of complex shapes emerged in physical problems, dissipative in most of the cases, and exactly solvable models. One of such processes is the Laplacian growth, in which the harmonic (Richardson's) moments are conserved under the evolution, see e.g., \cite{GustafssonVasiliev, Mineev}. The infinite number of evolution parameters reflects the infinite number of degrees of freedom of the system, and clearly suggests to apply field theory methods as a natural tool of study. The Virasoro algebra provides a structural background in most of field theories, and it is not surprising that it appears in soliton-like problems, e.g., KP, KdV or Toda hierarchies, see \cite{Faddeev, Gervais}.

Another group of models, in which the evolution is governed by an infinite number of parameters, can be observed in controllable dynamical systems, where the infinite number of degrees of freedom follows from the infinite number of driving terms. Surprisingly,
the same algebraic structural background appears again for this group. We develop this viewpoint here.

One of the general approaches to the homotopic evolution of shapes starting from a canonical shape, the unit disk in our case,  is given by the L\"owner-Kufarev theory. A shape evolution is described by a time-dependent conformal parametric map from the canonical domain onto the domain
bounded by the shape at any fixed instant. In fact, these one-parameter conformal maps satisfy the L\"owner-Kufarev differential equation \eqref{LKord}, or an infinite dimensional controllable system, for which
the infinite number of conservation laws is given by the {\it Virasoro generators} in their covariant form.

Recently, Friedrich and
Werner \cite{FriedrichWerner}, and independently Bauer and Bernard \cite{BB}, found relations between SLE (stochastic or Schramm-L\"owner evolution) and the highest weight representation of the Virasoro algebra. Moreover, Friedrich developed the Grassmannian approach to relate SLE with the highest weight representation of the Virasoro algebra in~\cite{Friedrich}.

All  above results encourage us to conclude that the {\it Virasoro algebra} is a common algebraic structural basis for these and possibly other types of contour dynamics and we present the development in this direction here. At the same time, the infinite number of conservation laws suggests
a relation with exactly solvable models.

The geometry underlying classical integrable systems is reflected in Sato's  and Segal-Wilson's constructions of the infinite dimensional {\it Grassmannian} $\Gr$. Based on the idea that the evolution of shapes in the plane is related to an evolution in a general universal space, the Segal-Wilson Grassmannian in our case, we provide
an embedding of the L\"owner-Kufarev evolution into a fiber bundle with the cotangent bundle over $\mathcal F_0$ as a base space, and with the smooth Grassmannian $\Gr_{\infty}$ as a typical fiber. Here $\mathcal F_0\subset \Sb$ denotes the space of all conformal embeddings $f$ of the unit disk into $\mathbb C$ normalized by $f(z)=z\left(1+\sum_{n=1}^{\infty}c_nz^n\right)$ smooth on the boundary $S^1$, and under the  {\it smooth Grassmannian} $\Gr_{\infty}$ we understand  a dense subspace $\Gr_{\infty}$ of $\Gr$ defined further on.

So our plan is as follows. We will
\begin{itemize}
\item Consider homotopy in the space of shapes starting from the unit disk: $\mathbb D\to \Omega(t)$ given by the L\"owner-Kufarev equation;
\item Give its Hamiltonian formulation;
\item Find conservation laws;
\item Explore their algebraic structure;
\item Embed L\"owner-Kufarev trajectories into a moduli space (Grassmannian);
\item Construct $\tau$-function, Baker-Akhiezer function and finally  KP hierarchy.
\end{itemize}
Finally we present a class of solutions to KP hierarchy, which are preserving their form along the L\"owner-Kufarev trajectories.

Let us start with two useful lemmas \cite{MarkinaVasiliev, MV2011}.

\begin{lemma}\label{CarClass}
Let the function $w(z,t)$ be a solution to the Cauchy problem~\eqref{LKord}.
 If the driving function $p(\cdot,t)$, being from the Carath\'eodory class for almost all $t\geq 0$, is  $C^{\infty}$ smooth in the closure $\hat{\mathbb D}$ of the unit disk $\mathbb D$ and summable with respect to $t$, then the boundaries of
the domains $\Omega(t)=w(\mathbb D,t)\subset \mathbb D$ are smooth for all $t$ and $w(\cdot,t)$ extended to $S^1$ is injective on $S^1$.
\end{lemma}

\begin{lemma}\label{lemmaX}
With the above notations let $f(z)\in \mathcal F_0$. Then there exists a function $p(\cdot,t)$ from the Carath\'eodory class for almost all $t\geq 0$, and $C^{\infty}$ smooth in  $\hat{\mathbb D}$, such that $f(z)=\lim_{t\to\infty}f(z,t)$ is the final point of the L\"owner-Kufarev trajectory with the driving term $p(z,t)$.
\end{lemma}

\subsection{Witt and Virasoro algebras}\label{WV}

The complex {\it Witt algebra} is the Lie algebra of holomorphic vector fields defined on $\mathbb C^*=\mathbb C\setminus\{0\}$ acting by derivation
over the ring of Laurent polynomials $\mathbb C[z,z^{-1}]$. It is spanned by the basis ${L}_n=z^{n+1}\frac{\partial}{\partial z}$, $n\in \mathbb Z$.
The Lie bracket of two basis vector fields is given by the commutator $[L_n,L_m]=(m-n)L_{n+m}$. Its central extension is the complex {\it Virasoro algebra} $\mathfrak{vir}_{\mathbb C}$ with the central element $c$ commuting with all $L_n$, $[L_n,c]=0$, and with the Virasoro commutation relation
\[
[L_n,L_m]=(m-n)L_{n+m}+\frac{c}{12}n(n^2-1)\delta_{n,-m}, \quad n,m\in\mathbb Z,
\]
where $c\in \mathbb C$ is the central charge denoted by the same character. These algebras play important role in conformal field theory. In order to construct their representations   one can use an analytic realization.

\subsection{Segal-Wilson Grassmannian}

Sato's (universial) Grassmannian appeared first in 1982 in~\cite{Sato}  as an infinite dimensional generalization of the classical finite dimensional Grassmannian manifolds and they are described as `the topological closure of the inductive limit of' a finite dimensional Grassmanian as the dimensions of the ambient vector space and its subspaces tend to infinity.

It turned out to be a very important
 infinite dimensional manifold being related to the representation theory of loop groups, integrable hierarchies, micrological analysis, conformal and quantum field theories, the second quantization of fermions, and to many other topics~\cite{DJKM83,MPPR,SegalWilson,Witten}.
In the Segal and Wilson
approach~\cite{SegalWilson} the infinite dimensional Grassmannian $\Gr(H)$ is taken over the separable Hilbert space $H$. The first systematic description of the  infinite dimensional Grassmannian can be found in~\cite{PrSe}.

We present here a general definition of the infinite dimensional smooth Grassmannian $\Gr_{\infty}(H)$.
As a separable Hilbert space  we take the space $L^2(S^1)$ and consider its dense subspace $H=C^{\infty}_{\|\cdot\|_2}(S^1)$  of smooth complex valued functions defined on the unit circle endowed with  $L^2(S^1)$ inner product $\langle f,g\rangle=\frac{1}{2\pi}\int\limits_{S^1}f\bar g\,dw$, $f,g\in H$. The orthonormal basis of $H$ is $\{z^k\}_{k\in\mathbb Z}=\{e^{ik \theta}\}_{k\in\mathbb Z}$, $e^{i\theta} \in S^1$.

Let us split all integers $\mathbb Z$ into two sets $\mathbb Z^+=\{0,1,2,3,\dots\}$ and $\mathbb Z^-=\{\dots,-3,-2,-1\}$, and
let us define a polarization by  $$H_+=\spn_H\{z^k,\ k\in\mathbb Z^+\}, \qquad H_-=\spn_H\{z^k,\ k\in \mathbb Z^-\}.$$
Here and further span is taken in the appropriate space indicated as a subscription. The Grassmanian is thought of as the set of closed linear subspaces $W$ of $H$, which are commensurable with $H_+$ in the sense that they have finite codimension in both $H_+$ and~$W$. This can be defined by means of the descriptions of the orthogonal projections of the subspace $W\subset H$ to $H_+$ and $H_-$.

\begin{definition}
The infinite dimensional smooth Grassmannian $\Gr_{\infty}(H)$ over  the space $H$  is the set of  subspaces $W$ of $H$, such that
\begin{itemize}
\item[1.]{the orthogonal projection $pr_+\colon W\to H_+$ is a Fredholm operator,}
\item[2.]{the orthogonal projection $pr_-\colon W\to H_-$ is a compact operator.}
\end{itemize}
\end{definition}

The requirement that $pr_+$ is Fredholm means that the kernel and cokernel of $pr_+$ are finite dimensional. More information about Fredholm operators the reader can find in~\cite{Doug}. It was proved in~\cite{PrSe}, that $\Gr_{\infty}(H)$ is a dense submanifold in a Hilbert manifold modeled over the space $\mathcal L_2(H_+,H_-)$ of Hilbert-Schmidt operators from $H_+$ to $H_-$, that  itself has the structure of a Hilbert space, see~\cite{ReS}.  Any $W\in\Gr_{\infty}(H)$ can be thought of as a graph $W_T$ of a Hilbert-Schmidt operator $T\colon W\to W^{\bot}$, and  points of a neighborhood $U_W$ of $W\in \Gr_{\infty}(H)$ are in one-to-one correspondence with operators from $\mathcal L_2(W,W^{\bot})$.

Let us denote by $\mathfrak S$ the set of all collections $\mathbb S\subset \mathbb Z$ of integers such that $\mathbb S\setminus \mathbb Z^+$ and $\mathbb Z^+\setminus \mathbb S$ are finite. Thus, any sequence $\mathbb S$ of integers is bounded from below and contains all positive numbers starting from some number. It is clear that the sets $H_{\mathbb S}=\spn_{H}\{z^k,\ k\in\mathbb S\}$ are elements of the Grassmanian $\Gr_{\infty}(H)$ and they are usually called {\it special points}. The collection of neighborhoods $\{U_{\mathbb S}\}_{\mathbb S\in \mathfrak S},$
$$
U_{\mathbb S}=\{W\ \mid\ \text{there is an orthogonal projection}\ \pi\colon  W\to H_{\mathbb S}\ \text{that is an isomorphism}\}
$$
forms an open cover of $\Gr_{\infty}(H)$. The virtual cardinality of $\mathbb S$ defines the {\it virtual dimension} (v.d.) of $H_{\mathbb S}$, namely:
\begin{equation}\label{virtdim}
\virtcard(\mathbb S)  = \virtdim(H_{\mathbb S})=\dim(\mathbb N\setminus\mathbb S)-\dim(\mathbb S\setminus\mathbb N)
 = \ind(pr_+).
\end{equation}
The expression $\ind(pr_+)=\dim\kernel(pr_+)-\dim\cokernel(pr_-)$ is called the index of the Fredholm operator~$pr_+$. According to their virtual dimensions the points of $\Gr_{\infty}(H)$ belong to different components of connectivity. The  Grassmannian is the disjoint union of connected components parametrized by their virtual dimensions.

\subsection{Hamiltonian formalism}

Let the driving term $p(z,t)$ in the L\"owner-Kufarev ODE~\eqref{LKord} be from the Carath\'eodory class for almost all $t\geq 0$,   $C^{\infty}$-smooth in  $\hat{\mathbb D}$, and summable with respect to $t$ as in Lemma~\ref{CarClass}.
 Then the domains $\Omega(t)=f(\mathbb D,t)=e^t w(\mathbb D,t)$  have  smooth boundaries $\partial
 \Omega(t)$ and the function $f$ is injective on $S^1$, i.e.; $f\in \mathcal F_0$. So the L\"owner-Kufarev equation can be extended to the
 closed unit disk $\hat{\mathbb D}=\mathbb D\cup S^1$.

Let us consider the sections ${\psi}$ of $T^*\mathcal F_0\otimes\mathbb C$, that are from the class $C^{\infty}_{\|\cdot\|_2}$ of smooth complex-valued functions $S^1\to\mathbb C$ endowed with $L^2$ norm,
\[
\psi(z)=\sum\limits_{k\in\mathbb Z}\psi_kz^{k-1}, \quad |z|=1.
\]
We also introduce the space of observables on  $T^*\mathcal F_0\otimes\mathbb C$, given by integral functionals
\[
\mathcal R(f,\bar\psi, t)=\frac{1}{2\pi}\int_{z\in S^1}r(f(z),\bar\psi(z), t)\frac{dz}{iz},
\]
where the function $r(\xi,\eta,t)$ is smooth in variables $\xi,\eta$ and measurable in $t$.

We define a special observable, the time-dependent pseudo-Hamiltonian $\mathcal H$,  by
 \begin{equation}\label{Ham3}
 \mathcal H(f,\bar\psi,p,t)=\frac{1}{2\pi}\int_{z\in S^1}\bar{z}^2f(z,t)(1-p(e^{-t}f(z,t),t))\bar \psi(z,t)\frac{dz}{iz},
 \end{equation}
 with the driving function (control) $p(z,t)$ satisfying the above properties.
The Poisson structure on the space of observables is given by the canonical brackets
\[
\{\mathcal R_1, \mathcal R_2\}=2\pi \int_{z\in S^1}z^2\left(\frac{\delta \mathcal R_1}{\delta f}\frac{\delta \mathcal R_2}{\delta \bar \psi}-\frac{\delta \mathcal R_1}{\delta \bar \psi}\frac{\delta \mathcal R_2}{\delta f}\right)\frac{dz}{iz},
\]
where $\frac{\delta}{\delta f}$ and $\frac{\delta}{\delta \overline{\psi}}$ are the variational derivatives, $\frac{\delta}{\delta f}\mathcal R=\frac{1}{2\pi}\frac{\partial}{\partial f}r$, $\frac{\delta}{\delta \overline{\psi}}\mathcal R=\frac{1}{2\pi}\frac{\partial}{\partial \overline{\psi}}r$.

Representing the coefficients $c_n$ and $\bar\psi_m$ of $f$ and $\bar{\psi}$ as integral functionals
\[
c_n=\frac{1}{2\pi}\int_{z\in S^1}\bar{z}^{n+1}f(z,t)\frac{dz}{iz},\quad \bar\psi_m= \frac{1}{2\pi}\int_{z\in S^1}{z}^{m-1}\bar\psi(z,t)\frac{dz}{iz},
\]
$n\in \mathbb N$, $m\in\mathbb Z$, we obtain
$\{c_n, \bar\psi_m\}=\delta_{n,m}$, $\{c_n, c_k\}=0$, and $\{\bar\psi_l, \bar\psi_m\}=0$, where $n, k\in \mathbb N$, $l, m\in\mathbb Z$.

The infinite-dimensional Hamiltonian system is written as
\begin{equation}\label{sys1}
\frac{d c_k}{dt}=\{c_k, \mathcal H\},\quad c_k(0)=0,
\end{equation}
\begin{equation}\label{sys2}
\frac{d\bar \psi_k}{dt}=\{\bar{\psi_k}, \mathcal H\}, \quad \psi_k(0)=\xi_k,
\end{equation}
where $k\in \mathbb Z$ and $c_0=c_{-1}=c_{-2}=\dots=0$, or equivalently, multiplying by corresponding  powers of $z$ and summing up,
\begin{equation}\label{sys10}
\frac{d f(z,t)}{dt}=f(1-p(e^{-t}f,t))=2\pi \frac{\delta  \mathcal H}{\delta \overline{\psi}}z^2=\{f, \mathcal H\},\quad f(z,0)\equiv z,
\end{equation}
\begin{equation}\label{sys20}
\frac{d\bar \psi}{dt}=-(1-p(e^{-t}f,t)-e^{-t}fp'(e^{-t}f,t))\bar
\psi=-2\pi \frac{\delta \mathcal H}{\delta f}z^2=\{\bar\psi, \mathcal H\},
\end{equation}
where $\psi(z,0)=\xi(z)=\sum_{k\in\mathbb Z}\xi_kz^{k-1}$
 and $z\in S^1$. So the phase coordinates $(f,\bar{\psi})$ play the role of the canonical Hamiltonian pair and the coefficients $\xi_k$ are free parameters.  Observe that the equation~\eqref{sys10} is the L\"owner-Kufarev equation~\eqref{LKord} for the function $f=e^{t}w$.

Let us set up the {\it generating function} $\mathcal{G}(z)=\sum_{k\in\mathbb Z}\mathcal{G}_kz^{k-1}$, such that  $$\bar{\mathcal G}(z):=f'(z,t) \bar{\psi}(z,t).$$ Consider the `non-positive' $(\bar{\mathcal G}(z))_{\leq 0}$ and `positive' $(\bar{\mathcal G}(z))_{> 0}$ parts of the Laurent series for $\bar{\mathcal G}(z)$:
\begin{eqnarray*}
(\bar{\mathcal G}(z))_{\leq 0} &= &(\bar{\psi}_1+2 c_1\bar{\psi}_2+3 c_2\bar{\psi}_3+\dots)+(\bar{\psi}_2+2 c_1\bar{\psi}_3+\dots)z^{-1}+\dots \\
&= &\sum\limits_{k=0}^{\infty}\bar{\mathcal G}_{k+1}z^{-k}.
\end{eqnarray*}
\begin{eqnarray*}
(\bar{\mathcal G}(z))_{> 0} &= &(\bar{\psi}_0+2  c_1\bar{\psi}_1+3 c_2\bar{\psi}_2+\dots)z+(\bar{\psi}_{-1}+2c_1\bar{\psi}_0+3c_2\bar{\psi}_1\dots)z^{2}+\dots \\
& = & \sum\limits_{k=1}^{\infty}\bar{\mathcal G}_{-k+1}z^{k}.
\end{eqnarray*}

\begin{proposition}\label{timeindep}
Let the driving term $p(z,t)$ in the L\"owner-Kufarev ODE be from the Carath\'eodory class for almost all $t\geq 0$,   $C^{\infty}$-smooth in $\hat{\mathbb D}$, and summable with respect to~$t$.
The functions $\mathcal G(z)$,  $(\mathcal G(z))_{< 0}$, $(\mathcal G(z))_{\geq 0}$, and all coefficients $\mathcal G_n$ are time-independent for all $z\in S^1$.
\end{proposition}

\begin{proof}
It is sufficient to check the equality $\dot{\bar{\mathcal G}}=\{\bar{\mathcal G}, {\mathcal H}\}=0$ for the function $\mathcal G$, and then, the same holds for the coefficients of the Laurent series for $\mathcal G$.
\end{proof}

\begin{proposition}
The conjugates  $\bar{\mathcal G}_k$, $k=1,2,\ldots$, to the coefficients of the generating function satisfy the Witt commutation relation $\{\bar{\mathcal G}_m,\bar{\mathcal G}_n\}=(n-m)\bar{\mathcal G}_{n+m}$ for $n,m\geq 1$, with respect to our Poisson structure.
\end{proposition}
\noindent

The isomorphism $\iota: \, \bar\psi_k\to \partial_k=\frac{\partial}{\partial c_k}$, $k>0$, is a Lie algebra isomorphism $$(T^{{*}^{(0,1)}}_f\mathcal F_0, \{\ ,\,\})\to (T^{(1,0)}_f\mathcal F_0, [\ ,\,]).$$ It makes a correspondence between the conjugates $\bar{\mathcal G}_n$ of the
coefficients $\mathcal G_n$ of $(\mathcal G(z))_{\geq 0}$ at the point $(f,\bar{\psi})$ and the Kirillov vectors  $L_n[f]=\partial_n+\sum\limits_{k=1}^{\infty}(k+1)c_{k}\partial_{n+k}
$, $n\in\mathbb N$, see~\cite{KYu}. Both satisfy the Witt commutation relations
\[
[L_n,L_m]=(m-n)L_{n+m}.
\]

\subsection{Curves in Grassmannian}

Let us recall, that the underlying  space for the universal smooth Grassmannian $\Gr_{\infty}(H)$ is $H= C^{\infty}_{\|\cdot\|_2}(S^1)$ with the canonical $L^2$ inner product of functions defined on the unit circle. Its natural polarization
  \begin{eqnarray*}
H_+  = \spn_{H}\{1, z, z^2, z^3, \dots \}, \qquad
H_-= \spn_{H}\{z^{-1}, z^{-2}, \dots \},
 \end{eqnarray*}
was introduced before.
The pseudo-Hamiltonian $\mathcal H(f,\bar\psi,t)$ is defined for an arbitrary $\psi\in L^{2}(S^1)$, but we consider only smooth solutions of the Hamiltonian system, therefore, $\psi \in H$. We identify this  space with the dense subspace of $T^*_f\mathcal F_0\otimes\mathbb C$, $f\in\mathcal F_0$. The generating function $\mathcal G$ defines a linear map $\bar{\mathcal G}$ from the dense subspace of $T^*_f\mathcal F_0\otimes \mathbb C$ to $H$, which being written  in a block matrix form becomes
\begin{equation}\label{bloc}
\renewcommand{\arraystretch}{1.4}
\left(
\begin{array}{c}
\bar{\mathcal G}_{> 0}\\
\bar{\mathcal G}_{\leq 0}
\end{array} \right)=\left(
\begin{array}{cc}
C_{1,1} & C_{1,2} \\
0 & C_{1,1}
\end{array} \right) \left(
\begin{array}{c}
\bar\psi_{> 0}\\
\bar\psi_{\leq 0}
\end{array} \right),
\end{equation}
where
\[
\renewcommand{\arraystretch}{1.4}
\left(
\begin{array}{cc}
C_{1,1} & C_{1,2} \\
0 & C_{1,1}
\end{array} \right)=
 \left(\begin{array}{ccccc|ccccc}
\ddots&\ddots&\ddots&\ddots&\ddots &\ddots&\ddots & \ddots & \ddots & \ddots   \\
 \cdots & 0& {\bf 1} & 2c_1 & 3c_2 & 4c_3 & 5c_4 & 6c_5 & 7c_6 & \cdots \\
  \cdots &0 & 0 & {\bf 1} &2c_1 & 3c_2 & 4c_3 & 5c_4 & 6c_5&  \cdots\\
    \cdots &0 & 0 & 0 &{\bf 1} & 2c_1 & 3c_2 & 4c_3 &  5c_4& \cdots \\
 \hline
  \cdots & 0 & 0 & 0 & 0 & {\bf 1} & 2c_1 & 3c_2 & 4c_3 &  \cdots \\
\cdots &0 & 0 & 0 & 0 & 0 & {\bf 1} & 2c_1 & 3c_2 &   \cdots \\
 \cdots & 0 &0& 0 & 0 & 0 & 0 & {\bf 1} & 2c_1 &   \cdots \\
  \ddots & \ddots &\ddots &\ddots &\ddots&\ddots & \ddots & \ddots & \ddots & \ddots  \\
 \end{array}\right).
\]

\begin{proposition}  The operator $C_{1,1}\colon H_+\to H_+$ is invertible.
\end{proposition}

The generating function also defines a map ${\mathcal G}\colon T^{*}\mathcal F_0\otimes\mathbb C\to H$ by $$T^{*}\mathcal F_0\otimes\mathbb C\ni (f(z),\psi(z))\mapsto {\mathcal G}= \bar f'(z)\psi(z)\in H.$$ Observe that any solution $\big(f(z,t),\bar\psi(z,t)\big)$ of the Hamiltonian system is mapped into a single point of the space $H$, since all $\mathcal G_k$, $k\in\mathbb Z$ are time-independent  by Proposition~\ref{timeindep}.

Consider a bundle $\pi \colon\mathcal B\to T^{*}\mathcal F_0\otimes \mathbb C$ with a typical fiber isomorphic to $\Gr_{\infty}(H)$. We are aimed at construction of  a curve $\Gamma\colon [0,T]\to \mathcal B$ that is traced by the solutions to the Hamiltonian system, or in other words, by the L\"owner-Kufarev evolution. The curve $\Gamma$ will have the form $$\Gamma(t)=\Big(f(z,t),\psi(z,t),W_{T_n}(t)\Big)$$ in the local trivialization. Here $W_{T_n}$ is the graph of a finite rank operator $T_n\colon H_+\to H_-$, such that $W_{T_n}$ belongs to the connected component of $U_{H_+}$ of virtual dimension $0$.
In other words, we build an hierarchy of finite rank operators $T_{n}\colon H_+\to H_-$,  $n\in \mathbb Z^+$, whose graphs in the neighborhood $U_{H_+}$ of the point $H_+\in \Gr_{\infty}(H)$ are
 \[
T_{n}((\mathcal G(z))_{> 0})=  T_{n}({\mathcal G}_1,{\mathcal G}_2,\dots,{\mathcal G}_k,\dots )=\left\{
 \renewcommand{\arraystretch}{1.4}
 \begin{array}{l}
 { {G}_0}({\mathcal G}_1,{\mathcal G}_2,\dots,{\mathcal G}_k,\dots)\\
  { {G}_{-1}}({\mathcal G}_1,{\mathcal G}_2,\dots,{\mathcal G}_k,\dots)\\
 \dots \\
 { {G}_{-n+1}}({\mathcal G}_1,{\mathcal G}_2,\dots,{\mathcal G}_k,\dots),
 \end{array}
 \right.
 \]
 with ${G}_0z^{-1}+{G}_{-1}z^{-2}+ \ldots +{G}_{-n+1}z^{-n}\in H_-$. Let us denote by ${G}_k={\mathcal G}_k$, $k\in \mathbb N$. The elements ${G}_0, {G}_{-1},{G}_{-2}, \dots$ are constructed so that all $\{\bar{G}_k\}_{k=-n+1}^{\infty}$ satisfy the truncated Witt commutation relations
 \[
 \{\bar{G}_k,\bar{G}_l\}_n=
 \begin{cases}
 (l-k)\bar{G}_{k+l},\quad &\text{for $k+l\geq -n+1$},\\
 0,\quad &\text{otherwise},
 \end{cases}
 \]
and are related to  Kirilov's vector fields~\cite{KYu} under the isomorphism $\iota$. The projective limit as $n\leftarrow \infty$ recovers the whole Witt algebra and the Witt commutation relations.
Then the operators $T_n$ such that their conjugates are $\bar T_n=(\widetilde B^{(n)}+C^{(n)}_{2,1})\circ C_{1,1}^{-1}$, are operators from $H_+$ to $H_-$ of finite rank and their graphs $W_{T_n}=(\id+T_n)(H_+)$ are elements of the component of virtual dimension $0$ in $\Gr_{\infty}(H)$.
We can construct a basis $\{e_0,e_1,e_2,\dots\}$ in  $W_{T_n}$ as a set of Laurent polynomials defined by means of operators $T_n$ and $\bar C_{1,1}$ as a mapping
$$\{\psi_1,\psi_2,\dots\}\buildrel{\bar{C}_{1,1}}\over\longrightarrow \{G_1,G_2,\dots\}\buildrel{\id+T_n}\over\longrightarrow\{G_{-n+1},G_{-n+2},\dots, G_{0},G_1,G_2,\dots\},$$
of the canonical basis $\{1,0,0,\dots\}$, $\{0,1,0,\dots\}$, $\{0,0,1,\dots\}$,\dots

Let us  formulate the result as the following  main statement.

\begin{proposition}\label{graph}
The operator $T_n$ defines a graph $W_{T_n}=\spn\{e_0,e_1,e_2,\dots\}$ in the Grassmannian $\Gr_{\infty}$ of virtual dimension 0.
Given any $$\psi=\sum_{k=0}^{\infty}\psi_{k+1}z^k\in H_+\subset H,$$ the function
\[
{G}(z)=\sum_{k=-n}^{\infty}{G}_{k+1}z^k=\sum_{k=0}^{\infty} \psi_{k+1}e_k,
\]
is an element of $W_{T_n}$. \end{proposition}

\begin{proposition}
In the autonomous case of the Cauchy problem \eqref{LKord}, when the function $p(z,t)$ does not depend on $t$, the pseudo-Hamiltonian $\mathcal H$ plays the role of time-dependent energy and $\mathcal H(t)=\bar{G}_0(t)+const$, where $\bar{G}_0\big|_{t=0}=0$. The constant is defined as
$\sum_{n=1}^{\infty}p_k\bar{\psi}_k(0)$.
\end{proposition}

\begin{remark} The Virasoro generator $L_0$ plays the role of the energy functional in CFT. In the view of the isomorphism $\iota$, the observable $\bar{G}_0=\iota^{-1}(L_0)$ plays an analogous role.
\end{remark}

Thus, we constructed a countable family of curves $\Gamma_n\colon[0,T]\to\mathcal B$ in the trivial bundle $\mathcal B=T^{*}\mathcal F_0\otimes \mathbb C\,\times\,\Gr_{\infty}(H)$, such that the curve $\Gamma_n$ admits the form $\Gamma_n(t)=\Big(f(z,t),\psi(z,t),W_{T_n}(t)\Big)$, for $t\in[0,T]$ in the local trivialization. Here $\big(f(z,t),\bar\psi(z,t)\big)$ is the solution of the Hamiltonian system (\ref{sys1}--\ref{sys2}). Each operator $T_n(t)\colon H_+\to H_-$ that maps ${\mathcal G}_{>0}$ to $$\big({G}_0(t),{G}_{-1}(t),\ldots,{G}_{-n+1}(t)\big)$$ defined for any $t\in[0,T]$, $n=1,2,\ldots$, is of finite rank and its graph $W_{T_n}(t)$  is a point in $\Gr_{\infty}(H)$ for any $t$. The graphs $W_{T_n}$  belong to the connected component of the virtual dimension $0$ for every time $t\in[0,T]$ and for fixed~$n$.
The coordinates $(G_{-n+1},\ldots, G_{-2},G_{-1},G_0,{G}_1,{G}_2,\ldots)$ of a point in the graph $W_{T_n}$ considered as a function of $\psi$ are isomorphic to the Kirilov vector fields $$(L_{-n+1},\ldots,L_{-2},L_{-1},L_0,L_1,L_1,L_2,\ldots)$$ under the isomorphism $\iota$.

 \subsection{$\tau$-function}

Remind that any function $g$ holomorphic in the unit disc, non vanishing on the boundary and normalized by $g(0)=1$ defines the multiplication operator $g\varphi$, $\varphi(z)=\sum_{k\in\mathbb Z}\varphi_kz^k$, that can be written in the matrix form
\begin{equation}\label{multip}
\left(
\begin{array}{cc}
a & b \\
0 & d
\end{array} \right) \left(
\begin{array}{c}
\varphi_{\geq 0}\\
 \varphi_{< 0}
\end{array} \right).
\end{equation}
All these upper triangular matrices form a subgroup $GL_{res}^{+}$ of the group of automorphisms $GL_{res}$ of the Grassmannian $\Gr_{\infty}(H)$.

With  any function $g$ and any graph $W_{T_n}$ constructed in the previous section (which is transverse to $H_-$) we can relate the $\tau$-function $\tau_{W_{T_n}}(g)$ by the following formula
$$
\tau_{W_{T_n}}(g)=\det(1+a^{-1}bT_n),$$
where $a,b$ are the blocks in the multiplication operator generated by $g^{-1}$. If we write the function $g$ in the form $g(z)=\exp(\sum_{n=1}^{\infty}t_nz^n)=1+\sum_{k=1}^{\infty}S_k(\tb)z^k$, where the coefficients $S_k(\tb)$ are the homogeneous elementary Schur polynomials, then the coefficients $\tb=(t_1,t_2,\dots)$ are called generalized times.
 For any fixed $W_{T_n}$ we get an orbit in $\Gr_{\infty}(H)$ of curves $\Gamma_n$ constructed in the previous section under the action of the elements of the subgroup  $GL_{res}^{+}$ defined by the function $g$. On the other hand,  the $\tau$-function defines a section in the determinant bundle over $\Gr_{\infty}(H)$ for any fixed $f\in\mathcal F_0$ at each point of the curve $\Gamma_n$.

\subsection{Baker-Akhiezer function, KP flows, and KP equation}

Let us consider the component $\Gr^{0}$ of the Grassmannian $\Gr_{\infty}$  of virtual dimension $0$, and let $g$ be a holomorphic function in $\mathbb D$ considered as an element of $GL_{res}^{+}$ analogously to  the previous section. Then $g$ is an upper triangular matrix with 1s on the principal diagonal. Observe that $g(0)=1$ and $g$ does not vanish on $S^1$. Given a point $W\in \Gr^{0}$ let us define
a subset $\Gamma^+ \subset GL_{res}^{+}$ as $$\Gamma^+=\{g\in GL_{res}^{+}:\, g^{-1} W \text{\ is transverse to \ } H_-\}.$$ Then  there exists \cite{SegalWilson} a unique function $\Psi_W[g](z)$ defined on $S^1$, such that for each $g\in \Gamma^+$, the function $\Psi_W[g]$ is in $W$, and it admits the form
\[
\Psi_W[g](z)=g(z)\left(1+\sum_{k=1}^{\infty}\omega_k(g,W)\frac{1}{z^k}\right).
\]
The coefficients $\omega_k=\omega_k(g,W)$ depend both on $g\in \Gamma^{+}$ and on $W\in \Gr^{0}$, besides they are holomorphic on $\Gamma^+$ and extend to meromorphic functions on $GL_{res}^{+}$. The function $\Psi_W[g](z)$ is called the {\it Baker-Akhiezer function} of $W$.

 Henry Frederick Baker (1866--1956) was British mathematician known for his contribution in algebraic geometry, PDE, and Lie theory. Naum Ilyich Akhiezer (1901--1980) was a Soviet mathematician known for his contributions in approximation theory and the theory of differential and integral operators, mathematical physics and history of mathematics. His brother Alexander was known theoretical physicist.

The Baker-Akhiezer function  plays a crucial role in the definition of the KP (Kadomtsev-Petviashvili) hierarchy which we will define later. We are going to construct the Baker-Akhiezer function explicitly in our case.

Let $W=W_{T_n}$ be a point of $\Gr^0$ defined in Proposition~\ref{graph}. Take a function $g(z)=1+a_1 z+a_2z^2+\dots\in \Gamma^+$, and
let us write the corresponding bi-infinite series for the Baker-Akhiezer function $\Psi_W[g](z)$ explicitly as
\begin{eqnarray*}
\Psi_W[g](z)=\sum_{k\in\mathbb Z}\mathcal W_kz^k=(1 & + & a_1z+a_2z^2 + \dots)\left(1+\frac{\omega_1}{z}+\frac{\omega_2}{z^2}+\dots\right)=\\
= \dots &+& (a_2+a_3\omega_1+a_4\omega_2+a_5\omega_3+\dots)z^2\\
&+& (a_1+a_2\omega_1+a_3\omega_2+a_4\omega_3+\dots)z\\
&+& (1+a_1\omega_1+a_2\omega_2+a_3\omega_3+\dots)\\
&+& (\omega_1+a_1\omega_2+a_2\omega_3+\dots)\frac{1}{z}\\
&+& (\omega_2+a_1\omega_3+a_2\omega_4+\dots)\frac{1}{z^2}+\dots\\
\dots&+& (\omega_k+a_1\omega_{k+1}+a_2\omega_{k+2}+\dots)\frac{1}{z^k}+\dots\\
\end{eqnarray*}
The Baker-Akhiezer function for $g$ and $W_{T_n}$ must be of the form
\[
\Psi_{W_{T_n}}[g](z)=g(z)\left( 1+ \sum_{k=1}^{n}\omega_k(g)\frac{1}{z^k}\right)= \sum_{k=-n}^{\infty}\mathcal W_kz^k.
\]
For a fixed $n\in\mathbb N$ we truncate this bi-infinite series by putting $\omega_k=0$ for all $k>n$.
 In order to satisfy the definition of $W_{T_n}$, and determine the coefficients $\omega_1,\omega_2,\dots,\omega_n$, we must check
 that there exists a vector $\{\psi_1,\psi_2,\dots\}$, such that
 $\Psi_{W_{T_n}}[g](z)=\sum_{k=0}^{\infty} e_k\psi_{k+1}$.
 First we express $\psi_k$ as linear functions  $\psi_k=\psi_k(\omega_1,\omega_2,\dots,\omega_n)$ by
 \begin{equation}\label{psi1}
 (\psi_1,\psi_2,\psi_3,\dots)=\bar C_{1,1}^{-1}\Big(\mathcal W_0(\omega_1,\omega_2,\dots,\omega_n),\mathcal W_1(\omega_1,\omega_2,\dots,\omega_n),\dots\Big).
 \end{equation}
 Using Wronski formula we can write
 \begin{eqnarray*}
 \psi_1&=&\mathcal W_0-2\bar c_1\mathcal W_1-(3\bar c_2-4\bar c_1^2)\mathcal W_2-(4\bar c_3-12\bar c_2\bar c_1+8\bar c_1^3)\mathcal W_3+\dots,\\
\psi_2&=& \mathcal W_1-2\bar c_1\mathcal W_2-(3\bar c_2-4\bar c_1^2)\mathcal W_3-(4\bar c_3-12\bar c_2\bar c_1+8\bar c_1^3)\mathcal W_4+\dots,\\
\psi_3 &= &\mathcal W_2-2\bar c_1\mathcal W_3-(3\bar c_2-4\bar c_1^2)\mathcal W_4-(4\bar c_3-12\bar c_2\bar c_1+8\bar c_1^3)\mathcal W_5+\dots,\\
 \dots &\dots &\dots
 \end{eqnarray*}

 Next we define $\omega_1,\omega_2,\dots,\omega_n$ as functions of $g$ and $W_{T_n}$, or in other words, as functions
 of $a_k, \bar c_k$ by solving linear equations
 \begin{eqnarray*}
 \omega_1&=&\bar c_1\psi_1+2\bar c_2\psi_2+\dots k\bar c_k\psi_k+\dots,\\
 \omega_2&=& \sum_{k=1}^{\infty}\Big((k+2)\bar c_{k+1}-2\bar c_1\bar c_k\Big)\psi_k,\\
 \dots &\dots &\dots
 \end{eqnarray*}
 where $\psi_k$ are taken from \eqref{psi1}. The solution exists and unique because of the general fact of the existence of the Baker-Akhiezer function. It is quite difficult task in general, however, in the case $n=1$, it is possible to write the solution explicitly in the matrix form. If
 \[
 A=\left(
\begin{array}{c}
\dots \\
3\bar c_3\\
2\bar c_2\\
\bar c_1
\end{array} \right)^T\bar C_{1,1}^{-1} \left(\begin{array}{c}
\dots \\
 a_3\\
a_2\\
a_1
\end{array} \right),\quad
B=\left(
\begin{array}{c}
\dots \\
3\bar c_3\\
2\bar c_2\\
\bar c_1
\end{array} \right)^T\bar C_{1,1}^{-1} \left(\begin{array}{c}
\dots \\
 a_2\\
a_1\\
1
\end{array} \right).
 \]
 then $\omega_1=\frac{B}{1-A}$.

 In order to apply further theory of integrable systems we need
to change variables $a_n\to a_n(\tb)$, $n>0$, $\tb=\{t_1,t_2,\dots\}$ in the following way
\[
a_n=a_n(t_1,\dots, t_n)=
S_n(t_1,\dots, t_n),
\]
where $S_n$ is the $n$-th elementary homogeneous Schur polynomial
\[
1+\sum\limits_{k=1}^{\infty}S_k(\tb)z^k=\exp\left(\sum\limits_{k=1}^{\infty}t_kz^k\right)=e^{\xi(\tb,z)}.
\]
In particular,
\[
S_1=t_1,\quad S_2=\frac{t_1^2}{2}+t_2,\quad S_3=\frac{t_1^3}{6}+t_1t_2+t_3,
\]
\[
S_4=\frac{t_1^4}{24}+\frac{t_2^2}{2}+\frac{t_1^2t_2}{2}+t_1t_3+t_4.
\]
Then the Baker-Akhiezer function corresponding to the graph $W_{T_n}$ is written as
\[
\Psi_{W_{T_n}}[g](z)=\sum_{k=-n}^{\infty}\mathcal W_kz^k=e^{\xi(\tb, z)}\left(1+\sum_{k=1}^{n}\frac{\omega_k(\tb,W_{T_n})}{z^k}\right),
\]
and $\tb=\{t_1,t_2,\dots\}$ is called the vector of generalized times.
It is easy to see that
\[
\partial_{t_k}a_m=0,\quad\text{for all \ }m=1,2\dots,k-1,
\]
$\partial_{t_k}a_m=1$ and
\[
\partial_{t_k}a_m=a_{m-k},\quad\text{for all \ }m>k.
\]
In particular, $B=\partial_{t_1}A$. Let us denote $\partial:=\partial_{t_1}$. Then in the case $n=1$ we have
\begin{equation}\label{omega}
\omega_1=\frac{\partial A}{1-A}.
\end{equation}

Now we consider the associative algebra of pseudo-differential operators $\mathcal A=\sum_{k=-\infty}^na_k\partial^k$ over the space of smooth functions, where the derivation symbol $\partial$ satisfies the Leibniz rule and the integration symbol and its powers satisfy
the algebraic rules $\partial^{-1}\partial=\partial \partial^-1=1$ and $\partial^{-1}a$ is the operator $\partial^{-1}a=\sum_{k=0}^{\infty}(-1)^k(\partial^k a)\partial^{-k-1}$ (see; e.g., \cite{Dickey}).
The action of the operator $\partial^{m}$, $m\in\mathbb Z$, is well-defined over the function $e^{\xi(\tb,z)}$, where $\xi(\tb,z)=\sum_{k=1}^{\infty}t_kz^k$,  so that  the function $e^{\xi(\tb,z)}$ is the eigenfunction of the operator $\partial^m$ for any integer $m$, i.e., it satisfies the equation
\begin{equation}\label{eigen}
\partial^me^{\xi(\tb,z)}=z^m e^{\xi(\tb,z)},\quad m\in\mathbb Z,
\end{equation}
see; e.g., \cite{Babelon, Dickey}.
As usual, we identify  $\partial=\partial_{t_1}$, and $\partial^0=1$.

 Let us introduce the dressing operator $\Lambda=\phi\partial \phi^{-1}=\partial+\sum_{k=1}^{\infty}\lambda_k\partial^{-k}$, where $\phi$ is a pseudo-differential operator
$\phi=1+\sum_{k=1}^{\infty}w_k(\tb)\partial^{-k}$.
The operator $\Lambda$ is defined up to the multiplication on the right by a series $1+\sum_{k=1}^{\infty}b_k\partial^{-k}$ with constant coefficients $b_k$.
The $m$-th KP flow is defined by making use of the vector field
\[
\partial_m\phi:=-\Lambda^m_{<0}\phi,\quad \partial_m=\frac{\partial}{\partial t_m},
\]
and the flows commute. In the Lax form the KP flows are written as
\begin{equation}\label{Lax}
\partial_m\Lambda=[\Lambda^m_{\geq 0},\Lambda].
\end{equation}
If $m=1$, then $\partial \Lambda=[\partial, \Lambda]=\sum_{k=1}^{\infty}(\partial{\lambda_k})\partial^{-k}$, which justifies the identification
$\partial=\partial_{t_1}$.

Thus, the Baker-Akhiezer function $\Psi_{W_{T_n}}[g](z)$  admits the form $$\Psi_{W_{T_n}}[g](z)=\phi \exp(\xi(\tb,z)),$$ where $\phi$ is a pseudo-differential operator
$\phi=1+\sum_{k=1}^{n}\omega_k(\tb,W_{T_n})\partial^{-k}$.
The  function $\Psi_{W_{T_n}}[g](z)$ becomes the eigenfunction of the operator $\Lambda^m$, namely $\Lambda^m w=z^m w$, for $m\in\mathbb Z$. Besides, $\partial_m w= \Lambda^m_{>0}w$.  In the view of \eqref{eigen} we can write this function
as previously,
\[
\Psi_{W_{T_n}}[g](z)=(1+\sum_{k=1}^{n}\omega_k(\tb, W_{T_n})z^{-k})e^{\xi(\tb,z)}.
\]

\begin{proposition}
Let $n=1$, and let the Baker-Akhiezer function be of the form
\[
\Psi_{W_{T_n}}[g](z)=e^{\xi(\tb,z)}\left(1+\frac{\omega}{z}\right),
\]
where $\omega=\omega_1$ is given by the formula \eqref{omega}. Then
\[
\partial \omega=\frac{\partial^2 A}{1-A}+\left(\frac{\partial A}{1-A}\right)^2
\]
is a solution to the KP equation with  the Lax operator $L=\partial^2-2(\partial \omega)$.
\end{proposition}

\begin{remark}
Observe that the condition $\re p >0$ in $\mathbb D$ defines a L\"owner-Kufarev cone of trajectories in the class $\mathcal F_0$ of univalent smooth functions
which allowed us to construct L\"owner-Kufarev trajectories in the neighbourood  $U_{H_+}$, which is a cone in the Grassmannian $\Gr_{\infty}$. The form of solutions
to the KP hierarchy is preserved along the L\"owner-Kufarev trajectories. The solutions are parametrized by the initial conditions in the system  \eqref{sys2}.
\end{remark}

Of course, one can express the Baker-Akhiezer function  directly from the $\tau$-function by the Sato formula
\[
\Psi_{W_{T_n}}[g](z)=e^{\xi(\tb,z)}\frac{\tau_{W_{T_n}}(t_1-\frac{1}{z},t_2-\frac{1}{2z^2},t_3-\frac{1}{3z^3},\dots)}{\tau_{W_{T_n}}(t_1,t_2,t_3\dots)},
\]
or applying the {\it vertex operator} $V$ acting on the Fock space $\mathbb C[\tb]$ of homogeneous polynomials
\[
\Psi_{W_{T_n}}[g](z)=\frac{1}{\tau_{W_{T_n}}}V\tau_{W_{T_n}},
\]
where
\[
V=\exp\left(\sum\limits_{k=1}^{\infty}t_kz^k\right)\exp\left(-\sum\limits_{k=1}^{\infty}\frac{1}{k}\frac{\partial}{\partial t_k}z^{-k}\right).
\]
In the latter expression $\exp$ denotes the formal exponential series and $z$ is another formal variable that commutes with all Heisenberg operators $t_k$ and $\frac{\partial}{\partial t_k}$. Observe that
the exponents in $V$ do not commute and the product of exponentials is calculated by the Baker-Campbell-Hausdorff formula.
The operator $V$ is a vertex operator in which the coefficient $V_k$ in the expansion of $V$ is a well-defined linear operator on the space  $\mathbb C[\tb]$. The Lie algebra of operators spanned by $1,t_k,\frac{\partial}{\partial t_k}$, and $V_k$, is isomorphic to the affine Lie algebra $\hat{\mathfrak{s}\mathfrak{l}}(2)$.
The vertex operator  $V$ plays a central role in the highest weight representation of affine Kac-Moody algebras \cite{Kac, Moody}, and can be interpreted as the infinitesimal B\"acklund transformation for the Korteweg--de Vries equation~\cite{Date}.

The vertex operator $V$ recovers the Virasoro algebra in the following sense. Taken in two close points $z+\lambda/2$ and $z-\lambda/2$ the operator product
can be expanded in to the following formal Laurent-Fourier series
\[
:V(z+\frac{\lambda}{2})V(z-\frac{\lambda}{2}):=\sum\limits_{k\in \mathbb Z}W_k(z)\lambda^{k},
\]
where $:a b:$ stands for the bosonic normal ordering.
Then $W_2(z)=T(z)$ is the stress-energy tensor which we expand again as
\[
T(z)=\sum\limits_{n\in \mathbb Z}L_n(\tb)z^{n-2},
\]
where the operators $L_n$ are the Virasoro generators in the highest weight representation over  $\mathbb C[\tb]$. Observe that the generators $L_n$
span the full Virasoro algebra with central extension and with the central charge~1. This can be also interpreted as a quantization of the shape evolution.
We shall define a representation over the space of  conformal Fock space fields based on the Gaussian free field in the next section.

\begin{remark}
Let us remark that the relation of L\"owner equation to the integrable hierarchies of non-linear PDE was noticed by Gibbons and Tsarev \cite{GT}. They observed that
the chordal L\"owner equation plays an essential role in the classification of reductions of Benney's equations. Later Takebe, Teo and Zabrodin \cite{Takebe} showed that
the chordal and radial L\"owner PDE serve as consistency conditions for one-variable reductions of dispersionless KP and Toda hierarchies respectively.
\end{remark}

\begin{remark}
We also mention here relations between L\"owner half-plane  multi-slit equations and the estimates of spectral gaps of changing length  for the periodic Zakharov-Shabat  operators and
for Hamiltonians in KdV and non-linear Schr\"odinger equations elaborated in \cite{KK09, KK10, KK11}.
\end{remark}

\section{Stochastic L\"owner evolutions, Schramm, and connections to CFT}\label{SLE}

\subsection{SLE}

This section we dedicate to the stochastic counterpart of the L\"owner-Kufarev theory
first recalling that one of the last (but definitely not least) contributions to
this growing theory was the description by Oded Schramm  in
1999--2000 \cite{Schramm}, of the stochastic L\"owner evolution (SLE), also known as
the Schramm-L\"owner evolution.
The SLE  is a conformally invariant
stochastic process; more precisely, it is a family of random planar curves generated by solving L\"owner's differential
equation with the Brownian motion as a driving term.
This equation was studied and developed by Oded Schramm
together with Greg Lawler and Wendelin Werner in a series of
joint papers that led, among other things, to a proof of
Mandelbrot's conjecture about the Hausdorff dimension of the
Brownian frontier \cite{SLE}. This achievement was one of the reasons
Werner was awarded the Fields Medal in 2006. Sadly, Oded Schramm, born
10 December 1961 in Jerusalem,
died in a tragic hiking accident on 01 September 2008 while
climbing Guye Peak, north of Snoqualmie Pass in Washington.

The {\it (chordal) stochastic L\"owner
evolution} with parameter $k\geq 0$ (SLE$_k$) starting at a
point $x\in \R$ is the random family of maps $(g_t)$ obtained
from the chordal L\"owner equation \eqref{hydro-Low} by letting
$\xi(t)=\sqrt{k} B_t$, where $B_t$ is a standard one dimensional
Brownian motion such that $\sqrt{k}B_0=x$.
Namely, let us consider the equation
\begin{equation}\label{SL}
\frac{d g_t(z)}{d t}=\frac{2}{g_t(z)-\xi(t)}, \qquad g_0(z)=z,
\end{equation}
where $\xi(t)=\sqrt{k}B_t=\sqrt{k}B_t(\omega)$, where $B_t(\omega)$ is the
standard 1-dimensional Brownian motion defined on the standard filtered probability space $(\Omega,\mathcal G, (\mathcal G_t), P)$  of Brownian motion with the sample space $\omega\in \Omega$, and $t\in[0,\infty)$, $B_0=0$. The solution to \eqref{SL} exists as long as $g_t(z)-h(t)$ remains away from zero and we denote by $T_z$ the first time such that $\lim_{t\to T_z-0}(g_t(z)-h(t))=0$. The function $g_t$ satisfies the hydrodynamic normalization at infinity $g_t(z)=z+\frac{2t}{z}+\dots$ Let $K_t=\{z\in \hat{\mathbb H}\colon T_z\leq t\}$, and let $\mathbb H_t$ be its complement $\mathbb H\setminus K_t=\{z\in {\mathbb H}\colon T_z> t\}$. The set $K_t$ is called {\it SLE hull}. It is compact, $\mathbb H_t$ is a simply conneced domain and $g_t$ maps $\mathbb H_t$ onto $\mathbb H$. SLE hulls grow in time. The {\it trace} $\gamma_t$ is defined as $\lim_{z\to\xi(t)}g^{-1}_t(z)$, where the limit is taken in $\mathbb H$. The unbounded component of $\mathbb H\setminus \gamma_t$ is $\mathbb H_t$. the Hausdorff dimension of the SLE$_k$ trace is $\min(1+\frac{8}{k},2)$, see \cite{Beffara}.
Similarly, one can
define the radial stochastic L\"owner evolution. The terminology comes from the fact that the L\"owner trace tip tends almost surely to a boundary point in the chordal case ($\infty$ in the half-plane version) or to the origin in the disk version of the radial case.

Chordal SLE  enjoys two important properties: scaling invariance and the Markovian property. Namely,
\begin{itemize}
\item $g_t(z)$ and $\frac{1}{\lambda}g_{\lambda^2t}(\lambda z)$ are identically distributed;
\item $h_t(z)=g_t(z)-\xi(t)$ possesses the Markov property. Furthermore, $h_s\circ h_t^{-1}$ is distributed as $h_{s-t}$ for $s>t$.
\end{itemize}

The SLE$_k$ depends on the choice of $\omega$ and
it comes in several flavors depending on the type of Brownian
motion exploited. For example, it might start at a fixed point
or start at a uniformly distributed point, or might have a
built in drift and so on. The parameter $k$ controls the rate
of diffusion of the Brownian motion and the behaviour of
the SLE$_k$ critically depends on the value of~$k$.

The SLE$_2$ corresponds to the loop-erased random walk and the
uniform spanning tree. The SLE$_{8/3}$ is conjectured to be the
scaling limit of self-avoiding random walks. The SLE$_3$ is
proved \cite{ChelkakSmirnov} to be the limit of interfaces for the Ising model (another Fields Medal 2010 awarded to Stanislav Smirnov),
while the SLE$_4$ corresponds to the harmonic explorer and the
Gaussian free field. For all $0\leq k\leq 4$ SLE gives slit maps. The SLE$_6$ was used by
Lawler, Schramm and Werner in 2001 \cite{SLE} to prove the conjecture of
Mandelbrot (1982) that the boundary of planar Brownian motion
has fractal dimension $4/3$. Moreover, Smirnov \cite{Smirnov} proved the SLE$_6$ is the scaling limit of
critical site percolation on the triangular lattice. This result follows from his celebrated proof of Cardy's formula.
SLE$_8$ corresponds to the uniform spanning tree. For $4<k<8$ the curve intersects itself and every point is contained in a loop but the curve is not space-filling almost surely.
For $k\geq 8$ the curve is almost sure space-filling. This phase change is due to the Bessel process interpretation of SLE, see \eqref{Bessel}.
\begin{figure}[ht] \scalebox{0.7}{
\begin{pspicture}(2,1)(17,8)
\psline[linecolor=blue,linewidth=0.8mm](11,4)(13.6,4)
\psline[linecolor=blue,linewidth=0.8mm](16.6,4)(18,4)
\rput(15.3,7){$y$}\rput(18,4.3){$x$}
\rput(13.7,5){$\mathbb H$}\rput(15.2,4.2){0}
\psline[linewidth=0.15mm]{->}(15,1)(15,7)
\psline[linewidth=0.15mm]{->}(11,4)(18,4)
\rput(4.65,5.25){\includegraphics[height=1 in]{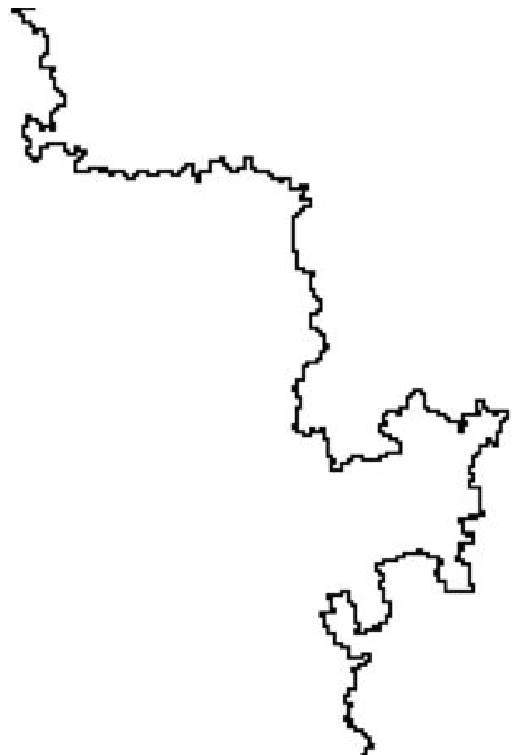}}
\psline[linecolor=blue,linewidth=0.8mm](2,4)(8,4)
\rput(5.3,7){$\eta$}\rput(8,4.3){$\xi$}
\rput(3,5){$\mathbb H_t$}\rput(5.2,4.2){0}
\psline[linewidth=0.15mm]{->}(5,1)(5,7)
\psline[linewidth=0.15mm]{->}(2,4)(8,4)
\rput(4.5,6.5){$\gamma_t$}
 \pscurve[linewidth=0.8mm,
linecolor=red]{->}(8,5)(9.5,5.5)(11,5)
\rput(9.5,5.9){$g_t(z)$}
\pscurve[linewidth=0.3mm,
linecolor=red](5.2,3.8)(8,2)(10,2)(12,2.5)
\psline[linecolor=red,linewidth=0.3mm]{->}(12,2.5)(13.5,3.8)
\psline[linecolor=red,linewidth=0.3mm]{->}(12,2.5)(16.5,3.8)
\rput(13.5,4.3){$g_t^-(0)$}
\rput(16.5,4.3){$g_t^+(0)$}
 \pscircle[fillstyle=solid,
fillcolor=black](13.5,4){.1}
 \pscircle[fillstyle=solid,
fillcolor=black](16.5,4){.1}
 \pscircle[fillstyle=solid,
fillcolor=black](5,4){.1}
\end{pspicture}}
\end{figure}

An invariant approach to SLE starts with probability measures on non-self-crossing random curves in a domain $\Omega$ connecting two given points $a,b\in\partial\Omega$
and satisfying the properties of 
\begin{itemize}
\item {\it Conformal invariance}. Consider a triple $(\Omega, a, b)$ and a conformal map $\phi$. If $\gamma$ is a trace SLE$_{k}(\Omega, a, b)$, then $\phi(\gamma)$ is  a trace SLE$_{k}(\phi(\Omega), \phi(a), \phi(b))$.;
\item {\it Domain Markov property}. Let $\{\mathcal F_t\}_{t\geq 0}$ be the filtration in $\mathcal F$ by $\{B_t\}_{t\geq 0}$ and let $g_t$ be a L\"owner flow generated by $\xi(t)=\sqrt{k}B_t$. Then the hulls $(g_t(K_{s+t}\cap \mathbb H_t)-\xi_t)_{s\geq 0}$ are also generated by SLE$_k$ and independent of the sigma-algebra $\mathcal F$.
\end{itemize}

The expository paper \cite{La} is perhaps the best option to start an exploration of this fascinating branch of mathematics. Nice papers \cite{Rohde05, Rohde11} give
up-to-date exposition of developments of SLE so we do not intend to survey SLE in details here. Rather, we are going to show relations with CFT and other related
stochastic variants of L\"owner (generalized) evolution.

Here let us also mention  the work by Carleson and  Makarov \cite{Carleson-Makarov} studying growth processes motivated by DLA
(diffusion-limited aggregation) via L\"owner's equations.

In this section we review the connections between conformal field theory (CFT)  and  Schramm-L\"owner evolution (SLE) following, e.g., \cite{BB, FriedrichWerner, MK}. Indeed, SLE, being, e.g., a continuous limit of CFT's archetypical Ising model at its critical point, gives an approach to CFT which emphasizes CFT's roots in statistical physics.

The equation \eqref{SL} is deterministic with a random entry and we solve it for every fixed $\omega$. The corresponding stochastic differential equation (SDE) in the   It\^{o} form
for the function $h_t(z)=g_t(z)-\xi(t)$ is
\begin{equation}\label{Bessel}
dh_t(z)=\frac{2}{h_t(z)}dt-\sqrt{k}dB_t,
\end{equation}
where $-h_t/\sqrt{k}$ represents a Bessel process (of order $(4+k)/k$).
For any holomorphic function $M(z)$ we have the It\^{o}  formula
\begin{equation}\label{L1}
(dM)(h_t)=-d\xi_t \mathcal L_{-1}M(h_t)+dt (\frac{k}{2} \mathcal L_{-1}^2-2 \mathcal L_{-2})M(h_t),
\end{equation}
where $ \mathcal L_n=-z^{n+1}\partial$.
From the form of the equation (\ref{Bessel}) one can see immediately that $h_t$ is a (time-homogeneous) diffusion, i.~e., a continuous strong Markov process.
The infinitesimal generator of $h_t$ is given by $A=(\frac{k}{2} \mathcal L_{-1}^2-2 \mathcal L_{-2})$ and this operator appears here for the first time. This differential operator makes it possible to reformulate many probabilistic questions about $h_t$ in the language of PDE theory. If we consider $h_t(z)$ with fixed $z$, then the equation \eqref{Bessel} for $h_t$ describes the motion of particles in the time-dependent field $v$ with
$dv=-d\xi_t \mathcal L_{-1}+dt A$. For instance, if we denote by $u_t(z)$ the mean function of $h_t(z)$
\[
 u_t(z) = \mathsf{E} h_t(z),
\]
then it follows from Kolmogorov's backward equation that $u_t$ satisfies
\[
\begin{cases}
 \frac{\partial u_t}{\partial t} = A u_t,\\
u_0(z)= z,
\end{cases}
 z\in \mathbb{\mathbb H}.
 \]
The kernel of the operator $A$ describes driftless observables with time-independent expectation known as local martingales or conservation (in mean) laws of the process.

\subsection{SLE and CFT}

A general picture of the connections between SLE and CFT can be viewed as follows. The axiomatic approach to CFT grew up from
the Hilbert sixth problem \cite{Wightman}, and the Euclidean axioms were suggested by Osterwalder and Schrader \cite{OS}. They are centered around group symmetry, relative to unitary representations of Lie groups in Hilbert space. They
define first {\it correlators} (complex values amplitudes)
dependent on $n$ complex variables,  and
a Lie group of conformal transformations of the correlators under the M\"obius group. This can be extended at the infinitesimal level of the Lie algebra to invariance
under infinitesimal conformal transforms, and therefore, to so-called Ward identities. The adjoint representation of this group is given with the help of the  enveloping algebra of
an algebra of special operators, that act on a Fock space of {\it fields}. The existence of fields and relation between correlators and fields are given by a {\it reconstruction theorem}, see e.g., \cite{Schott}. The boundary version of this approach BCFT, i.e., CFT on domains with boundary,
was developed by Cardy \cite{Cardy}. SLE approach starts with a family of statistical fields generated by non-random central charge modification of the random fields
defined initially by the {\it Gaussian Free Field} (GFF) and the algebra of Fock space fields,  and then, defines SLE correlators,  which turn to be local martingales after coupling of modified GFF on SLE random domains. These correlators satisfy  the axiomatic properties of BCFT in which the infinitesimal boundary distortion leads to Ward identities
involving a special boundary changing operator of conformal dimension that depends on the amplitude $\sqrt{k}$ of the Brownian motion in SLE.
 
Connections between SLE and CFT were considered for the first time by Bauer and Bernard \cite{BB02}. General motivation was as follows. Belavin, Polyakov and Zamolodchikov \cite{BPZ} defined in 1984 a class of conformal theories `minimal models', which described some discrete models (Ising, Potts, etc.) at criticality. Central theme is {\it universality}, i.e.,  the properties of a system close to the critical point are independent of its microscopic realization. Universal classes are characterized by a special parameter, central charge. Schramm's approach  is based on a special evolution of  conformal maps describing possible candidates for the scaling limit of interface curves. How these approaches are related? BPZ conjectured that the behaviour  of the system at criticality should be  described by critical exponents  identified as highest weights of  degenerate representations of infinite-dimensional Lie algebras, Virasoro in our case.

In order to make a bridge to CFT, let us address the definition of {\it fields} under consideration, which will be in fact {\it Fock space fields} and non-random {\it martingale-observables} following
Kang and Makarov exposition~\cite{MK}, see also \cite{RBGW} for physical encouragement. The underlying idea is to construct the representation Fock space based on the Gaussian Free Field (GFF). GFF is a particular case of the L\'evy Brownian motion (1947) \cite{Levy} as a space extension of the classical Brownian motion. Nelson (1973) \cite{Nelson} considered relations of Markov properties of generalized random fields and QFT, in particular, he proved that GFF possesses the Markov property. Albeverio, H{\o}egh-Krohn (1979) \cite{AH} and  R\"ockner (1983,85)
\cite{Rockner83, Rockner86} proved the Markov property of measure-indexed GFF and revealed relations of GFF to potential theory. Difficulty in the approach described below
comes from the fact that in the definition of the Markov property, the domain of reference is chosen to be random which requires a special coupling between random fields and domains,
which was realized by Schramm and Sheffield \cite{SchrammSheffield10} and Dub\'edat \cite{Dub}.

GFF $\Phi(z)$ is defined on a simply connected domain $D$ with the Dirichlet boundary conditions, i.e., GFF
is indexed by the Hilbert space $\mathcal E(D)$, the completion of test functions $f\in C^{\infty}_0(D)$ (with compact support in $D$) equipped with the norm
\[
\|f\|^2=\int_{D}\int_{D}f(\zeta)\bar{f}(z)G(z,\zeta)d\zeta dz,
\]
where $G(z,\zeta)$ is the Green function of the domain $D$, so
$\Phi\colon \mathcal E(D)\to L^2(\Omega)$. For example, if $\mathbb H=D,$ then
$G(z,\zeta)=\frac{1}{2\pi}\log\bigg|\frac{\zeta-\bar{z}}{\zeta-z}\bigg|$. One can think of GFF as a generalization of the Brownian motion to complex time, however, such analogy is very rough since for a fixed $z\in D$ the expression $\Phi(z)$
is not well-defined as a random variable, e.g., the correlation is $\mathsf{E}(\Phi(z)\Phi(\zeta))=G(z,\zeta)$, but the variance does not exist in a usual sense.
Instead, in terms of distributions $(\Phi, f)$ over the space  $\mathcal E(D)$ we have covariances
\[
\text{Cov}((\Phi, f_1), (\Phi, f_2))= \frac{1}{2\pi}\int_{D}\int_{D}f_1(\zeta)\bar{f}_2(z)G(z,\zeta)d\zeta dz.
\]

The distributional derivatives $J=\partial \Phi$ and $\bar{J}=\bar{\partial} \Phi$ are well-defined as, e.g.,  $J(f)=-\Phi(\partial f)$, $J\colon \mathcal E(D)\to L^2(\Omega)$. They are  again Gaussian distributional fields.

The tensor $n$-th symmetric product of Hilbert spaces $\mathcal H(D)$ we denote by $\mathcal H^{\odot n}$, $\mathcal H^{\odot 0}=\mathbb C$ and
the Fock symmetric space is defined as the direct sum $\bigoplus_{n=0}^{\infty}\mathcal H^{\odot n}$. Here the sign $\odot$ of the Wick product (defined below) denotes  an isomorphism to the symmetric tensor algebra multiplication. Gian Carlo Wick (1909-1992) introduced originally the product $:\cdot:\equiv \odot$ in  \cite{Wick}  in 1950, in order to provide useful information of infinite quantities in Quantum Field Theory. In physics, Wick product  is related to the normal ordering of operators over a representation space, namely, in terms of annihilation and creation operators all the creation
operators appear to the left of all annihilation operators. The Wick product $\odot$ of random variables $x_j$ is a random variable, commutative, and is defined formally as $:\,\,:=1$, $\frac{\partial}{\partial x_j}(x_1\odot \dots \odot x_n)=(x_1\odot \dots\odot x_{j-1}\odot x_{j+1}\odot \dots \odot x_n)$, $\mathsf{E}(x_1\odot \dots \odot x_n)=0$ for any $k\geq 1$. For example, $:x:=\odot x=x-\mathsf{E}(x)$. However, we have to understand that $\Phi(z)$ and its derivatives are not random variables in the usual sense and must be think of as distributional random variables. Becides, $\odot \Phi=\Phi$. The mean of the field $\mathsf{E}\Phi$ is harmonic.

The {\it Fock space correlation functionals} are defined as  span of basis correlation functionals 1 and $X(z_1)\odot \dots \odot X(z_n)$, $z_1,\dots,z_n\in D$, $X_k=\partial^{\alpha}\bar{\partial}^{\beta}\Phi$ as well as infinite combinations (exponentials) which will play a special role in the definition of the vertex operator.
We have  $\partial\bar{\partial} \Phi\approx 0$, i.e., $\mathsf{E}((\partial\bar{\partial} \Phi(z) Y(\zeta))=0$, for any functional $Y(\zeta)$, $\zeta\neq z$. In view of this notation $\Delta \Phi\approx 0$.

The {\it basis Fock space fields} are formal Wick products  of the derivatives of the Gaussian free field $1$, $\Phi$, $\Phi\odot \Phi$, $\partial \Phi\odot \Phi$, etc.
Again, since GFF and its derivatives are understood in distributional sense, the above Wick product is formal. A {\it Fock space field} $F_{D}$ is a linear combination of basis Fock space fields over the ring of smooth functions in $D$. If $F_1, \dots, F_n$ are Fock space fields and $z_1,\dots, z_n$ are distinct points of $D$, then $F_1(z_1)\dots F_n(z_n)$ is a correlation functional and $\mathsf{E}(F_1(z_1)\dots F_n(z_n))$ is a {\it correlation function} of simply correlator. The product here is thought of as  a tensor product defined by the Wick formula over Feynman diagrams.

Being formally defined all these objects can be recovered through approximation by well-defined objects (scaling limit of lattice GFF) and expectations and correlators can be calculated.

We continue specifying the Markov property of domains and fields. For a shrinking deterministic subordination $D(t)\subset D(s)$ for $t<s$, the Markov property of \linebreak $\{D(t)\}_{t\geq 0}$ means that domains $D(t)\big| {D(s)}$ and $D_{t-s}$ coincide in law, which simply means in the L\"owner case that if $\varphi_t\colon \mathbb H_t\to \mathbb H$ preserves  $\infty$, and such that the tip $\gamma(t)\to 0$, then $\varphi_s\circ g_t$ and $g_{t-s}$ satisfy the same equation fixing $s$. In the random case we have that the domains $D(t)\big| {D(s)}$ and $D_{t-s}$ coincide in law in the sense of distribution, or Law($D(s)\big|{D(t)}$) = Law($D_{s-t}$). For more thorough treatment
of the Markov property of domains, see \cite{La05}.

In order to
formalize the Markov property for fields and domains, we coming at  coupling of $\{D(t)\}_{t\geq 0}$ SLE and GFF to be defined on the same probability space $(\Omega,\mathcal G, (\mathcal G_t), P)$, see \cite{Dub, SchrammSheffield}.

Let us start with a toy example of a classical real-valued field $F_D$ defined on a simply connected domain $D$. Similarly to the Markov property of $g_t$, we say that
$F$ possesses the Markov property if for a decreasing $t$-dependent family of domains $D_t$, $F_{D_{t}}\big|{D_s}(z)=F_{D_{s-t}}$, where
$0<s<t$, $D_0=D$. If in \eqref{SL} we put $k=0$, then the family of domains $\mathbb H_t$ is the family  $D_t=\mathbb H\setminus \gamma(t)$, $\gamma(t)=(0,2i\sqrt{t}]$. As a trivial example, construct the field $F_{D_t}(z)=\arg \varphi_t'(z)$, where $\varphi_t=z^2+4t^2$. The function $\varphi_t$ maps $D_t$ onto $\mathbb C\setminus [0,\infty)$ and $F_{D_t}$  satisfies the Markov property.

A collection $\mathcal F=\{F_k\}$ of random holomorphic Fock space fields defined on $D_t$ is said to satisfy the {Markov property}, if the process
\[
M_t(z_1,\dots,z_n):=\mathsf{E}[F_{t1}(z_1)\dots F_{tn}(z_n)\big|D_t],
\]
  is a local martingale. Here $F_t(z_j)$ is push-forward of $F(z_j)$ with respect to a conformal map $D\to D_t$. That is, let Aut$(D,p,q)$ is the group of conformal automoprphisms
  of $D$ fixing $p,q\in \partial D$ ($0,\infty$ for $\mathbb H$ or $\gamma(t),\infty$ for $\mathbb H_t$). The first requirement is that $F$ is invariant under Aut$(D,p,q)$. The second is
that if $\varphi_t\colon D_t\to D$ is a conformal map fixing $q$, then primary fields $F$ are $(\lambda,\bar{\lambda})$-forms, i.e.,  $F_t(z)=(\varphi_t')^{\lambda}(\overline{\varphi_t'})^{\bar{\lambda}}F(\varphi_t(z))$. A Schwarzian form $F$ satisfies $F_t(z)=(\varphi_t')^2F(\varphi_t(z))+\mu S_{\varphi_t}(z)$, etc.

The space of  conformal Fock space fields is a graded commutative differential algebra over the ring of non-random conformal
fields. Conformal invariance in $D$ is assumed for non-random conformal fields, and for random conformal fields this means that all correlations are conformally invariant as non-random conformal fields.  Observe that all definitions can be given for Riemann surfaces as in~\cite{MK}.
The non-random field $M(z_1,\dots,z_n):=M_0(z_1,\dots,z_n)$ is called an {\it SLE martingale-observable}.

Now we want to use some Fock space fields as {\it states} and others as {\it observables} acting on states by {\it operator product expansion} (OPE).
OPE is defined being based on the expansion for GFF
\[
\Phi(\zeta)\Phi(z)=\log\frac{1}{|\zeta-z|^2}+2\log R_D(z)+\Phi(z)\odot\Phi(z)+o(1),\quad  \mbox{as $\zeta\to z$},
\]
where $R_D(z)$ is the conformal radius of $D$ with respect to $z$. The product $\Phi(\zeta)\Phi(z)$ is defined as a tensor product and is given by the Wick formula, see e.g.,~\cite[\S~4.3]{Peskin}.  A Fock space field $F(z)$ is called holomorphic if $\mathsf{E} (F(z)Y(\zeta))$ is a holomorphic function of $z$ for any
field $Y(\zeta)$, $z\neq \zeta$. The operator product expansion of a holomorphic field is given as a Laurent series
\[
F(\zeta)Y(z)=\sum_{n\in\mathbb Z}C_n(z)(\zeta-z)^n,\quad \text{as $\zeta\to z$}.
\]
Obviously, OPE is neither commutative nor associative unless one of the fields is non-random. The coefficients $C_n$ are also Fock space fields. We denote  $F*Y=C_0$ and $F*_{(n)}Y=C_{n}$ for holomorphic Fock space fields. This product satisfies the Leibniz rule.

The complex Virasoro algebra was introduced in Section~\ref{WV}. Let us define a special field, Virasoro field of GFF, $T(z)$ by the equality $T=-\frac{1}{2}J*J$. In particular,
\[
J(\zeta)J(z)=-\frac{1}{(\zeta-z)^2}-2T(z)+o(1),\quad\text{as $\zeta\to z$ in $\mathbb H$},
\]
where $J(z)=\partial \Phi(z)$.

Let us define a non-random modification $\Phi_{(b)}$ on a simply connected domain $D$ of the Gaussian free field $\Phi_{(0)}=\Phi$ on  $D$ parametrized by a real constant $b$, $\Phi_{(b)}=\Phi_{(0)}-2b\arg\varphi'$, where $\varphi\colon D\to D$ fixing a point $\in \partial D$. Then $J_{(b)}=J_{(0)}+ib\frac{\varphi''}{\varphi'}$ is a pre-Schwarzian form and $T_{(b)}=T_{(0)}-b^2S_{\varphi}$  is a Schwarzian form. The modified families of Fock space fields $\mathcal F_{(b)}$ have the central charge $c=1-12b^2$. In order to simplify notations we omit subscript writing simply $T:=T_{(b)}$, $J:=J_{(b)}$, etc.

The field $T$ becomes a stress-energy tensor in Quantum Field Theory and satisfies the equality
$$
T(\zeta)T(z)=\frac{c/2}{(\zeta-z)^4}+\frac{2T(z)}{(\zeta-z)^2}+\frac{\partial T(z)}{\zeta-z}+o(1),
$$
It is not a primary field and changes under conformal change of variables as $T(\zeta')=T(\zeta)-\frac{c}{12}S_{\zeta'}(\zeta)$,
where $S_{\zeta'}(\zeta)$, as usual, is the Schwarzian derivative of $\zeta'(\zeta)$.

The Virasoro field has the expansion
\[
T(\zeta)X(z)=\sum\limits_{n\in \mathbb Z}\frac{L_nX(z)}{(\zeta-z)^{n+2}},
\]
for any Fock space field $X(z)$. This way the Virasoro generators act on a field $X$: $L_nX=T*_{(-n-2)}X$ and can be viewed as operators over the space of Fock space fields.
The result is again a Fock space field and one can define iteratively the field $L_{n_k}L_{n_{k-1}}\dots L_{n_1}X\equiv L_{n_k}L_{n_{k-1}}\dots L_{n_1}|X\rangle$, understanding $L_j:=L_j*$ as operators acting on a `vector' $X$, where we adapt Dirac's notations `bra' and `ket' for vectors, operators and correlators. So we obtained a representation of the   Virasoro algebra  on Fock space fields
\[
[L_n,L_m]=L_n*L_m-L_m*L_n =(n-m)L_{n+m}+\frac{c}{12}n(n^2-1)\delta_{n,-m}, \quad n,m\in\mathbb Z.
\]

Let us recall that a field $X$ is called {\it primary} of conformal weight $(\lambda,\bar{\lambda})$  if $X(z)=X(\phi(z))(\phi')^{\lambda}(\bar\phi')^{\bar{\lambda}}$.
The {\it Virasoro primary field} $|V\rangle$ of conformal dimension $\lambda$, or ($(\lambda,\bar{\lambda}$), is defined to satisfy $L_n|V\rangle=0$, $n>0$, $L_0|V\rangle=\lambda |V\rangle$ and $L_{-1}|V\rangle=\partial |V\rangle$. Analogously, the conjugate part of this definition is valid for the dimension $\bar{\lambda}$. Further on we omit the conjugate part
because of its complete symmetry with the non-conjugate one.
The {\it Virasoro-Verma module} $\mathcal V(\lambda,c)$ is constructed spanning and completeing the basis vectors $L_{n_k}L_{n_{k-1}}\dots L_{n_1}|V\rangle$, where $n_k<n_{k-1}<\dots<n_1<0$, and $|V\rangle$ is taken to be the highest weight vector. We have a decomposition
\[
\mathcal V(\lambda,c)=\bigoplus\limits_{m=0}^{\infty}\mathcal V_m(\lambda,c),
\]
where the level $m$ space $\mathcal V_m(\lambda,c)$ is the eigenspace of $L_0$ with eigenvalue $\lambda+m$.
A singular vector $|X\rangle$ by definition lies in some $\mathcal V_m(\lambda,c)$ and $L_{n}|X\rangle=0$ for any $n>0$.
The Virasoro-Verma module is generically irreducible, having only singular-generated submodules. The Virasoro primary field becomes the highest weight vector in the representation of the Virasoro algebra $\mathfrak{vir}_{\mathbb C}$.

\begin{proposition}
Let $|V\rangle$ be a Virasoro primary field in $\mathcal F$ of conformal dimension $\lambda$ with central charge $c$. Then the field
\[
[L_{-m}+\eta_1L_{-1}[L_{-m+1}+\eta_2L_{-1}[L_{-m+2}+\dots+\eta_{m-2}L_{-1}[L_{-2}+\eta_{m-1}L_{-1}^2]\dots]|V\rangle
\]
is a primary singular field of dimension $(\lambda+m)$, if and only if $\eta_1,\dots,\eta_{m-1}$, $c$, and $\lambda$ satisfy a system of $m$ linear equations.
\end{proposition}

For example, if $m=2$, then the singular field is
\begin{equation}\label{singular}
[L_{-2}+\eta_{1}L_{-1}^2]|V\rangle,
\end{equation}
and
\[
\begin{cases}
3+2\eta_1+4\eta_1\lambda=0,\\
c+8\lambda+12\eta_1\lambda=0,
\end{cases}
\]
 if $m=3$, then the singular field is $[L_{-3}+\eta_2L_{-1}L_{-2}+\eta_{1}\eta_2L_{-1}^3]|V\rangle$ and
 \[
\begin{cases}
2+(\lambda+2)\eta_2=0,\\
1+4(\lambda+1)\eta_1=0,\\
5+(4\lambda+\frac{c}{2}+3)\eta_2+6(3\lambda+1)\eta_1\eta_2=0.
\end{cases}
\]

Now we consider the `holomorphic part' of GFF $\phi(z)=\int^zJ(z)$, where $J(z)=\partial \Phi$. Of course the definition requires more work because the field constructed this way is ramified at all points. However, in correlations with a Fock space functional, the monodromy group is well-defined and finitely generated. A vertex operator is a  field $\mathcal V_{\star}^{\alpha}(z)=e^{* i\alpha \phi(z)}$. It is a primary field of  conformal dimension  $\lambda=\frac{\alpha^2}{2}+\alpha\sqrt{\frac{1-c}{12}}$. Considering an infinitesimal boundary distortion $w_{\varepsilon}(z)=z+\frac{\varepsilon}{z-\xi}+o(\varepsilon)$ is equivalent  to the insertion operator $|X(z)\rangle\to T(\xi)|X(z)\rangle$. The Ward identity implies the BCFT Cardy equation
\begin{equation}\label{Cardy}
(\mathcal L_{-1}^2-2\alpha^2\mathcal L_{-2})\mathsf{E} \left[\mathcal V_{\star}^{\alpha}(\xi)|X(z)\rangle\right]=0,\quad
\text{for \ }\alpha\left(\alpha+\sqrt{\frac{1-c}{12}}\right)=1.
\end{equation}
 Here the representation of the Virasoro algebra is $$\mathcal L_n=-(z-\xi)^{n+1}\partial-\lambda(n+1)(z-\xi)^n.$$

Merging the form of the infinitesimal generator $A$ in \eqref{L1}, Virasoro primary singular field \eqref{singular} and the Cardy equation \eqref{Cardy} we arrive at
 the SLE numerology $\eta=-2\alpha^2$, $\alpha=\frac{\sqrt{k}}{2}$,
\[
c=\frac{(6-k)(3k-8)}{2k},\quad \lambda=\frac{6-k}{2k},
\]
with the unique free parameter $k$.

Given a Fock space field $X(z)$, its push-forward $X_t(z)$ generically does not possess the Markov property. However, summarizing above, we can formulate the following statement.
\begin{proposition} For SLE evolution we have
\begin{itemize}
\item   { $\mathcal V_{\star}^{\alpha}$} is a Virasoro primary field;
\item  { $V=\mathcal V_{\star,t}^{\alpha}(\xi_t)X_t(z)$} possesses the Markov property, $\xi_t=\sqrt{k}B_t$;
\item { $M(z)=\mathsf{E} \left[\mathcal V_{\star}^{\alpha}(0)X(z)\right]$} is a one-point martingale-observable;
\item  The process { $M_t(z)=\mathsf{E} \left[\mathcal V_{\star,t}^{\alpha}(\xi_t)X_t(z)\big| D_t\right]$} is a one-point martingale;
\item {$(\frac{k}{2}L_{-1}^2-2L_{-2})|\mathcal V_{\star}^{\alpha}\rangle$} is  Virasoro primary singular field of level 2;
\item   { $(\frac{k}{2}\mathcal L_{-1}^2-2\mathcal L_{-2})M(z)=0$} is the Cardy equation for the SLE martingale observables.
\end{itemize}
\end{proposition}

As a simple example, consider the Fock space field  $V=\mathcal V_{\star}^{\alpha}(0)T(z)$. Then
$M(z)=\mathsf{E} \left[\mathcal V_{\star}^{\alpha}(0)T(z)\right]$ is a one-point martingale-observable and
 $M(z)=1/z^2$.
If $k=8/3$,  then  $M_t(z)=(h'_t(z)/h_t(z))^2$  is a local martingale.

 More general construction including multi-point observables requires more technical work related to a vertex field $V$ and boundary condition changing operators, see \cite[\S 8.4]{MK}. Examples of martingale observables were found, e.g., by Friedrich and Werner \cite{FriedrichWerner} and Schramm and Sheffield \cite{SchrammSheffield, SchrammSheffield10}. A radial version of SLE and relations to conformal field theory one can find  in \cite{MK2}.

Another construction of the stress-energy tensor of CFT comes as a local observable of the CLE (Conformal Loop Ensemble) (see, e.g., \cite{Sheffield}) loops for any central charge, see \cite{Doyon1, Doyon2}. More general construction is performed on a groupoid of conformal maps of a simply connected domain, a natural generalization of the finite-dimensional conformal group. The underlying manifold structure is Fr\'echet. Similarly to moduli (Teichm\"uller) spaces, the elements of the cotangent bundle are analogues of quadratic differentials, see  \cite{Doyon3}. It is shown there, that the stress-energy tensor of CFT is exactly such a quadratic differential.

\subsection{Generalized L\"owner-Kufarev stochastic evolution}

Another attempt  to construct random interfaces different from SLE has been launched by conformal welding in \cite{Astala}.

We considered a setup \cite{IV} in which the sample paths are represented by the trajectories of a point (e.g., the origin) in the unit disk $\mathbb D$ evolving randomly under the generalized L\"owner equation. The driving mechanism differs from SLE. In the SLE case the Denjoy-Wolff attracting point ($\infty$ in the chordal case or a boundary point of the unit disk in the radial case) is fixed. In our case, the attracting point is the driving mechanism and the Denjoy-Wolff point is different from it.  Relation with this model to CFT is the subject of a forthcoming paper.
Let us consider the generalized L\"owner evolution driven by a Brownian particle on the unit circle. In other words, we study the following initial value problem.
\begin{equation}
\label{eq:randomloewner}
\begin{cases}
 \frac{d}{dt} \phi_t (z,\omega) = \frac{(\tau(t,\omega)  - \phi_t(z,\omega))^2}{\tau(t,\omega)} p(\phi_t (z,\omega),t,\omega), \\
\phi_0(z,\omega) = z,
\end{cases}
 t \geq 0, \, z \in\mathbb{D}, \, \omega \in \Omega.
\end{equation}

The function $p(z,t,\omega)$ is a Herglotz function for each fixed $\omega\in \Omega.$ In order for $\phi_t(z,\omega)$ to be an It\^o process adapted to the Brownian filtration, we require that the function $p(z,t,\omega)$ is adapted to the Brownian filtration for each $z\in\mathbb{D}.$
 Even though the driving mechanism in our case differs from that of SLE, the generated families of conformal maps still possess the important  time-homogeneous Markov property.

For each fixed $\omega\in \Omega,$ equation \eqref{eq:randomloewner} similarly to SLE, may be considered as a deterministic generalized L\"owner equation with the Berkson-Porta data $(\tau(\cdot,\omega),p(\cdot, \cdot, \omega)).$ In particular, the solution $\phi_t(z,\omega)$ exists, is unique for each $t>0$ and $\omega\in\Omega,$ and moreover, is a family of holomorphic self-maps of the unit disk.

The equation in (\ref{eq:randomloewner}) is an example of a so-called \emph{random differential equation} (see, for instance, \cite{soong73}). Since for each fixed $\omega \in\Omega$ it may be regarded as an ordinary differential equation, the sample paths $t \mapsto \phi_t(z,\omega)$ have continuous first derivatives for almost all $\omega$. See an example of a sample path of $\phi_t(0,\omega)$ for $p(z,t) = \frac{\tau(t) + z }{ \tau(t)-z }$, $\tau(t) = e^{ikB_t},$ $k = 5$, $t \in[0,30]$ in the figure to the left.

In order to give an explicitly solvable example let $p(z,t,\omega) = \frac{\tau(t,\omega)}{\tau(t,\omega) - z}= \frac{e^{ikB_t(\omega)}}{e^{ikB_t(\omega)} - z}$. It makes  equation  \eqref{eq:randomloewner} linear:
\[
\frac{d}{dt} \phi_t (z,\omega) = e^{ikB_t(\omega)} - \phi_t(z,\omega),
\]
and a well-known formula from the theory of ordinary differential equation yields
\[
 \phi_t(z,\omega) = e^{-t} \left(z + \int_0^t e^{s} e^{ikB_s(\omega)}ds\right).
\]

Taking into account the fact that $\mathsf{E} e^{ikB_t(\omega)} = e^{-\frac12 t k^2},$ we can also write the expression for the mean function  $ \mathsf{E}\phi_t(z,\omega)$

\begin{equation}
\label{eq:meanfunctionphi}
\mathsf{E}\phi_t(z,\omega) = \begin{cases}
                               e^{-t}(z+t), \quad \quad \quad \quad k^2 = 2, \\
			      e^{-t} z + \frac{e^{-tk^2/2}-e^{-t}}{1-k^2/2},\quad \textrm{otherwise.}
                              \end{cases}
\end{equation}

Thus, in this example all maps $\phi_t$ and $\mathsf{E} \phi_t$ are affine transformations (compositions of a scaling and a translation).

In general, solving the random differential equation (\ref{eq:randomloewner}) is much more complicated than solving its deterministic counterpart, mostly because of the fact that for almost all $\omega$ the function $t \mapsto \tau(t,\omega)$ is nowhere differentiable.

If we assume  that the Herglotz function has the form $p(z,t,\omega) = \tilde{p}(z/\tau(t,\omega)),$ then it turns out that the process $\phi_t(z,\omega)$ has an important invariance property, that were crucial in development of SLE.

Let $s>0$ and introduce the notation
\[
 \tilde{\phi}_t(z) = \frac{\phi_{s+t}(z)}{\tau(s)}.
\]
Then $\tilde{\phi}_t(z)$ is the solution to the initial-value problem
\[
 \begin{cases}
  \frac{d}{dt} \tilde{\phi}_t (z,\omega) = \frac{\left(\tilde{\tau}(t,\omega)  - \tilde{\phi}_t(z,\omega)\right)^2}{\tilde{\tau}(t,\omega)} \tilde{p}\left(\tilde{\phi}_t (z,\omega)/\tilde{\tau}(t)\right), \\
\tilde{\phi}_0(z,\omega) =\phi_s(z,\omega)/\tilde{\tau}(s),
 \end{cases}
\]
where $\tilde{\tau}(t) = \tau(s+t)/ \tau(s) = e^{ik(B_{s+t}-B_s)}$ is again a Brownian motion on $\mathbb{T}$ (because $\tilde{B}_t = B_{s+t}-B_s$ is a standard Brownian motion). In other words, the conditional distribution of $\tilde{\phi}_{t}$ given $\phi_{r},$ $r\in[0,s] $ is the same as the distribution of $\phi_t.$

By the complex It\^o formula, the process $\frac{1}{\tau(t,\omega)} = e^{-ik B_t}$ satisfies the equation
\[
 d e^{-ikB_t} = -ik e^{-i k B_t} dB_t - \frac{k^2}{2} e^{- ikB_t}dt.
\]
Let us denote $\frac{\phi_t(z,\omega)}{\tau(t,\omega)}$ by $\Psi_t(z,\omega).$ Applying the integration by parts formula to $\Psi_t,$ we arrive at the following initial value problem for the It\^o stochastic differential equation
\begin{equation}
\begin{cases}
\label{eq:stochasticloewner}
  d\Psi_t =  - ik \Psi_t dB_t  + \left( - \frac{k^2}{2}\Psi_t + (\Psi_t -1)^2 p(\Psi_t e^{ikB_t(\omega)}, t, \omega)\right) dt, \\
\Psi_0(z) = z.
\end{cases}
\end{equation}
A numerical solution to this equation for a specific choice for $p(z,t) \equiv i,$ $k=1$, and $t\in[0,2]$, is shown in the figure to the right.

Analyzing the process  $\frac{\phi_t(z,\omega)}{\tau(t,\omega)}$ instead of the original process $\phi_t(z,\omega)$ is in many ways similar to one of the approaches used in SLE theory.

The image domains $\Psi_t(\mathbb{D},\omega)$ differ from $\phi_t(\mathbb{D},\omega)$ only by rotation. Due to the fact that $|\Psi_t(z,\omega)|=|\phi_t(z,\omega)|$, if we compare the processes $\phi_t(0,\omega)$ and $\Psi_t(0,\omega),$ we note that their first hit times of the circle $\mathbb{T}_r$ with radius $r<1$ coincide, i.~e.,
\[
\inf \{t \geq 0, |\phi_t(0,\omega)| = r \} = \inf \{t \geq 0, | \Psi_t(0,\omega)| = r \}.                                                                                                                                                                                                                                                                                                                                                                                                                                                                                                                                                                                                                                                                                                                                                                                                                                                                                                                                                                                                                                                                                                                                 \]
In other words, the answers to probabilistic questions about the expected time of hitting the circle $\mathbb{T}_r,$ the probability of exit from the disk $\mathbb{D}_r = \{z : |z| < r\},$ etc. are the same for $\phi_t(0,\omega)$ and $\Psi_t(0,\omega).$

If the Herglotz function has the form $p(z,t,\omega) = \tilde{p}(z/\tau(t,\omega))$, then the equation \eqref{eq:stochasticloewner} becomes
\begin{equation}
\begin{cases}
\label{eq:loewnerdiffusion}
  d\Psi_t =  - ik \Psi_t dB_t  + \left( - \frac{k^2}{2}\Psi_t + (\Psi_t -1)^2 \tilde{p}(\Psi_t)\right) dt, \\
\Psi_0(z) = z,
\end{cases}
\end{equation}
and may be regarded as an equation of a 2-dimensional time-homogeneous real diffusion written in complex form. This implies, in particular, that $\Psi_t$ is a time-homogeneous strong Markov process. By construction, $\Psi_t(z)$ always stays in the unit disk.

Analogously to SLE we can consider random conformal Fock space fields defined on $\mathbb D$ changing correspondingly the definition using the Green function for $\mathbb D$
instead of $\mathbb H$. Coupling of equation  (\ref{eq:randomloewner}) and the Gaussian free field in $\mathbb D$ we define the correlators $f_t(z_1,\dots, z_n)=f(z_1,\dots, z_n)\big|_{\mathbb D_t}$ as martingale-observables.

For a smooth function $f(z)$ defined in $\mathbb D$  we derive the It\^o differential in the complex form
\begin{equation*}
df(\Psi_t)=-ikdB_t (L_{-1}-\bar{L}_{-1})f(\Psi_t)+dt Af(\Psi_t),
\end{equation*}
where $A$ is the infinitesimal generator of $\Psi_t$,
\begin{multline*}
  A = \left(-\frac{k^2}{2}z + (z-1)^2 \tilde{p}(z)\right) \frac{\partial}{\partial z} -\frac{1}{2}k^2z^2 \frac{\partial^2}{\partial z^2} \\
+\left(-\frac{k^2}{2}\bar{z} + (\bar{z}-1)^2 \overline{\tilde{p}(z)}\right) \frac{\partial}{\partial \bar{z}} -\frac{1}{2}k^2\bar{z}^2 \frac{\partial^2}{\partial \bar{z}^2} + k^2 |z|^2 \frac{\partial^2}{\partial z \partial \bar{z}}.
\end{multline*}
In particular, if $f$ is holomorphic, then
\begin{equation}
\label{eq:generator}
 A = \left(-\frac{k^2}{2}z + (z-1)^2 \tilde{p}(z)\right) \frac{\partial}{\partial z} -\frac{1}{2}k^2z^2 \frac{\partial^2}{\partial z^2}.
\end{equation}
If $f(z)$ is a martingale-observable, then  $Af=0$.

In \cite{IV} we proved the existence of a unique stationary point of $\Psi_t$ in terms of  the stochastic vector field
\[
 \frac{d}{dt}\Psi_t(z,\omega) = G_0(\Psi_t(z,\omega)),
\]
where the Herglotz vector field $G_0(z,\omega)$ is given by
\[
G_0(z,\omega) = -ikz W_t(\omega)  - \frac{k^2}{2}z+  (z -1)^2 \tilde{p}(z).
\]
Here, $W_t(\omega)$ denotes a generalized stochastic process  known as  \emph{white noise}. Also $n$-th moments were calculated and the boundary diffusion on the unit circle was
considered, which corresponds, in particular, to North-South flow, see e.g., \cite{Carverhill85}.

\section{Related topics}

\subsection{Hele-Shaw flows, Laplacian growth}

One of the most influential works in fluid dynamics at the end of
the 19-th century was a series of papers, see e.g., \cite{Hele} written by {\it
Henry Selby Hele-Shaw} (1854--1941). There Hele-Shaw first
described his famous cell that became a subject of deep
investigation only more than 50 years later. A Hele-Shaw cell is a
device for investigating two-dimensional flow (Hele-Shaw flow or Laplacian growth) of a viscous fluid
in a narrow gap between two parallel plates.

This cell is the
simplest system in which multi-dimensional convection is present.
Probably the most important characteristic of flows in such a cell
is that when the Reynolds number based on gap width is
sufficiently small, the Navier-Stokes equations averaged over the
gap reduce to a linear relation for the velocity similar to Darcy's law and then to
a Laplace equation for the fluid pressure. Different driving mechanisms can
be considered, such as surface tension or external forces
(e.g., suction, injection). Through the similarity in the governing
equations, Hele-Shaw flows are particularly useful for
visualization of saturated flows in porous media, assuming they
are slow enough to be governed by Darcy's law. Nowadays, the
Hele-Shaw cell is used as a powerful tool in several fields of
natural sciences and engineering, in particular,  soft condensed matter physics,
materials science, crystal growth and, of course, fluid mechanics. See more on historical and scientific account in \cite{Vas09}.

The century-long development connecting the original Hele-Shaw experiments, the  conformal mapping formulation of the Hele-Shaw flow by {\it
Pelageya Yakovlevna Polubarinova-Kochina}
(1899--1999) and {\it Lev
Aleksan\-dro\-vich Galin} (1912--1981) \cite{Polub1}, \cite{Polub2}, \cite{Galin}, and
the modern treatment of the Hele-Shaw evolution  based on integrable systems and on the general theory of plane contour motion, was marked by
several important contributions by individuals and groups.

The main idea of Polubarinova-Kochina and Galin was to apply the Riemann
mapping from an appropriate canonical domain (the unit disk $\mathbb D$ in
most situations) onto the phase domain in order to parameterize the free
boundary. The evolution equation for this map, named after its creators,
allows to construct many explicit solutions and to apply methods
of conformal analysis and geometric function theory to investigate
Hele-Shaw flows. In particular, solutions to this equation in the
case of advancing fluid give subordination chains of simply
connected domains which have been studied by L\"owner and Kufarev. The resulting equation
for the family of functions $f(z,t)=a_1(t)z+a_2(t)z^2+\dots$ from $\mathbb D$ onto domains occupied by viscous fluid is
\begin{equation}\label{eq4}
\re[\dot{f}(\zeta,t)\overline{\zeta
f'(\zeta,t)}]=\frac{Q}{2\pi}, \quad |\zeta|=1.
\end{equation}
The corresponding equation in $\mathbb D$ is a first-order integro-PDE
\[
\dot{f}(\zeta,t)=\zeta f'(\zeta,t) \int_0^{2\pi}\frac{Q}{4\pi^2|f'(e^{i\theta},t)|^2}\frac{e^{i\theta}+\zeta}{e^{i\theta}-\zeta}d\theta,\quad |\zeta|<1.
\]

Here $Q$ is positive in the case of injection of negative in the case of suction.
The Polubarinova-Galin  and the
L\"owner-Kufarev equations, having some evident geometric connections,
are of somewhat different nature. While the evolution of the Laplacian growth given by the Polubarinova-Galin equation is completely defined by
the initial conditions,  the L\"owner-Kufarev evolution depends also on an arbitrary control function.
The Polubarinova-Galin
equation is essentially non-linear and the corresponding
subordination chains are of rather complicated nature. The treatment of classical Laplacian growth was given in the monograph \cite{GustafssonVasiliev}.

The newest direction in the development of Hele-Shaw flow is related to Integrable Systems and Mathematical Physics, as a particular case of a contour dynamics.
This story started in 2000 with a short note \cite{Mineev} by Mineev-Weinstein, Wiegmann and Zabrodin, where the idea of the equivalence
of 2D contour dynamics and the dispersionless limit of the integrable 2-D Toda hierarcy appeared for the first time.
A list of complete references
to corresponding works would be rather long so we only list the names of some key contributors:
Wiegmann, Mineev-Weinstein, Zabrodin, Krichever, Kostov, Marshakov, Takebe, Teo  {\it et al.},  and some important references: \cite{Kostov, Krichever3, Krichever4, Marshakov, Mineev3, Takebe, Wiegmann}.
\begin{figure}[h!] \scalebox{1.0}{
\begin{pspicture}(2,1)(15,7)
\pscurve[linewidth=0.15mm, fillstyle=solid,
fillcolor=lightgray](5,5)(3.5,6)(2.5,2)(6,2.5)(5,5) \rput(4.7,3.7){0}
\rput(5.3,7){$y$}\rput(8,4.3){$x$}
\rput(3.7,3){$\Omega^-$}\rput(6,5){$\Omega^+$}
\psline[linewidth=0.15mm]{->}(5,1)(5,7)
\psline[linewidth=0.15mm]{->}(1,4)(8,4)
 \pscircle[fillstyle=solid,
fillcolor=black](5,4){.1}\rput(12.5,4){
\begin{minipage}{8cm}\vspace{-2cm}\begin{itemize}
\item ${\displaystyle M_{-k}=-\int_{\Omega^+}z^{-k} d\sigma_z}$;
\item ${\displaystyle M_0=|\Omega^-|}$;
\item ${\displaystyle M_{k}=\int_{\Omega^-}z^{k} d\sigma_z}$;
\item $k\geq 1$,
\item $t =M_0/\pi$, $t_k=M_k/\pi k$ \newline generalized times.
\end{itemize}
\end{minipage}}
\end{pspicture}}
\end{figure}
Let us consider an `exterior' version of the process when source/sink of viscous fluid is at  $\infty$ and the bounded domain $\Omega^-$ is occupied by the inviscid
one. Then the conformal map $f$ of the exterior of the unit disk onto $\Omega^+$
\[ 
f(\zeta,t)=b(t)\zeta+b_0(t)+\sum\limits_{k=1}^{\infty}\frac{b_j(t)}{\zeta^k},\quad b(t)>0;
\] 
satisfies the analogous boundary equation
\begin{equation}\label{eq4a}
\re[\dot{f}(\zeta,t)\overline{\zeta
f'(\zeta,t)}]=-Q, 
\end{equation}
 Following the definition of Richardson's moments let us define interior and exterior moments as in the above figure.
The integrals for $k=1,2$ are assumed to be properly regularized. 
Then the properties of the moments are
\begin{itemize}
\item $M_0(t)=M_0(0)-Qt$ is `physical time';
\item $M_k$ are conserved for $k\geq 1$;
\item $M_0$ and $\{M_k\}_{k\geq 1}$ determine the domain $\Omega^+$ locally (given $\partial \Omega^+$ is smooth);
\item $\{M_k\}_{k< 0}$ evolve in time in a quite complicated manner;
\item $M_0$ and $\{M_k\}_{k\geq 1}$ can be viewed as local parameters on the space of `shapes';
\end{itemize}

Suppose $\Gamma =\partial \Omega^-=\partial \Omega^+$ is analytic and  
$\phi(x_1,x_2,t)=0$ is an implicit representation of the free
boundary $\Gamma(t)$.
Substituting $x_1=(z+\bar{z})/2$ and $x_2=(z-\bar{z})/2i$ into
this equation and solving it for $\bar{z}$ we obtain that $\Gamma=\{z:\,\,S(z,t)=\bar z\}$.
is given by the {\it Schwarz function} $S(z,t)$ which is defined and analytic in a neighbourhood of $\Gamma$. The Schwarz function can be constructed by means of the
 Cauchy integral
\[
g(z)=-\frac{1}{2\pi i}\int\limits_{\partial
\Omega}\frac{\bar{\zeta}d\zeta}{\zeta-z}.
\]
Define the analytic functions, $g_e(z)$, in 
$\Omega^+$ and  $g_i(z)$, in $\Omega^-$.
On $\Gamma=\partial \Omega$
the jump condition is
\[
g_e(z)-g_i(z)=\bar{z},\quad z\in \partial \Omega.
\]
Then the Schwarz function $S(z)$ is defined formally by
$S(z)=g_e(z)-g_i(z)$. The Cauchy integral implies the Cauchy transform  of $\Omega^+$
\[
g_e(z)=-\frac{1}{2\pi i}\int\limits_{\partial
\Omega}\frac{\bar{\zeta}d\zeta}{\zeta-z}=-\frac{1}{\pi}\int\limits_{\Omega^+}\frac{d\sigma_{\zeta}}{\zeta-z},
\]
with the Laurent expansion
$$
g_e(z)=\sum_{k=0}^\infty \frac{M_k}{z^{k+1}}, \quad z\sim \infty.
$$
 Similarly for $g_i(z)$
$$
g_i(z)=-\sum_{k=1}^\infty {M_{-k}}{z^{k-1}},  \quad z\sim 0;
$$
So formally
$$
S(z)=\sum_{k=-\infty}^\infty \frac{M_k}{z^{k+1}}=\sum_{k=1}^\infty {M_{-k}}{z^{k-1}}+\frac{M_0}{z}+\sum_{k=1}^\infty \frac{M_k}{z^{k+1}}.
$$
Let us write the logarithmic energy as
$$
\mathcal{F}(\Omega^+)=
-\frac{1}{\pi^2}\int_{\mathbb C}\int_{\mathbb C} \log
|z-\zeta| d\sigma_{z}d\sigma_{\zeta},
$$
where $\sigma_{z}$ is a measure supported in $\Omega^+$. It is the potential for the momens
$$
\frac{1}{k}M_{-k}(\Omega) = \frac{\partial \mathcal{F}(\Omega)}{\partial M_k}.
$$
So (see, e.g., \cite{Mineev})  the moments satisfy the 2-D Toda dispersionless lattice hierarchy
 $$\frac{\partial M_{-k}}{\partial t_j}=\frac{ \partial M_{-j}}{\partial t_k},\;\;\;\;\;
  \frac{ \partial M_{-k}}{\partial \bar t_j}=
\frac{\partial \bar M_{-j}}{\partial t_k}.$$
The function $\exp(\mathcal{F}(\Omega^+))=:\tau(\Omega^+)=\tau(\mathbf{t})$ is the $\tau$-function,
and $t=M_0/\pi$,\dots, $t_k=M_k/\pi k$ are generalized times.

The real-valued $\tau$- function becomes the solution to the Hirota equation
\[
S_{f^{-1}}(z)=\frac{6}{z^2}\sum\limits_{k,n=1}^{\infty}\frac{1}{z^{n+k}}\frac{\partial^2\mbox{log}\, \tau}{\partial t_k \partial t_n},
\]
where $z=f(\zeta)$ is the parametric map of the unit disk onto the exterior phase domain and $S_{f^{-1}}(z)$ denotes the Schwarzian derivative of the inverse to $f$.
\[
\frac{M_{-k}}{\pi}=
  \frac{\partial \,\mbox{log}\,\tau}{\partial t_k}\,,
\;\;\;\;\;\;
\frac{\bar M_{-k}}{\pi}=
  \frac{\partial \,\mbox{log}\, \tau}{\partial \bar t_k}\,,
\;\;\;\;\;\; k\geq 1.
\]
The $\tau$-function introduced by the `Kyoto School' as a central element in the description of soliton equation hierarchies.

If $\zeta = e^{i\theta}$, $M_0(t)=M_0(0)-Qt$, then the derivatives are
\[
\frac{\partial f}{\partial \theta} = i \zeta \frac{\partial
f}{\partial\zeta},\quad \frac{\partial f}{\partial t}=Q\frac{\partial f}{\partial M_0}.
\]
Let $f^*(\zeta)=\overline{f(1/\bar \zeta)}$, and let 
$$ \{f,g\}=\zeta \frac{\partial f}{\partial \zeta}\frac{\partial g}{\partial M_0}-\zeta \frac{\partial g}{\partial \zeta}\frac{\partial f}{\partial M_0}.$$
In view of this the Polubarinova-Galin equation (\ref{eq4a}) can be rewritten as $\{f,f^*\}=1$.
This equation is known as the {\it string constraint}.
The equation for the $\tau$-function with a proper initial condition provided by the string equation solves the inverse moments problem for small deformations of a simply connected domain with analytic boundary. Indeed, define the
Schwarzian derivative $S_{F}=\frac{F'''}{F''}-\frac{3}{2}\left(\frac{F''}{F'}\right)^2$, $F=f^{-1}$.
 If we know the Schwarzian derivative $S(\zeta)$, we know the conformal map $w=F(\zeta)=\eta_1/\eta_2$ normalized accordingly, where
$\eta_1$ and $\eta_2$ are linearly independent solutions to the Fuchs equation
\[
w''+\frac{1}{2}S(\zeta)w=0.
\]

The connection extends to the Lax-Sato approach to the dispersionless 2-D Toda hierarchy. In this setting it is shown that a Laurent series for a univalent function that provides an invertible conformal map of the exterior of the unit circle  to the exterior of the domain, can be identified with the Lax function. The $\tau$-function appears to be a generating function for the inverse map. The formalism allows one to associate a notion of $\tau$-function to the analytic curves.

The conformal map 
\[
f(\zeta,t)=b(t)\zeta+b_0(t)+\sum\limits_{k=1}^{\infty}\frac{b_j(t)}{\zeta^k},\quad b(t)>0
\]
 obeys
 the relations
\[
\frac{\partial f}{\partial t_k}=\{H_k,f\}, \quad \frac{\partial f}{\partial \bar t_k}=\{\bar H_k,f\},
\]
where 
\begin{itemize}
\item  $H_k=(f^k(\zeta))_{+}+\frac{1}{2}(f^k(\zeta))_{0}$,
\item $\bar{H}_k=(\bar f^k(1/\zeta))_{-}+\frac{1}{2}( \bar f^k(1/\zeta))_{0}$.
\end{itemize}

In the  paper \cite{Krichever4}, an analog of this theory for multiply-connected domains is developed. The answers are formulated in terms of the so-called Schottky double of the plane with holes. The Laurent basis used in the simply connected case is replaced by the Krichever-Novikov basis. As a corollary, analogs of the 2-D Toda hierarchy depend on standard times plus a finite set of additional variables. The solution of the Dirichlet problem is written in terms of the $\tau$-function of this hierarchy. The relation to some matrix problems is briefly discussed.

\subsection{Fractal growth}

Beno{\^\i}t  Mandelbrot (1924, Warsaw, Poland--2010, Cambridge, Massachusetts, United States)
brought to the world's attention that many natural
objects simply do not have a preconceived form determined by a
characteristic scale. He \cite{M} first saw a visualization of the set named after him, at IBM's Thomas J. Watson Research Center in upstate New York.

Many of the structures in space and
processes reveal new features when magnified beyond their usual
scale in a wide variety of natural and industrial processes, such
as crystal growth, vapor deposition, chemical dissolution,
corrosion, erosion, fluid flow in porous media and biological
growth a surface or an interface, biological processes. A
fractal, a structure coined by Mandelbrot in 1975 (`fractal' from Latin `fractus'), is a
rough or fragmented geometric shape that can be subdivided in
parts, each of which is (at least approximately) a reduced-size
copy of the whole. Fractals are generally self-similar,
independent of scale, and have (by Mandelbrot's own definition)
the Hausdorff dimension strictly
greater than the topological dimension. There are many
mathematical structures that are fractals, e.g., the Sierpinski
triangle, the Koch snowflake, the Peano curve, the Mandelbrot set,
and the Lorenz attractor. One of the ways to model a fractal is
the process of fractal growth that can be either stochastic or
deterministic. A nice overview of fractal growth phenomena is
found in \cite{Viczek}.

Many models of fractal growth patterns combine complex geometry
with randomness. A typical and important model for pattern
formation is {\it Diffusion-Limited
Aggregation} (DLA) (see a
survey in \cite{Hasley}). Considering colloidal particles
undergoing Brownian motion in some fluid and
letting them adhere irreversibly on contact with another one bring
us to the basics of DLA. Fix a seed particle at the origin and
start another one form infinity letting it perform a random walk.
Ultimately, that second particle will either escape to infinity or
contact the seed, to which it will stick irreversibly. Next
another particle starts at infinity to walk randomly until it
either  sticks to the two-particle cluster or escapes to infinity.
This process is repeated to an extent limited only by modeler's
patience. The clusters generated by this process are highly
branched and fractal, see figure to the left.

\begin{wrapfigure}{r}{0.3\textwidth}
\vspace{-20pt}
\begin{center}
\includegraphics[width=0.25\textwidth]{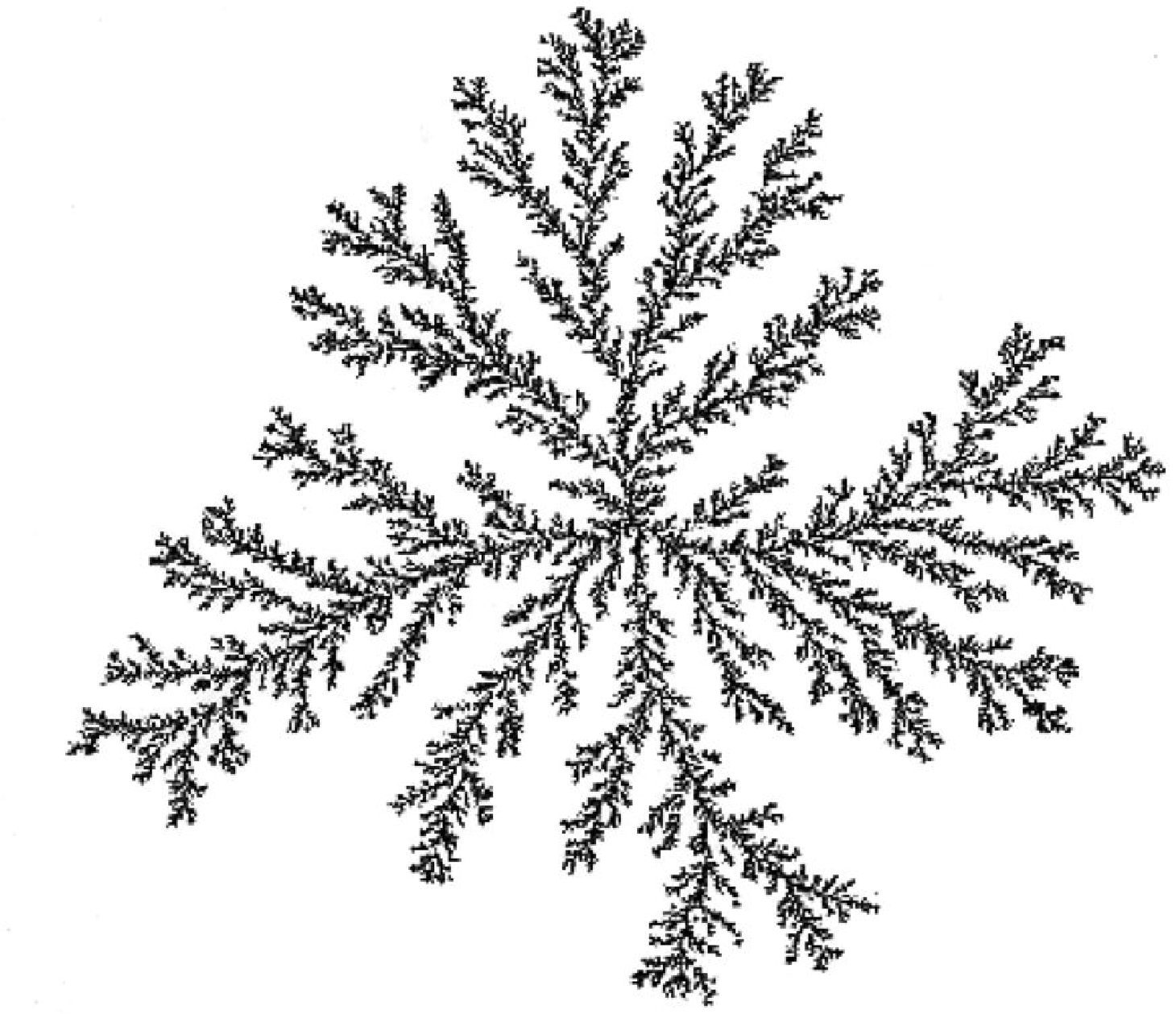}
\end{center}
\vspace{-20pt}
\end{wrapfigure}

The DLA model was introduced in 1981 by Witten and Sander
\cite{WittenSander1}, \cite{WittenSander2}. It was shown to
have relation to dielectric breakdown \cite{Niemeyer}, one-phase
fluid flow in porous media \cite{Chan}, electro-chemical
deposition \cite{Grier}, medical sciences \cite{Savakis}, etc. A
new conformal mapping language to study DLA was proposed by
Hastings and Levitov \cite{Hastings1}, \cite{Hastings2}. They
showed that two-dimensional DLA can be grown by iterating
stochastic conformal maps. Later this method was thoroughly
handled in \cite{Davidovitch}.

For a continuous random walk in 2-D the diffusion equation provides the law for the probability $u(z,t)$ that
the walk reaches a point $z$ at the time $t$,
\[
\frac{\partial u}{\partial t}=\eta \Delta u,
\]
where $\eta$ is the diffusion coefficient. When the cluster growth
rate per surface site is negligible compared to the diffusive
relaxation time, the time dependence of the relaxation may be
neglected (see, e.g., \cite{WittenSander2}). With a steady flux
from infinity and the slow growth of the cluster the left-hand
side derivative can be neglected and we have just the Laplacian
equation for $u$. If $K(t)$ is the closed aggregate at the time
$t$ and $\Omega(t)$ is the connected part of the complement of
$K(t)$ containing infinity, then the probability of the appearance
of the random walker in $\mathbb C\setminus \Omega(t)$ is zero.
Thus, the boundary condition $u(z,t)\big|_{\Gamma(t)}=0$,
$\Gamma(t)=\partial\Omega(t)$ is set. The only source of time
dependence of $u$ is the motion of $\Gamma(t)$. The problem
resembles the classical Hele-Shaw problem, but the complex structure of $\Gamma(t)$ does not allow
us to define the normal velocity in a good way although it is
possible to do this in the discrete models.

Now let us construct a Riemann conformal map $f:\,\mathbb D^-\to
\hat{\mathbb C}$, $\mathbb D^-=\{z\colon |z|>1\}$, which is meromorphic in $\mathbb D^-$,
$\displaystyle
f(\zeta,t)=\alpha(t)\zeta+a_0(t)+\frac{a_1(t)}{\zeta}+\dots$,
$\alpha(t)>0$, and maps $\mathbb D^-$ onto $\Omega(t)$. The boundary
$\Gamma(t)$ need not even be a quasidisk, as considered earlier.
While we are not able to construct a differential equation
analogous to the Polubarinova-Galin one on the unit circle, the
retracting L\"owner subordination chain still exists, and the
function $f(\zeta,t)$ satisfies the equation
\begin{equation}
\dot{f}(\zeta,t)=\zeta f'(\zeta,t)p_f(\zeta,t),\quad\zeta\in \mathbb D^-,\label{basic}
\end{equation}
where $p_f(\zeta,t)=p_0(t)+p_1(t)/\zeta+\dots$ is a Carath\'eodory
function: $\re p(z,t)>0$ for all $\zeta\in \mathbb D^-$ and for almost all
$t\in [0,\infty)$. A difference from the Hele-Shaw problem is that
the DLA problem is well-posed on each level of discreteness  by
construction. An analogue of DLA model was treated by means of
L\"owner chains by Carleson and Makarov in \cite{Carleson-Makarov}. In
this section we follow their ideas as well as those from
\cite{Hohlov2}.

Of course, the fractal growth phenomena can be seen without
randomness. A simplest example of such growth is the Koch
snowflake (Helge von Koch, 1870--1924).
\begin{figure}
\centering
{\scalebox{0.6}{\includegraphics{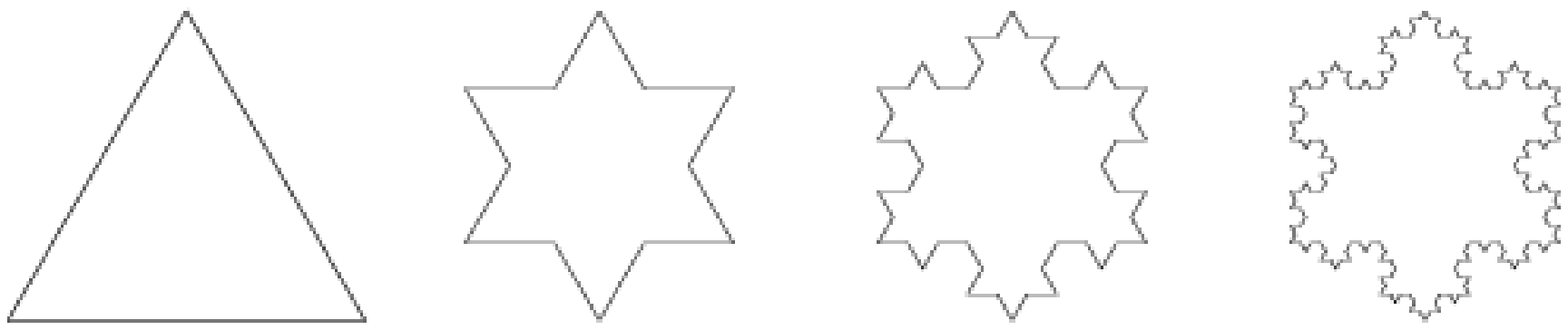}}}
\end{figure}
DLA-like fractal growth without randomness can be found, e.g.,
in \cite{Davidovitch2}.

Returning to the fractal growth we want to study a rather wide
class of models with complex growing structure. We note that
$\alpha(t)=\emk K(t)=\emk \Gamma(t)$.  Let $M(0,2\pi)$ be the
class of positive  measures $\gamma$ on $[0,2\pi]$. The control
function $p_f(\zeta,t)$ in (\ref{basic}) can be represented by the
Riesz--Herglotz formula
\[
p_f(\zeta, t)=\int\limits_{0}^{2\pi}\frac{e^{i\theta}+\zeta}{e^{i\theta}-\zeta}d\gamma_t(\theta),
\]
and $p_0(t)=\|\gamma_t\|$, where $\gamma_t(\theta)\in M(0,2\pi)$
for almost all $t\geq 0$ and absolutely continuous in $t\geq 0$.
Consequently, $\dot{\alpha}(t)=\alpha(t)\|\gamma_t\|$. There is a
one-to-one correspondence between  one-parameter ($t$) families of
measures $\gamma_t$ and L\"owner chains $\Omega(t)$ (in our case
of growing domains $\mathbb C\setminus \Omega(t)$ we have only
surjective correspondence).
\medskip

\noindent
{\it Example 1.} Suppose we have an initial domain $\Omega(0)$.
If the derivative of the measure $\gamma_t$ with respect to the Lebesgue measure is the Dirac
measure $d\gamma_t(\theta)\equiv\delta_{\theta_0}(\theta)d\theta$, then
\[
p_f(\zeta,t)\equiv\frac{e^{i\theta_0}+\zeta}{e^{i\theta_0}-\zeta},
\]
and $\Omega(t)$ is obtained by cutting $\Omega(0)$ along a geodesic arc. The preimage of the endpoint
of this slit is exactly $e^{i\theta_0}$. In particular, if $\Omega(0)$ is a complement of a disk, then
$\Omega(t)$ is $\Omega(0)$ minus a radial slit.
\medskip

\noindent
{\it Example 2.} Let $\Omega(0)$ be a domain bounded by an analytic curve $\Gamma(t)$.
If the derivative of the measure $\gamma_t$ with respect to the Lebesgue measure is
\[
\frac{d\gamma_t(\theta)}{d\theta}=\frac{1}{2\pi|f'(e^{i\theta},t)|^2},
\]
then
\[
p_f(\zeta,t)=\frac{1}{2\pi}
\int\limits_{0}^{2\pi}\frac{1}{|f'(e^{i\theta},t)|^2}\frac{e^{i\theta}+\zeta}{e^{i\theta}-\zeta}d\theta,
\]
and letting $\zeta$ tend to the unit circle we obtain
$\re[\dot{f}\,\,\overline{\zeta f'}]=1$, which corresponds to the
classical Hele-Shaw case, for which the solution exists locally in
time.
\medskip

In the classical Hele-Shaw process the boundary develops by fluid
particles moving in the normal direction. In the discrete DLA
models either lattice or with circular patterns the attaching are
developed in the normal direction too. However, in the continuous
limit it is usually impossible to speak of any normal direction
because of the irregularity of $\Gamma(t)$.

In \cite[Section 2.3]{Carleson-Makarov} this difficulty was circumvented
by evaluating the derivative of $f$ occurring in $\gamma_t$ in the
above L\"owner model slightly outside the boundary of the unit
disk.

  Let $\Omega(0)$ be any simply connected domain, $\infty \in
 \Omega(0)$, $0\not\in \Omega(0)$.
 The derivative of the measure $\gamma_t$ with respect to the
Lebesgue measure is
\[
\frac{d\gamma_t(\theta)}{d\theta}=\frac{1}{2\pi|f'((1+\varepsilon)e^{i\theta},t)|^2},
\]
with sufficiently small positive $\varepsilon$. In this case the
derivative is well defined.

It is worth to mention that the estimate
\[
\frac{\partial\emk \Gamma(t)}{\partial t}=\dot{\alpha}(t)\lesssim
\frac{1}{\varepsilon}
\]
would be equivalent to the Brennan conjecture\index{Brennan
conjecture} (see \cite[Chapter 8]{Pom3}) which is still unproved.
However, Theorem 2.1 \cite{Carleson-Makarov} states that if
\[
R(t)=\max\limits_{\theta\in
[0,2\pi)}|f((1+\varepsilon)e^{i\theta},t)|,
\]
then
\[
\limsup\limits_{\Delta t\to 0}\frac{R(t+\Delta t)-R(t)}{\Delta
t}\leq \frac{C}{\varepsilon},
\]
for some absolute constant $C$.  Carleson and Makarov
\cite{Carleson-Makarov} were, with the above model, able to establish an
estimate for the growth of the cluster or aggregate given as a
lower bound for the time needed to multiply the capacity of the
aggregate by a suitable constant. This is an analogue of the upper
bound  for the size of the cluster in two-dimensional stochastic
DLA given by \cite{Kesten}.

\subsection{Extension to several complex variables}

Pfaltzgraff in 1974 was the first one who extended the basic
L\"owner theory to $\C^n$ with the aim of giving bounds and
growth estimates to some classes of univalent mappings from the
unit ball of $\C^n$. The theory was later developed by 
Poreda, Graham, Kohr, Kohr, Hamada and others (see
\cite{G-K} and \cite{BCM2}).

Since then, a lot of work was devoted to successfully extend the theory to several complex variables, and finally, it has been accomplished. The main and dramatic difference between the one dimensional case and the higher dimensional case is essentially due to the lack of a Riemann mapping theorem or, which is the same, to the existence of the so-called Fatou-Bieberbach phenomena, that is, the existence of proper open subsets of $\C^n$, $n\geq 2$, which are biholomorphic to $\C^n$. This in turn implies that there are no satisfactory growth estimates for univalent functions on the ball (nor in any other simply connected domain) of $\C^n$, $n\geq 2$.

The right way to proceed is then to look at the L\"owner theory in higher dimension as a discrete complex dynamical system, in the sense of random iteration, and to consider abstract basins of attraction as the analogous of the L\"owner chains. In order to state the most general results, we first give some definitions and comment on that.

Most estimates in the unit disc can be rephrased in terms of the Poincar\'e distance, which gives a more intrinsic point of view. In higher dimension one can replace the Poincar\'e distance with the Kobayashi distance. First, we recall the definition of Kobayashi distance (see \cite{Kob} for details and properties). Let $M$ be a complex manifold and let $z,w\in M$. A {\sl chain of analytic discs} between $z$ and $w$ is a finite family of holomorphic mappings $f_j:\D \to M$, $j=1,\ldots, m$ and points $t_j\in (0,1)$ such that
\[
f_1(0)=z, f_1(t_1)=f_2(0),\ldots, f_{m-1}(t_{m-1})=f_m(0), f_m(t_m)=w.
\]
We denote by $\mathcal C_{z,w}$ the set of all chains of analytic discs joining $z$ to $w$. Let $L\in \mathcal C_{z,w}$. The {\sl length of $L$}, denoted by $\ell (L)$ is given by
\[
\ell (L):=\sum_{j=1}^m \omega (0, t_j)=\sum_{j=1}^m \frac{1}{2}\log \frac{1+t_j}{1-t_j}.
\]
We define the {\sl Kobayashi (pseudo)distance $k_M(z,w)$} as follows:
\[
k_M(z,w):=\inf_{L\in \mathcal C_{z,w}} \ell(L).
\]
If $M$ is connected, then $k_M(z,w)<+\infty$ for all $z,w\in M$. Moreover, by construction, it satisfies the triangular inequality. However, it might be that $k_M(z,w)=0$ even if $z\neq w$ (a simple example is represented by $M=\C$, where $k_\C\equiv 0$). In the unit disc $\D$, $k_\D$ coincides with the Poincar\'e distance.

\begin{definition}
A complex manifold $M$ is said to be {\sl (Kobayashi) hyperbolic} if $k_M(z,w)>0$ for all $z,w\in M$ such that $z\neq w$. Moreover, $M$ is said {\sl complete hyperbolic} if $k_M$ is complete.
\end{definition}

Important examples of complete hyperbolic manifolds are given by bounded convex domains in $\C^n$.

The main property of the Kobayashi distance is the following: let $M,N$ be two complex manifolds and let $f: M \to N$ be holomorphic. Then for all $z,w\in M$ it holds
\[
k_N(f(z), f(w))\leq k_M(z,w).
\]

It can be proved that if $M$ is complete hyperbolic, then $k_M$ is Lipschitz continuous (see \cite{AB}). If $M$ is a bounded strongly convex domain in $\C^n$ with smooth boundary, Lempert (see, {\sl e.g.} \cite{Kob}) proved that the Kobayashi distance is of class $C^\infty$ outside the diagonal. In any case, even if $k_M$ is not smooth, one can consider the differential $dk_M$ as the Dini-derivative of $k_M$, which coincides with the usual differential at almost every point in $M\times M$.

As it is clear from the one-dimensional general theory of L\"owner's equations, evolution families and Herglotz vector fields are pretty much related to semigroups and infinitesimal generators. Kobayashi distance can be used to characterize infinitesimal generators of continuous semigroups of holomorphic self-maps of complete hyperbolic manifolds. The following characterization of infinitesimal generators is proved for strongly convex domains in \cite{BCD}, and in general in \cite{AB}:

\begin{theorem}\label{autonomo}
Let $M$ be a complete hyperbolic complex manifold and let $H:M \to TM$ be an holomorphic vector field on $M$. Then the following are equivalent.
\begin{enumerate}
\item $H$ is an infinitesimal generator,
\item For all $z, w\in M$ with $z\neq w$
it holds $$(dk_M)_{(z,w)}\cdot (H(z),H(w))\leq 0.$$
\end{enumerate}
\end{theorem}

This apparently harmless characterization contains instead all the needed information to get good growth estimates. In particular, it is equivalent to the Berkson-Porta representation formula in the unit disc.

\subsection{$L^d$-Herglotz vector fields and Evolution families on complete hyperbolic manifolds}

Let $M$ be a complex manifold, and denote by  $\|\cdot \|$  a Hermitian metric on $TM$ and by $d_M$ the corresponding integrated distance.

\begin{definition}
\label{Her-vec-man} Let $M$ be a complex manifold. A \textit{weak holomorphic vector field of
order $d\geq 1$} on $M$ is a mapping $G:M\times \R^+\to
TM$ with the following properties:
\begin{itemize}
\item[(i)] The mapping $G(z,\cdot)$ is measurable on $\R^+$ for all
$z\in M$.
\item[(ii)] The mapping $G(\cdot,t)$ is a holomorphic vector field on $M$ for all $t\in \R^+$.
\item[(iii)] For any compact set $K\subset M$ and all $T>0$, there exists
a function $C_{K,T}\in L^d([0,T],\mathbb{R}^+)$ such that
$$\|G(z,t)\|\leq C_{K,T}(t),\quad z\in K, \mbox{ a.e.}\ t\in [0,T].$$
\end{itemize}

A \textit{Herglotz vector field of order $d\geq 1$} is a weak holomorphic vector field $G(z,t)$  of order $d$ with the property that
$M\ni z\mapsto G(z,t)$ is an infinitesimal generator
for almost all $t\in [0,+\infty)$.
\end{definition}

If $M$ is complete hyperbolic, due to the previous characterization of infinitesimal generators,  a weak holomorphic vector field $G(z,t)$  of order $d$  if a Herglotz vector field
of order $d$ if and only if
\begin{equation}\label{herglotz_def}
(dk_M)_{(z,w)}\cdot(G(z,t),G(w,t))\leq 0,\quad z,w\in M, z\neq w,\ \mbox{ a.e. }
t\geq 0.
\end{equation}
This was proved in \cite{BCD}  for strongly convex domains, and in \cite{AB} for the general case.

One can also generalize the concept of evolution families:

\begin{definition}\label{L^d_EF}
Let $M$ be a complex manifold.
A family $(\varphi_{s,t})_{0\leq s\leq t}$ of holomorphic
self-mappings of $M$ is  an {\sl
 evolution family of order $d\geq 1$} (or $L^d$-evolution family) if it satisfies the {\sl evolution property}
\begin{equation}\label{evolution_property}
\varphi_{s,s}={\sf id},\quad \varphi_{s,t}=\varphi_{u,t}\circ \varphi_{s,u},\quad 0\leq s\leq u\leq t,
\end{equation}
 and if for any $T>0$ and for any compact set
  $K\subset\subset M$ there exists a  function $c_{T,K}\in
  L^d([0,T],\R^+)$ such that
  \begin{equation}\label{ck-evd}
d_M(\varphi_{s,t}(z), \varphi_{s,u}(z))\leq \int_{u}^t
c_{T,K}(\xi)d \xi, \quad z\in K,\  0\leq s\leq u\leq t\leq T.
  \end{equation}
\end{definition}
It can be proved that  all elements of  an evolution family are  univalent (cf. \cite[Prop. 2.3]{AbstractLoew}).

The classical L\"owner and Kufarev-L\"owner equations can now be completely generalized as follows:

\begin{theorem}\label{prel-thm}
Let $M$ be a complete hyperbolic complex manifold. Then for any  Herglotz vector field $G$ of order
$d\in [1,+\infty]$ there exists a unique $L^d$-evolution family
$(\varphi_{s,t})$  over $M$ such that for all
$z\in M$
  \begin{equation}\label{solve}
\frac{\de \varphi_{s,t}}{\de t}(z)=G(\varphi_{s,t}(z),t) \quad
\hbox{a.e.\ } t\in [s,+\infty).
  \end{equation}
Conversely for any  $L^d$-evolution family $(\varphi_{s,t})$  over $M$ there exists a  Herglotz vector field $G$ of
order $d$ such that \eqref{solve} is satisfied. Moreover,
if $H$ is another weak holomorphic vector field which satisfies
\eqref{solve} then $G(z,t)=H(z,t)$ for all $z\in M$ and almost
every $t\in \R^+$.
\end{theorem}

Equation \eqref{solve} is the bridge between the $L^d$-Herglotz vector fields and $L^d$-evolution families. In \cite{BCM1} the result has been proved for any complete hyperbolic complex manifold $M$ with Kobayashi distance of class $C^1$ outside the diagonal, but the construction given there only allowed to start with evolution families of order $d=+\infty$. Next, in \cite{HKM} the case of $L^d$-evolution families has been proved for the case $M=\B^n$  the unit ball in $\C^n$. Finally, in \cite{AB},  Theorem \ref{prel-thm} was proved in full generality.

The previous equation, especially in the case of the unit ball of $\C^n$ and for the case $d=+\infty$, with evolution families fixing the origin and having some particular first jets at the origin has been studied by many authors, we cite here Pfaltzgraff \cite{Pf74}, \cite{Pf75}, Poreda \cite{Por91}, Graham, Hamada, Kohr \cite{GHK01}, Graham, Hamada, Kohr, Kohr \cite{G-H-K-K} (see also \cite{GKP03}).

Using the so-called `product formula', proved in convex domains by Reich and Shoikhet \cite{RS} (see also \cite{Reich-Shoikhet}), and later generalized on complete hyperbolic manifold in \cite{AB} we get a strong relation between the semigroups generated at a fixed time by a Herglotz vector field and the associated evolution family. Let $G(z,t)$ be a Herglotz vector field on a complete hyperbolic complex manifold $M$. For almost all $t\geq 0$, the holomorphic vector field $M\ni z\mapsto G(z,t)$ is an infinitesimal generator. Let $(\phi_r^t)$ be the associated semigroups of holomorphic self-maps of $M$. Let $(\varphi_{s,t})$ be the evolution family associated to $G(z,t)$. Then, uniformly on compacta of $M$ it holds
\[
\phi_t^r=\lim_{m\to \infty} \varphi_{t,t+\frac{r}{m}}^{\circ m}=\lim_{m\to \infty} \underbrace{(\varphi_{t,t+\frac{r}{m}}\circ \ldots\circ \varphi_{t,t+\frac{r}{m}})}_m.
\]

\subsection{L\"owner chains on complete hyperbolic manifolds}

Although one could easily guess how to extend the notion  of Herglotz vector fields and evolution families  to several complex variables, the concept of L\"owner chains is not so easy to extend. For instance, starting from a Herglotz vector field on the unit ball of $\C^n$, one would be tempted to define in a natural way L\"owner chains with range in $\C^n$. However, sticking with such a definition, it is rather hard to get a complete solution to the L\"owner PDE. In fact, in case $D=\B^n$ the unit ball, much effort has been done to show that, given an evolution family $(\varphi_{s,t})$ on $\B^n$ such that $\varphi_{s,t}(0)=0$ and $d(\varphi_{s,t})_0$ has a special form, then there exists an associated L\"owner chain. We cite here the contributions of  Pfaltzgraff \cite{Pf74}, \cite{Pf75}, Poreda \cite{Por91}, Graham, Hamada, Kohr \cite{GHK01}, Graham, Hamada, Kohr, Kohr \cite{G-H-K-K}, Arosio \cite{A10}, Voda \cite{Vo}. In the last two mentioned papers, resonances phenomena among the eigenvalues of $d(\varphi_{s,t})_0$ are taken into account.

The reason for these difficulties are due to the fact that, although apparently natural, the definition of L\"owner chains as a more or less regular family of univalent mappings from the unit ball to $\C^n$ is not the right one. And the reason why this is meaningful in one dimension is just because of the Riemann mapping theorem, as we will explain.

Indeed, as shown before, there is essentially no difference in considering evolution families or Herglotz vector fields in the unit ball of $\C^n$ or on complete hyperbolic manifolds, since the right estimates to produce the L\"owner equation are provided just by the completeness of the Kobayashi distance and its contractiveness properties.

The right point of view is to consider evolution families as random iteration families,  and thus, the `L\"owner chains' are just the charts of the abstract basins of attraction of such a dynamical system. To be more precise, let us recall the theory developed in \cite{AbstractLoew}. Interesting and surprisingly enough, regularity conditions--which were basic in the classical theory for assuming the classical limiting process to converge---do not play any role.

\begin{definition}\label{algebraic_EF}
Let $M$ be a complex manifold.
An {\sl algebraic evolution family} is a family $(\varphi_{s,t})_{0\leq s\leq t}$ of univalent self-mappings of $M$  satisfying the  evolution property \eqref{evolution_property}.
\end{definition}

A $L^d$-evolution family is an algebraic evolution family because all elements of a $L^d$-evolution family are injective as we said before.

\begin{definition}
Let $M, N$ be two complex manifolds of the same dimension.  A family  $(f_t)_{t\geq 0}$ of holomorphic mappings $f_t:M\to N$ is a {\sl subordination chain}  if for each $0\leq s\leq t$ there
exists a holomorphic mapping $v_{s,t}:M\to M$ such that
$f_s=f_t\circ v_{s,t}$.
A subordination chain $(f_t)$ and an algebraic evolution family $(\varphi_{s,t})$ are {\sl associated} if
$$ f_s=f_t\circ \varphi_{s,t},\quad 0\leq s\leq t.$$

An {\sl algebraic L\"owner chain} is a  subordination chain such that each mapping $f_t: M\to N$ is univalent.   The {\sl range}  of an algebraic  L\"owner chain is defined as
\[
\rg(f_t):=\bigcup_{t\geq 0}f_t(M).
\]
\end{definition}

Note that an algebraic L\"owner chain $(f_t)$ has the property that
$$f_s(M)\subset f_t(M),\quad 0\leq s\leq t.$$

We have the following result which relates algebraic evolution families with algebraic L\"owner chains, whose proof is essentially based on abstract categorial analysis:

\begin{theorem}\cite{AbstractLoew}\label{ABHKthm}
Let $M$ be a complex manifold. Then any  algebraic evolution family $(\varphi_{s,t})$ on $M$ admits  an associated algebraic L\"owner chain $(f_t\colon M\to N)$.
Moreover if $(g_t\colon M\to Q)$ is a subordination chain associated with $(\varphi_{s,t})$ then there exist a holomorphic mapping $\Lambda\colon \rg(f_t)\to Q$ such that $$g_t=\Lambda\circ f_t,\quad \forall t\geq 0.$$ The mapping $\Lambda$ is univalent if and only if $(g_t)$ is an algebraic L\"owner chain, and in that case $\rg(g_t)=\Lambda(\rg(f_t)).$
\end{theorem}

The previous theorem shows that the range $\rg (f_t)$ of an algebraic L\"owner chain $(f_t)$ is uniquely defined up to biholomorphisms. In particular, given an algebraic evolution family $(\varphi_{s,t})$ one can define its {\sl L\"owner range} ${\sf Lr}(\varphi_{s,t})$  as the biholomorphism class  of the range of any associated algebraic L\"owner chain.

In particular, if $M=\D$ the unit disc, then the L\"owner range of any evolution family on $\D$ is a simply connected non compact Riemann surface, thus, by the uniformization theorem, the L\"owner range is either the unit disc $\D$ or $\C$. Therefore, in the one-dimensional case, one can harmlessly stay with the classical definition of L\"owner chains as a family of univalent mappings with image in $\C$.

One can also impose $L^d$-regularity as follows:

\begin{definition}
Let $d\in [1,+\infty]$. Let $M, N$ be two complex manifolds of
the same dimension. Let $d_N$ be the distance induced by a Hermitian metric on $N$. An algebraic L\"owner chain $(f_t\colon M\to N)$  is a {\sl $L^d$-L\"owner chain} (for $d\in [1,+\infty]$) if for any compact set $K\subset\subset M$ and any
  $T>0$ there exists a $k_{K,T}\in L^d([0,T], \R^+)$ such that
 \begin{equation}\label{LCdef}
d_N(f_s(z), f_t(z))\leq \int_{s}^t k_{K,T}(\xi)d\xi
 \end{equation}
for all $z\in K$ and for all $0\leq s\leq t\leq T$.
\end{definition}

The $L^d$-regularity passes from evolution family to L\"owner chains:

\begin{theorem}\cite{AbstractLoew}
Let  $M$ be a  complete hyperbolic manifold with a given Hermitian metric and $d\in [1,+\infty]$. Let $(\varphi_{s,t})$ be an algebraic evolution family on $M$ and let $(f_t\colon M\to N)$ be an associated algebraic L\"owner
chain.
Then $(\varphi_{s,t})$ is a $L^d$-evolution family on $M$ if and only if $(f_t)$ is a $L^d$-L\"owner chain.
\end{theorem}

Once the general L\"owner equation is established and L\"owner chains have been well defined, even the L\"owner-Kufarev PDE can be generalized:

\begin{theorem}\cite{AbstractLoew}
Let $M$ be a complete hyperbolic complex manifold, and let $N$ be a complex manifold of the same dimension.
Let $G:M\times \R^+\to TM$ be a Herglotz vector
field of order $d\in[1,+\infty]$
associated with the $L^d$-evolution family $(\varphi_{s,t})$. Then a family of univalent mappings $(f_t\colon M\to N)$ is an $L^d$-L\"owner chain associated with $(\varphi_{s,t})$ if and only if   it is locally absolutely continuous
on $\R^+$ locally uniformly with respect to $z\in M$ and solves the  L\"owner-Kufarev PDE
\begin{equation*}
\frac{\partial f_s}{\partial s}(z)=-(df_s)_zG(z,s),\quad\mbox{a.e. }s\geq 0,z\in M.
\end{equation*}
\end{theorem}

\subsection{The L\"owner range and the general L\"owner PDE in $\C^n$} As we saw before, given a $L^d$-evolution family (or just an algebraic evolution family) on a complex manifold, it is well defined the L\"owner range ${\sf Lr}(\varphi_{s,t})$ as the class of biholomorphism of the range of any associated L\"owner chain.

In practice, it is interesting to understand the L\"owner range of an evolution family on a given manifold. For instance, one may ask whether, starting from an evolution family on the ball, the L\"owner range is always biholomorphic to an open subset of $\C^n$. This problem turns out to be related to the so-called {\sl Bedford's conjecture}. Such a conjecture states that given a complex manifold $M$, an automorphism $f:M\to M$ and a $f$-invariant compact subset $K\subset M$ on which the action of $f$ is hyperbolic, then the stable manifold of $K$ is biholomorphic to $\C^m$ for some $m\leq \dim M$. The equivalent formulation which resembles the problem of finding the L\"owner range of an evolution family in the unit ball is in \cite{FornaessStensones}, see also \cite{A} where such a relation is well explained.

In  \cite[Section 9.4]{A10} it is shown that there exists an algebraic evolution family $(\varphi_{s,t})$ on $\B^3$ which does not admit any associated algebraic L\"owner chain with range in $\C^3$. Such an evolution family is however not $L^d$ for any $d\in [1,+\infty]$.

In the recent paper \cite{ABW} it has been proved the following result:

\begin{theorem}
Let $D\subset \C^n$ be a complete hyperbolic starlike domain (for instance the unit ball). Let $(\varphi_{s,t})$ be an $L^d$-evolution family, $d\in [1,+\infty]$. Then the L\"owner range ${\sf Lr}(\varphi_{s,t})$ is biholomorphic to a Runge and Stein open domain in $\C^n$.
\end{theorem}

The proof, which starts from the existence of a L\"owner chain with abstract range, is based on the study of manifolds which are union of balls, using a result by Docquier and Grauert to show that the regularity hypothesis guarantees Runge-ness and then one can use approximation results of Anders\'{e}n   and  Lempert in order to construct a suitable embedding.

As a corollary of the previous consideration, we have a general solution to L\"owner PDE in higher dimension, which is the full analogue of the one-dimensional situation:

\begin{theorem}\cite{ABW}
Let $D\subset \C^N$ be a complete hyperbolic starlike domain.
Let $G:D\times \R^+\to \C^N$ be a Herglotz vector
field of order $d\in[1,+\infty]$. Then there exists a family  of univalent mappings $(f_t\colon D\to \C^N)$ of order $d$ which solves the L\"owner PDE
\begin{equation}\label{L-PDE}
\frac{\partial f_t}{\partial t}(z)=-df_t(z)G(z,t),\quad\mbox{a.a. }t\geq 0,\forall z\in D.
\end{equation}
Moreover, $R:=\cup_{t\geq 0}f_t(D)$ is a Runge and Stein domain in $\C^N$ and any other solution to \eqref{L-PDE} is of the form $(\Phi \circ f_t)$ for a suitable holomorphic map $\Phi:R\to \C^N$.
\end{theorem}

In general, one can infer some property of the L\"owner range from the dynamics of the evolution family.
In order to state the result, let us recall what the Kobayashi pseudometric is:

\begin{definition}
Let $M$ be a complex manifold. The {\sl Kobayashi pseudometric} $\kappa_M: TM \to \R^+$ is defined by
\[
\kappa_M(z; v):=\inf\{ r>0: \exists g: \D \to M\  \hbox{holomorphic}\ : g(0)=z, g'(0)=\frac{1}{r}v\}.
\]
\end{definition}

The Kobayashi pseudometric has the remarkable property of being contracted by holomorphic maps, and its integrated distance is exactly the Kobayashi pseudodistance. We refer the reader to  \cite{Kob} for details.

\begin{definition}
Let $(\varphi_{s,t})$ be an algebraic evolution family on a complex manifold $M$. For $v\in T_zM$ and $s\geq 0$  we define
\begin{equation}\label{beta}
\beta^s_z(v):=\lim_{t\to \infty} \kappa_M(\varphi_{s,t}(z); (d\varphi_{s,t})_z(v)).
\end{equation}
\end{definition}

Since the Kobayashi pseudometric is
contracted by holomorphic mappings the limit in \eqref{beta} is well defined.

The function $\beta$ is the bridge between the dynamics of an algebraic evolution family  $(\varphi_{s,t})$ and the geometry of its L\"owner range. Indeed, in \cite{AbstractLoew} it is proved that, if $N$ is a representative of the L\"owner range of $(\varphi_{s,t})$ and $(f_t:M\to N)$ is an associated algebraic L\"owner chain, then
for all $z\in M$ and $v\in T_zM$ it follows
\[
f_s^*\kappa_N(z;v)=\beta^s_z(v).
\]

In the unit disc case, if $(\varphi_{s,t})$ is an algebraic evolution family, the previous formula allows to determine the L\"owner range: by the Riemann mapping theorem the L\"owner range is either $\C$ or $\D$. The first being non-hyperbolic, if $\beta^s_z(v)=0$ for some $s>0, z\in \D$ ($v$ can be taken to be $1$), then the L\"owner range is $\C$, otherwise it is $\D$.

Such a result can be generalized to a complex manifold $M$.
Let ${\sf aut}(M)$ denote the group of
holomorphic automorphisms of a complex manifold $M$. Using a result by Forn\ae ss and Sibony \cite{F-S}, in \cite{AbstractLoew} it is shown that the previous formula implies

\begin{theorem}
Let $M$ be a complete hyperbolic complex manifold and assume that
$M/\sf{aut}(M)$ is compact. Let $(\varphi_{s,t})$ be an algebraic evolution
family on $M$.  Then
\begin{enumerate}
\item If there exists $z\in M$, $s\geq0$ such that $\beta^s_z(v)\neq 0$ for
all $v\in T_zM$ with $v\neq 0$ then  ${\sf Lr}(\varphi_{s,t})$ is biholomorphic to $M$.
\item If there exists $z\in M$, $s\geq0$ such that $\dim_\C\{v\in T_zM: \beta_z^s(v)= 0\}=1$  then  ${\sf Lr}(\varphi_{s,t})$ is a fiber bundle with fiber $\mathbb{C}$
over a closed complex submanifold of $M$.
\end{enumerate}
\end{theorem}

\end{document}